\documentclass[a4paper,twoside,english]{amsart}
\usepackage[T1]{fontenc}
\usepackage[latin9]{inputenc}
\usepackage{babel}
\usepackage{verbatim}
\usepackage{amsthm}
\usepackage{amsmath}
\usepackage{amssymb}
\usepackage[unicode=true,
 bookmarks=true,bookmarksnumbered=true,bookmarksopen=false,
 breaklinks=false,pdfborder={0 0 1},backref=false,colorlinks=false]
 {hyperref}
 
 \usepackage{tikz-cd}  
 
\hypersetup{
 pdfauthor={P. Corvaja  and U. Zannier}}

\makeatletter

\pdfpageheight\paperheight 
\pdfpagewidth\paperwidth

\theoremstyle{plain}
\newtheorem{thm}{\protect\theoremname}[section]
 \newcommand\thmsname{\protect\theoremname}
 \newcommand\nm@thmtype{theorem}
 \theoremstyle{plain}

  \theoremstyle{remark}
  \newtheorem{rem}[thm]{\protect\remarkname}
  \theoremstyle{definition}
  \newtheorem*{example*}{\protect\examplename}
  \theoremstyle{definition}
  \newtheorem{example}[thm]{\protect\examplename}
  \theoremstyle{plain}
  \newtheorem{lem}[thm]{\protect\lemmaname}
  \theoremstyle{plain}
  \newtheorem{prop}[thm]{\protect\propositionname}
  \theoremstyle{plain}

\makeatletter
  \theoremstyle{definition}
  \newtheorem{my@rem}[thm]{Remark}
  \renewenvironment{rem}{\begin{my@rem}}{\end{my@rem}}
\makeatother

\makeatother

  \providecommand{\examplename}{Example}
  \providecommand{\lemmaname}{Lemma}
  \providecommand{\propositionname}{Proposition}
  \providecommand{\remarkname}{Remark}
  \providecommand{\theoremname}{Theorem}
\providecommand{\theoremname}{Theorem}
 \providecommand{\corollaryname}{Corollary}

\usepackage[left=3.1cm,right=3.5cm,top=3cm,bottom=3cm]{geometry}

\usepackage{mathtools}

\def\N{{\Bbb N}}
\def\Q{{\Bbb Q}}
\def\R{{\Bbb R}}

\def\L{{\mathcal L}}
\def\E{{\mathcal E}}
\def\Z{{\Bbb Z}}
\def\G{{\Bbb G}}

\def\T{{\mathcal T}}

\def\Aa{{\mathcal A}}

\def\P{{\Bbb P}}
\def\C{{\Bbb C}}

\def\u{{\bf u}}
\def\v{{\bf v}}
\def\d{{\rm d}}

\def\CVD{{\hfill\hfil{\lower 2pt\hbox{\vrule\vbox to 7pt
{\hrule width  5pt\varphifill\hrule}\varphirule}}}\par}

\begin{document} 

\title[Finiteness Theorems  on Elliptical Billiards]{Finiteness Theorems  on Elliptical Billiards and a Variant of the Dynamical Mordell-Lang conjecture}

\author[Pietro Corvaja and Umberto  Zannier, with an appendix with J. Demeio]{ Pietro Corvaja and Umberto  Zannier \\ with an appendix by P. Corvaja, U. Zannier and J. Demeio}
\maketitle




\date{\today}

{\tt Abstract.}  {\it We offer  some theorems, mainly of  finiteness, for certain patterns in elliptical billiards, related to periodic trajectories. For instance, if two players hit a ball  at a given  position and with directions forming a fixed angle in $(0,\pi)$, there are only finitely many cases for both trajectories being periodic. Another instance is the finiteness of the billiard shots which send a given ball into another one so that this falls eventually in a hole. These results have their origin in 
 `relative' cases of the Manin-Mumford conjecture, and constitute instances of how arithmetical content  
 may affect chaotic behaviour (in billiards). We shall also 
  interpret  the statements    through  a   variant  of the dynamical Mordell-Lang conjecture. In turn, this variant embraces    cases, 
which, somewhat surprisingly,  sometimes  can be treated (only) by completely different methods compared to the former ones; here we shall offer an explicit example related to diophantine equations in algebraic tori.}


\section{Introduction} 

The study of billiards is of course a classical mathematical topic, still of wide current  interest.   Borrowing from Ya. Sinai's ICM  survey paper \cite{Sin}, {\it ``Billiards are dynamical systems which correspond to the uniform motion of a point inside a domain on a Riemmanian manifold with elastic reflections on the boundary. [...] The theory of billiards suggests many beautiful problems.}  
And, to quote  from the preface of S. Tabachnikov's book \cite{T} (who in turn refers partly to A. Katok):   {\it  ``Billiards is not a
single mathematical theory; [...]  it is rather a mathematician's playground where various methods and approaches are
tested and honed.''} 

This viewpoint particularly  fits with (the aim of) the present paper: we authors never worked specifically on billiards, but rather  we realized 
 that certain (fairly recent)  theorems of arithmetico-geometrical nature may be   applied so to yield, after suitable  work, natural conclusions in the realm of {\it elliptical billiards}. More precisely, these results have implications regarding periodical orbits and finiteness of  billiard trajectories with certain (simple and seemingly natural) patterns or properties which we shall  introduce. Naturally the study of elliptical billiards goes back to long ago \footnote{P. Sarnak pointed out to us the reference \cite{B} by  the Jesuit priest R. Boscovich, an astronomer and scientist, going back to the XVIII Centrury.} and is present also in recent literature (see e.g. the book \cite{T}, the paper \cite{ConnesZagier} by A. Connes and D. Zagier, and the paper \cite{RGK} by D. Reznik, R. Garcia, J. Koiller); however apparently no situation similar to what appears in this paper has been already analysed. 

The  results that we intend to apply for the main purposes of this paper have their origin in the celebrated {\it Manin-Mumford conjecture} (a theorem of M. Raynaud since the 80s), which predicted finiteness of torsion points in a curve of genus $\ge 2$ embedded in its Jacobian. This was later extended in the realm of abelian schemes, with statements that are part of the  so-called {\it Pink-Zilber conjecture(s)}. We shall say more on this below, a fairly  extended account (of the basics)   being  given in the second author's book \cite{Z3}.  

\medskip

The link with billiards arises because  a billiard shot in an elliptical billiard  corresponds to a point on an  elliptic curve which is a member of an elliptic scheme denoted $\L$; periodicity  corresponds to this point being  torsion.  This correspondence too goes back to long ago, essentially to Jacobi's proof of a famous theorem of Poncelet; see also the book \cite{DR} by V. Dragovic and M. Radnovic for a detailed study of this, also in a more general context of hyperelliptic Jacobians, and see the papers \cite{BM} by W. Barth and J. Michel, and \cite{Jak} by B. Jakob. We especially point out the book by J. Duistermaat \cite{Du}, which (see Ch. 11) treats in detail some issues of elliptic surfaces with reference  to billiards. But to our knowledge no systematic application of  this correspondence, similar to what we propose,     has been  yet developed or   pointed out so far  in the literature. Also, we shall develop several details and equations for the basic setting which seem not to appear in full  in the existing literature, and  this study will involve the investigation of relevant (elliptic) algebraic surfaces and corresponding results in their realm which appear to be new.   (The  interesting book \cite{Du} contains some results and formulae which we develop independently in the Appendix, however  the   viewpoint  of \cite{Du} is somewhat different  and  no arithmetical information is introduced.) 

\medskip

{\tt Some goals of the paper}. In the first place we shall use the said description mainly for  the mentioned  {\it finiteness}  theorems, related to periodic orbits in elliptical billiards, in which certain simultaneous conditions are required: see Theorems \ref{T.angolo}, \ref{T.buca}, \ref{T.ritorno}.  

But we shall also obtain other information about periodic orbits, as for instance asymptotic  results on the distribution of  billiard trajectories starting at  a given point and having given period $n$; the relevant constants will turn out to be expressed explicitly in terms of certain elliptic integrals on the Legendre curve.  (See Thm. \ref{T.esiste} below. In the complex case (treated in the Appendix because it is more distant from the main theme) we shall point out  an arithmetical meaning of the relevant constant in the asymptotic.) 

Moreover, in \S \ref{S.formula} we shall use the elliptic picture to  give a very simple explanation of a striking formula appearing in \cite{RGK} (see formula (8)) which expresses the sum of cosines of angles between consecutive segments in a periodic trajectory; this is an instance of a Birkoff sum in the theory of dynamical systems. We shall also `characterise' all such formulae, and prove an extension of the formula  for trajectories which are not necessarily periodic.

\medskip

{\tt Further links and results}. 
Several instances of connections between billiards and arithmetical or geometrical questions are of course already known, but we wonder whether the present study may suggest new ones, as for instance  finiteness   issues similar to those mentioned above,  possibly on other types of  billiards (as in Remark \ref{R.altri} below), so as  to be still meaningful and maybe   raise  sensible questions of diophantine type. 

\bigskip

The terminology  {\it elliptical} for  the billiards  we shall deal with,  refers to the shape of the billiard table. (See  e.g. \cite{T}, Ch. 4.) \footnote{This  has  not to be confounded with {\it hyperbolic billiards}, for which we refer to  
\cite{Sin}.}  It seems worth mentioning that several   remarkable investigations in the topic of billiards regard especially  {\it polygonal billiards}, which indeed occur among the most natural first examples (see e.g. \cite{T}, Ch. 7, for generalities). These are related   in particular  to the geometry of suitable curves of genus at least $2$ and their Jacobians  (see e.g.  W. Veech's paper \cite{V} and C. McMullen's papers \cite{MC}, \cite{MC2}, and references therein).  While we do not see any definite direct connection of our billiards with the polygonal ones, it is puzzling that both our context and the latter have links with  torsion divisors in Jacobians and the {\it Pell-Abel equation}\footnote{We follow J-P. Serre's S\'eminaire Bourbaki paper \cite{Sb} for this terminology.}, i.e. a polynomial version of the  famous Pell equation in Number Theory.

\begin{small}
See \cite{MC2}  for the relation of this equation with polygonal billiards, and see the second author's paper \cite{Z2} for a survey of this topic and    links with some results at the basis of the present paper. A further connection comes from the fact that  polygonal billiards are related to real multiplication in certain Jacobians (see \cite{MC}), which occurs also concerning  the {\it Theorem of Poncelet} (recalled below) which is at the basis of the billiard-map law in elliptical billiards: see J. Wilson's paper \cite{W} for proofs of  Poncelet's theorem coming from real multiplication in Jacobians. We have presently no general direct explanations for these links but we hope this will be explored in some detail in the future. 
\end{small}

\medskip

While giving the proofs, we shall see that  some of our results on elliptical billiards 
may be naturally  framed within the following  general issue:

\medskip

\noindent{\bf Question}.  {\it  Given an algebraic  surface  with a finitely generated commutative semigroup $\Gamma$ of (rational) endomorphisms,  and given three curves on the surface, what can be said assuming the existence of infinitely many  $\Gamma$-orbits intersecting each of the curves ?}
 
 \medskip
 
 This may be formulated  for varieties of arbitrary dimension and reminds of the  {\it dynamical Mordell-Lang conjecture}, studied by various authors and still open in its most general case (see for instance the book \cite{BGT} by J. Bell, D. Ghioca, and T. Tucker). Indeed  such  conjecture overlaps with our issue: 
in our setting, related to Theorem \ref{T.buca}, the surface will be an elliptic surface (a surface fibred in elliptic curves) and $\Gamma$ will be generated by a single automorphism, corresponding to the translation by a section of the elliptic fibration. 

This can also be phrased as a case of the Question where the surface is $\P_2$ and the endomorphism is a certain Cremona transformation. Instead, taking the surface again to be $\P_2$, and considering a linear automorphism, and taking for the curves the simplest one, i.e. lines, we obtain 
another instance of this Question, not directly related to billiards. In the last part of the paper we shall treat this by proving another finiteness theorem,    (see Theorem \ref{T.P_2}), answering another case of the  Question.

We note that  this time the arguments shall rely on completely different ingredients compared to the billiard case.  This double nature of methods for instances of  the same context seems rather peculiar to us.  In the present case the difference seems to come from the nature of the endopmorphism. 

We also wonder about a possible  general statement containing the dynamical Mordell-Lang conjecture and this Question as well \footnote{Somewhat similarly to how the  Pink-Zilber conjectures contain the Mordell-Lang's, see in this direction the paper \cite{GN} by Ghioca and D. Nguyen). For instance, it is possible to see Falting's finiteness theorem for curves of genus $2$ as an instance of our Question.},  linking the whole context to billiards.

\medskip

\subsection{Main statements} 
Let us now go to  state  the main results of this paper. Our conclusions will be just samples of what can be proved with the same methods, and there are many possible variations or generalisations, which we leave to the interested readers. 

\medskip

{\tt Our setting}.  Let $C$ be an ellipse in the real plane, which together with its interior, denoted  $\T^o$, will constitute  our billiard table, denoted $\T=C\cup \T^o$.  The ``balls'' will be points in  $\T$  and a billiard shot to a ball $p$ will correspond to the motion of $p$ along a   half-line until it meets the boundary $C$, in which case it will be reflected according to the usual principle\footnote{Namely, the old and new directions will form equal angles with the tangent to $C$ at the boundary point.}, and so on.  We suppose that there is no friction  so that the ball will continue to move indefinitely, and the speed will be immaterial for us. 

We recall that the {\it phase space}  (see \cite{T}, Ch. 3) is the set of pairs  $(p,v)$ where $p\in C$ and $v\in S_1$ is a unit vector such such that $p+\epsilon v$ is   inside $\T$   for small $\epsilon>0$.  So $v$ varies in a semi-circle depending on $p\in C$. (During the proofs we shall see how  in our context the phase space gives rise to elliptic algebraic surfaces of various types.)   In general by a {\it billiard shot} we shall mean a pair $(p,v)$ either in the phase space  or such that   $p\in\T^o$, where  $v$  is a  nonzero vector, often assumed of unit length,  representing  the direction along  which $p$ is sent. (So in case $p\in\T^o$ the vector $v$ varies in $S_1$.)  This shot generates a  {\it billard trajectory}, consisting of the shot together with the further segments which arise by reflection.

\medskip

\begin{small}
\begin{rem} \label{R.real} {\tt Real and complex solutions}. We remark at once that the said   motion of the balls in the billiard  will be described by {\it real} solutions to certain (algebraic) equations. Now, these  equations  will make sense even  considering their {\it complex} solutions.  Of course in general these new solutions will not correspond to an actual  billiard trajectory. Nevertheless many result will maintain their validity even in the extended realm of solutions. We shall usually omit this comment in the sequel unless the situation is special.  For instance in the Appendix we shall treat explicitly both cases, and we shall see that certain constants arising from the counting of periodic trajectories have quite a different meaning in the two cases (see also Thm. \ref{T.esiste} below and the comments which follow). 
\end{rem}
\end{small}

\medskip

Most of our conclusions will assert {\it finiteness} of certain patterns. But let us start now
with an easier  result in the opposite direction. Since we shall supplement this result with other ones which are a bit more distant from the main topic, this will be proved in the Appendix written also with J. Demeio.

\begin{thm} \label{T.esiste}  Let $p_1,p_2\in\T$, not both foci. Then for each integer $n>0$ there exists a billiard trajectory from $p_1$ to $p_2$ with exactly $n$ bounces.  For odd (resp. even) integer $n>0$ the number of periodic trajectories from $p_1$,  of period $n$, equals $c_{o}\cdot n+O(1)$ (resp. $c_{e}\cdot n+O(1)$) where $c_{o},c_{e}$ are positive numbers (depending on $p_1$)  which may be `explicitly' expressed in terms of elliptic integrals on the Legendre curve (as indicated in the Appendix).
\end{thm}

At least for algebraic values of the parameters which express the ellipse and the point, the numbers $c_{o}, c_{e}$ will turn out to be ratios of  {\it periods}, in the sense of the paper \cite{KZ}  by M. Kontsevich and  D. Zagier. 

We shall also observe that, as $n$ grows the set of slopes of the first segment of such  trajectories is dense in the appropriate circle or semicircle, and tends to be  uniformly distributed with respect to a certain measure, made explicit in \S \ref{S.pre}. (This does not hold in the $p$-adic case, as shown in \cite{LZ}.

Some of these  results are maybe known to experts, but we have no reference.

The asymptotic becomes $cn^2+O(1)$ if one looks at complex points, where now  the constant is related to the (functional) height of the relevant section (see the paper \cite{CDMZ} of J. Demeio, Masser and the authors).  
Our proofs, which work for  sections more general than the billiard ones, will rely on   the so-called {\it Betti map (of a section)} (in its simplest version), introduced implicitly by Ju. Manin and studied recently by various authors, e.g.  in the papers \cite{CMZ} and \cite{ACZ}  of the authors resp. with Masser and  Y. Andr\'e   (with an appendix of Z. Gao).  
It turns out that the elliptical billiard delivers a {\it section} on an associated elliptic scheme, which may be taken as the Legendre scheme after a suitable quadratic base change. We shall analyse this scheme from several viewpoints. In the Appendix we shall show that analogous distributional results hold for general sections, i.e. not associated to billiards.

\begin{small}  \begin{rem}\label{R.realicommenti}  

(i) {\tt Algebraic points, heights}.  When everything is defined over the field $\overline\Q$ of algebraic numbers, one derives further information, as for instance  that the relevant slopes, though dense, are {\it sparse}. Indeed, it follows from results of J. Silverman and J. Tate (1980s) that such slopes are algebraic numbers of bounded (Weil) height (in particular this implies that there are only finitely many ones of bounded degree over $\Q$). See \cite{Z3}, especially Appendix C by Masser. 
 In other words, for given points $p_1,p_2$ on the billiard table
 
 \medskip
 
 $\bullet$ {\it a trajectory starting at $p_1$ and with slope of large enough height will never pass through $p_2$}.
 
 \medskip
  
   This is an instance of how arithmetical content can affect chaotic behaviour (in this case of the trajectories).  
 Since the height is contributed by all absolute values, not only the archimedean ones, this fact says that  in the present circumstances chaotic behaviour is not a mere consequence of the {\it real} usual absolute value of the relevant parameters.  A similar comment holds on thinking that, as already mentioned,  density fails in the $p$-adic context.

(ii) {\tt Relation with integral points over function fields}.  Consider the case of Theorem \ref{T.esiste} when $p_1=p_2$, so we want periodic trajectories. If we write this condition in terms of sections of the elliptic scheme $\L$ corresponding to the billiard game, the condition amounts to a section not being torsion, but becoming torsion at the relevant slopes. These slopes correspond to poles of the Legendre coordinates expressing the section. In this way the infinity of these slopes may be deduced from a function field version of  Siegel's theorem on integral points, applied here with the elliptic curve $\L$ viewed over a finite extension of $\Q(s)$. See \cite{Z3} and \cite{CDMZ} for more on this viewpoint, which is valid for all sections (even if not coming from the billiard). 
\end{rem}

\end{small}

\medskip

Our next result concerns  two players which start from the same point and with directions which stay apart by a fixed angle, asking for a periodic path in both cases. We can prove finiteness for this pattern:

\begin{thm}\label{T.angolo} Let $p_0\in\T$ and $\alpha\in (0,\pi)$. Suppose also that $C$ is not a circle. There are only finitely many  pairs of billiard  trajectories $(p_0,v)$, $(p_0,v')$ which are both periodic and such that $v,v'$ form an angle $\alpha$. If $C$ is a circle then the same conclusion holds if $p_0\in\T^o$ is an interior point, not its center.
\end{thm}

  In a previous version of the paper we only considered the case of a non-circular elliptic billiard. We owe to A. Sorrentino the remark that for a circular billiard the situation is slightly different; we then considered this issue in the present version. 

Recently G. Binyamini \cite{Bin} has produced effectivity of the relevant ingredients of the proofs, so if we work with computable quantities one can exhibit in principle the finitely many directions in question, and similarly for the next results.
Here also the question arose about quantitative estimates for the number of solutions in the cases where finiteness holds. This issue was put forward by A. Glutsyuk, whom we thank. We do not know the exact dependence of the bounds on these data.

  If we let also the point $p_0$ vary, then, using e.g. the Betti map for a higher dimensional base (see in particular \cite{ACZ}), one can show that the relevant pairs form an infinite set, actually containing a denumerable union of sets of positive dimension. 
 We shall say a little more on these problems in the final section.

\medskip

Our second finiteness conclusion  is inspired by an elliptical (possibly circular) billiard with two balls $p_1,p_2\in\T$ and a hole  $h\in C$.  The purpose is that a billiard shot sends $p_1$ to hit $p_2$ (after any number of bounces) so that $p_2$ in turn  goes eventually into $h$ (where we suppose that $p_2$ maintains the direction of $p_1$). Again, there is finiteness, except in a well described situation:

\begin{thm}\label{T.buca} Let $p_1,p_2\in \T^0$ be distinct interior points and  $h\in C$ a point in the border.  
There are only finitely many shots $(p_1,v)$ such that the trajectory meets both $p_2$ and $h$, unless $p_1,p_2$ are the foci of $C$. 
\end{thm} 

Again, the proof in the special case when $C$ is a circle is  different (and easier) from the diophantine viewpoint, involving roots of unity instead of elliptic torsion.  \medskip

Still another finiteness result  
concerns trajectories passing several times  through a given point. A shot from a point $p_0$ giving rise to a trajectory eventually passing again through $p_0$ will be called a {\it boomerang shot} (from $p_0$).

Boomerang shots from a point $p_0$ can be of three types:

\begin{enumerate}

\item periodic trajectories: in that case they pass through $p_0$ infinitely often;

\item the trajectory passes through $p_0$ a second time with the same direction but a different orientation; in general these trajectories intersect $p_0$ exactly two times;

\item the trajectory passes through $p_0$ with another direction; again these trajectories generically intersect $p_0$ exactly two times.

\end{enumerate}


Note that if a trajectory is of type $(2)$ and $(3)$ at the same time, then it intersects $p_0$ at least three times.  If it passes through $p_0$ five times, then clearly it is periodic.

We have:

\begin{thm}\label{T.ritorno} Let $p \in\T^o$ be an interior point, not a foucs of $C$. Then there are infinitely many boomerang shots from $p$  of each of the three types above. There are only finitely many boomerang shots which belong to   two distinct  types as above. In particular, there are only finitely many shots giving rise to a trajectory  passing through $p$ exactly three or four times. 
\end{thm}

By the same methods, we could prove the further  following  finiteness result: 

{\it Let $p_0,p_1$ be two points of $\T$, not both foci. There exist only finitely many boomerang shots from $p_0$ hitting $p_1$.}

\smallskip

In general, if the point $p_0$ is given in advance,  there exist no boomerang shots from $p_0$ hitting another given point $p_1$; also there are no non-periodic boomerang shots passing three times through $p_0$. Thanks to the mentioned effective results of Binyamini in some ingredients of our proofs, it  should be possible to effectively  decide whether such trajectories exist or not for a given point $p_0$.

\begin{small} 
\begin{rem}\label{R.Betti-map}
As mentioned above, if the starting point $p_0$ is allowed to vary on the billiard table, then this extra degree of freedom enables to produce infinitely many cases where e.g. a non-periodic boomerang shot hitting $p_0$ three (and even four) times exists; or, given $p_1$, one can find infinitely many points $p_0$ from which a boomerang shot hitting $p_1$ can be found. The proof of these existence results, however, needs a study of the Betti map associated to a section of a higher dimensional abelian scheme; in particular, one should prove the non degeneracy of this map in an open set of real points. This could be done by combining the arguments of \cite{CMZ} with those of section 9 of \cite{ACZ}.
\end{rem}
\end{small}

\smallskip

\begin{rem}  \label{R.altri} {\tt About other  billiards}. 
It is natural to ask what happens of these statements on considering other types of billiards. 

Let us for instance take rectangular ones, or more generally those whose table is a parallelogram. We may view them as associated to  a lattice $L$ in $\C$.  
It is easy to see that the analogue of Theorem \ref{T.esiste} holds (though the results on heights recalled in Remark \ref{R.realicommenti} do not). 

As to Theorem \ref{T.angolo}, we shall prove the following:

\begin{thm}\label{T.angolo-parallelogramma}
Let $L\subset \C$ be a lattice. If for some point $p_0\in \C$ and some angle $\alpha\in (0,\pi)$ there are more than three  pairs of periodic trajectories for the billiard associated to $L$ which pass through $p_0$ and form an angle $\alpha$ at $p_0$, then $\C/L$ has Complex Multiplication. Vice-versa, if  $\C/L$ has CM then for  infinitely many $\alpha\in (0,\pi)$ and for every $p_0\in \C$ there are infinitely many pairs of periodic orbits passing through $p_0$  and forming at $p_0$ an angle $\alpha$. 
\end{thm}  
We shall give the proof  in \S \ref{SS.altri}. 

One can carry out a similar 
analysis also for the other results here; for Theorem \ref{T.buca} this depends not only on the billiard but also on the points $p_1,p_2,h$ and similarly for Theorem \ref{T.ritorno}.  

 We think it could be  not free of interest to explore what happens of these assertions for more general polygonal billiards, e.g. those considered in the above quoted papers. We owe to C. McMullen the remark that for regular $n$-agonal billiards, the cross ratios of the slopes giving rise to periodical orbits belong to a fixed number field, actually to a cyclotomic field (see e.g. the paper \cite{CalSm} by K. Calta and J. Smillie). This implies that the field of definition of the periodic slopes is finitely generated.
 In the case of the square billiard, this field is in fact $\Q$. This phenomenon is in contrast with what happens with elliptical billiard, where the degree of the field of definition of periodic orbits tends to infinity with the order of the orbit. This again indicates that the arithmetic in the case of elliptic billiards plays a special role. 
 \smallskip
 
 As to  billiard tables which are smooth convex curves, again the finiteness results do not generally hold. For instance, in the case of Theorem \ref{T.angolo},  suppose that the players' first shot sends the initial ball from $p_0$ to points  $p_1$, resp. $p_2$ on the border, so that the directions differ by an angle $\alpha$. Now, on deforming smoothly the border curve in small neighbourhoods of $p_1$ and $p_2$ one can ensure that there are infinitely many shots for which both trajectories are periodic. One could ask what would happen by imposing further assumptions on such billiards, e.g. {\it integrability } conditions. However, this would easily lead again to elliptic billiards, according to a conjecture of Birkhoff (see \cite{sor}).
 
 \smallskip
 
 {\tt Diophantine content}. 
 It will be clear from the argument  (take for instance Theorem \ref{T.angolo}) that they depend on the fact that a  real-analytic arc  in a real torus $\R^2/\Z^2$  under appropriate assumptions can contain only finitely many rational points. In the present situations, these assumptions are strongly related to  the algebraicity of the elliptic family  $\L$ associated to the billiard.  In the case of more general curves inside such a torus, this finiteness may well fail;  in any case such finiteness would be probably  linked to deep results in transcendence, depending heavily on the context and generally falling outside the present methods. 
 
 Similar considerations hold for the other finiteness results proved in the paper.
\end{rem}

\smallskip

As mentioned above, we shall see how the above  results may be framed  as cases of a general situation  in algebraic dynamics, stated above as a ``Question''. 
In this direction, we provide evidence by proving another theorem, not directly related to billiards, but perfectly fitting into the mentioned general  issue, actually being one of the simplest possible examples of it.

\begin{thm}\label{T.P_2} Let $L_1,L_2,L_3$ be (complex)  lines in $\P_2$ and let $\beta\in{\rm Aut}(\P_2)$ be a linear automorphism. 
Suppose that the three lines lie in different orbits for the action of $\beta$ and that none of them contains a fixed  point for $\beta$.
Then there are only finitely many $\beta$-orbits of points in $\P_2$ which intersect all the three lines $L_1,L_2,L_3$.
\end{thm}

It is clear that if $\beta$ has infinite order there are always countably many  orbits intersecting two given lines. On the other hand, it is natural to expect finiteness starting with three lines. Such finiteness cannot hold whenever one of the three lines is sent into another one by  a power of $\beta$. On the contrary, the condition that the three lines contain no fixed point for $\beta$ can be relaxed, but not omitted. In Proposition \ref{P.P_2}, and in the examples following its proof,  we shall give a complete description of the cases in which there exist infinitely many orbits intersecting three given lines. 

\smallskip
Our finiteness statement is formulated in terms of orbits, not of points. One can ask what happens if one of the three lines, say $L_1$, contains infinitely many points whose orbits meets both $L_2$ and $L_3$. These points might lie in only finitely many orbits. Combining the classical Skolem-Mahler-Lech theorem with our Theorem \ref{T.P_2} we can give a complete classification of the  cases when this happens (see Proposition \ref{P.SML}): in particular, $L_1$ must have only finitely many images under $\beta$. 
\smallskip

As said, the proof of Theorem \ref{T.P_2}  will require entirely different tools compared to the previous theorems, though several of the preceding statements (especially Theorem 1.5) could be phrased in terms of finiteness of orbits intersecting given curves.  After the proofs, we shall discuss about a plausible general  conjecture for finiteness in these situations. We shall remark how the most obvious tentative would lead to counterexamples, so some care is needed before putting forward some general (hypothetical) statement. 

\medskip

{\tt About our proofs}. The starting points for the proofs of all our theorems on elliptical billiards are the following two facts:

(i)  first, all segments of a billiard trajectory are tangent to a same conic, confocal with the billiard table $C$, called the `caustic' of the trajectory; 

(ii) second, once a caustic is fixed, the     set of possible pairs $(p,l)$ where $p\in C$ and $l$ is a line passing through $p$ and tangent to the given caustic, is naturally 
an algebraic variety which turns out to be a curve of genus one.

Varying the caustic (through a parameter)  produces an algebraic  one-dimensional family of   curves of genus one: in another language, an elliptic surface. Recall that genus one curves can be given a group structure (starting from a marked point),  producing  
elliptic curves.  The `billiard map', sending the pair $(p,l)$ to the pair $(p',l')$ obtained after the bounce,  consists in a translation with respect to this group structure.  (These two facts will be recalled  in detail in next paragraphs.)

Once  this setting is appropriately formulated, the proofs of our finiteness results (Theorems \ref{T.angolo}, \ref{T.buca}, \ref{T.ritorno}) rely on fairly recent results, mainly obtained via a method introduced by D. Masser with the second author  in \cite{MZ}, stating the finiteness of simultaneous torsion conditions  for several sections of an elliptic scheme, or, more generally, finiteness  for the set of points of the base curve where independent sections take values satisfying several dependence relations (these results, also by other authors, will be recalled in detail below).

\smallskip

Consider for instance the case of Theorem \ref{T.angolo}. Choosing a point and shooting the ball from it with a certain direction produces a point on a curve of genus $1$ (depending on the direction); after the first bounce one obtains another point on the same curve. Well,  their `difference'  on the elliptic curve associated to the curve of genus $1$ is a torsion point  if and only if the orbit is periodic. Two shots whose directions differ by a fixed angle as in Theorem \ref{T.angolo} will produce two different points on two different elliptic curves. The theorem asserts that these two points cannot be simultaneously torsion, apart for finitely many cases.

Geometrically, the real points of an elliptic curve form a circle (or a pair of circles); locally, the family of such circles, for varying the direction of the shots, is topologically constant, so every shot gives rise to a point on a fixed circle.  Varying the pairs of shots with a fixed  angle between them  produces a curve on the product of two copies of a circle, i.e. a two-dimensional real torus. Theorem \ref{T.angolo} follows from the fact that this curve does not meet infinitely many rational  points on the torus.  This kind of result strongly depends on the special nature of the curves which arise and would not be true in a  general context. For this reason we suspect that hardly  other proofs of these results could be obtained without appealing to the arithmetic and geometry of the context.

Yet in another language, working on the universal cover of the torus, we obtain a real-analytic  (transcendental) curve in $\R^2$ and our theorem follows from the fact that this curve cannot contain infinitely many points with rational coordinates. Ultimately,  a diophantine result of J. Pila (originating in work by Bombieri-Pila and eventually heavily generalized by Pila and J. Wilkie to higher dimensions) on rational points on transcendental surfaces must be used to prove this last fact, together with height considerations and Galois-theoretic properties coming from the arithmetic theory of elliptic curves.  For an account on these techniques, see the second author's book \cite{Z}. However, in this work some further ingredients  are needed for the proofs of both  the finiteness results like Theorem \ref{T.angolo}, \ref{T.buca} and the existence results like Theorem \ref{T.esiste}. 

\smallskip

Concerning our Theorem \ref{T.P_2}, the methods of proof are completely different; indeed, the possible infinite families of points on the first line whose orbit intersects the two other lines are automatically defined over a fixed number field. They will give rise to a system of exponential diophantine equations, or equations of mixed polynomial-exponential type. One then can apply rather classical results from the theory of diophantine equations involving linear recurrence sequences and relying on the Schmidt Subspace Theorem.

\bigskip

{\bf Acknowledgements}. The authors are grateful to S. Marmi and P. Oliverio  for interesting discussions and for pointing out some relevant references. 
The topic of this work was the object of some talks, for which the authors thank the organizers.  The discussion arising in these talks involved attracting issues and led to some improvements in the paper. In particular, we are pleased to thank C. McMullen and A. Sorrentino.

\bigskip

\section{Preliminaries and auxiliary results}\label{S.pre}

In this section we shall recall some basics from the theory of elliptical billiards, and we shall see how the space parametrizing  billiard shots gives rise to an elliptic scheme. Similarly, we shall  see how the phase space gives rise to an elliptic surface. Finally, we shall recall some auxiliary theorems of number-theoretical/geometrical nature which will be applied to derive the sought results.

\subsection{Elliptical billiards} The first appearance of properties of the elliptical billiard seems to be rather old, i.e. going back to the XVIII Century,  in Boscovich's \cite{B} (a reference which we thank  P. Sarnak for).  

As above, by {\it elliptical billiard} we mean a billiard whose table $\T$ is an ellipse $C$ together with its interior $\T^o$. For definiteness, suppose that $C$ has  foci at $(\pm c,0)$ for a $c\in (0,1)$, and (affine) equation 
\begin{equation}\label{E.C}
C:\qquad x^2+{y^2\over 1-c^2}=1.
\end{equation}
It is a remarkable theorem (capable of an `Euclidean proof', see e.g. \cite{T}, Thm. 4.4) that a billiard trajectory remains tangent to a confocal conic, called {\it caustic}.  This will be a hyperbola  or an ellipse  according as the trajectory crosses or not the segment joining the foci, a property which shall be shared by all segments of trajectory. (When the trajectory passes through a focus, it will continue to pass alternatively through the two foci, which amounts to a degenerate case of the former description, when the caustic becomes  the segment connecting the foci.)    Let us write an affine equation for the family of confocal conics:
\begin{equation}\label{E.caustic}
C_s:\qquad {x^2\over s}+{y^2\over s-c^2}=1,
\end{equation}
where $s$ is a parameter.  For {\it real}  billiard trajectories, we shall have $0<s<1$, and the caustic will be a hyperbola precisely  for $s<c^2$.

Suppose that $(p,v)$ is a billiard shot in the phase space, so $p\in C$.  Then  the line determined by $(p,v)$ will be tangent to a unique caustic $C_s$ ($s$ depending on $(p,v)$), so $(p,v)$   corresponds to a point in the so-called {\it dual} conic $\widehat{C_s}$. This has equation, in affine coordinates $t,u$, 
\begin{equation}\label{E.dual}
\widehat{C_s}:\qquad st^2+(s-c^2)u^2=1,
\end{equation}
in the sense that a line defined in the $xy$-plane by $tx+uy=1$ is tangent to $C_s$ if and only if $(t,u)$ verifies \eqref{E.dual}. 

\smallskip

{\tt A rational parametrisation of $C$}.  We give explicitly a rational parametrisation  of the ellipse $C$, which will be independently useful. The projective closure in $\P_2$ of $C$ becomes isomorphic to $\P_1$ on using for instance the map $z$ and its  inverse map given by 
\begin{equation}\label{E.par} 
z:=(x,y)\mapsto y/(x-1),\qquad \quad x=(z^2+c^2-1)/(z^2+1-c^2), \quad y=   2z(c^2-1)/(z^2+1-c^2).
\end{equation}

The group of automorphisms of $C$ generated by sign change on the coordinates corresponds to the group generated by $z\to -z$ (corr. to $(x,y)\to (x,-y)$) and $z\to (c^2-1)z^{-1}$ (corr. to $(x,y)\to (-x,-y)$). 

\smallskip

$\bullet$ We note once and for all that the above equations are  subject to restrictions (for instance $s\neq  0, \pm c$ and for $s=\pm 1$ the caustic becomes $C$), and are {\it affine}. However, as usually happens, we shall always  tacitly consider their completions in $\P_2$,  and  occasionally suitably interpret the objects also for the forbidden values.


\subsection{Correspondence with elliptic curves} We now resume in short  the correspondence  of this context with  elliptic curves, due essentially to Jacobi, who used it to give a proof of  the {\it Theorem of Poncelet}, which contains as a special case the situation of the elliptical billiard. (See  for instance the above quoted sources and the authors's booklet \cite{CZPonc}, and  the references therein for more.)  

Given  a caustic $C_s$, one can consider the curve $\E_s\subset C\times \widehat{C_s}\subset \P_2\times \P_2$ defined as the closure of the set of pairs $(p,l)$, where $p\in C$ belongs to the line $l\in \widehat{C_s}$ tangent to $C_s$. Note that for $s,p,l$ defined over $\R$ such a pair determines uniquely a point $(p,v)$ in the phase space.

  It turns out that $\E_s$ is smooth and has genus $1$ for complex  $s \neq 0, 1,c^2,\infty$.  The curve $\E_s$ has a Jacobian, which is an elliptic curve (see \cite{SilEC}), and may be seen as the connected component of the identity in the group of automorphisms of $\E_s$ (they have no fixed points except for the identity). If we let  $s$ vary arbitrarily in $\C$, the  field of definition is immaterial, but if we consider $s$ as a variable, then the ground field of the whole construction may be taken as  $\Q(c,s)$. This will be relevant in some verifications below. 

The above equations give the following presentation for $\E_s$ inside $\P_2\times\P_2$, written for simplicity in affine coordinates $(x,y)\times (t,u)$, but which should be of course extended on using bi-homogeneous coordinates:
\begin{equation}\label{E.Ys}
\E_s:\quad \left\{
\begin{matrix} x^2+{y^2\over 1-c^2}&=&1\\
 tx+uy&=&1\\ 
 s t^2+(s -c^2)u^2 &=&1.
 \end{matrix}\right.
\end{equation}
Here $p=(x,y)$ is a point on the ellipse $C$ whereas the line $l$ (containing $p$) corresponds to $(t,u)$ under the usual duality on $\P_2$ (expressed through the middle equation).

 \smallskip

{\tt Automorphisms of order $2$}. Observe that  for $s\neq 0, 1, c^2,\infty$   there is a group of four automorphisms of $\E_s$, defined over $\Q$, without fixed points, and  isomorphic to $(\Z/2)^2$:  the group is represented by $(x,y)\times (t,u)\to (\epsilon x,\eta y)\times (\epsilon t,\eta u)$, where $\epsilon,\eta=\pm 1$, indeed without fixed points for the said values of $s$. (These elements correspond to  the four points of order $2$ on $J_s$.) 
\smallskip

$\bullet$ If we want to refer to the whole {\it total} space of all $\E_s$ we shall use the notation $\E$.

A choice of a point as an origin will give a model of $\E_s$ as an elliptic curve (omitting the said values of $s$). Below we shall give an explicit Legendre model.

\subsubsection{The billiard map} A segment of billiard trajectory is determined by two ordered points $p,p'$ on $C$. If $v$ is a unit (real) vector which represents  the direction $p\to p'$, then $(p,v)$ is in the phase space. We let $(p',v')$ correspond to  the next segment of trajectory. The billiard map, usually denoted $T$ in this paper,  sends $(p,v)$ to $(p',v')$.  Because of the result recalled above, both of these pairs  determine the same caustic $C_s$, and then the billiard   map  acts on any $\E_s$ and in fact  is an (algebraic) automorphism of $\E_s$.  Thus $T$ makes sense also for {\it complex} points of $C$, numbers $s$ and complex vectors $v$.\footnote{We wonder whether the complex billiard dynamics, seen as a dynamics on a real surface   corresponding to the ellipse, may be given  a direct simple geometrical description.}

The billiard map  may be realized as the composition of two other automorphisms, actually involutions. A first involution, denoted $\iota$, sends $(p,l)$ to   $(p',l)$ (so it sends $(p,v)$ to $(p',-v)$ if we want to work on the phase space).  To define a second involution, denoted $\iota^*$, let  $l^*$ be  the other tangent from $p$ to $C_s$ (possibly $l^*=l$ whenever $p\in C\cap C_s$). Then we put $\iota^*=(p,l^*)$. (Note that the map $\pi:\E_s\to C$ has the pairs $x,\iota^*(x)$ as fibers.) Clearly $\iota, \iota^*$ are involutions. It is immediately checked that $T=\iota^*\circ \iota$. 

One may also easily see that both $\iota,\iota^*$ have fixed points (which are ramification points of the corresponding quadratic maps); it follows from general theory of curves of genus $1$  that $T$ has no fixed points and thus  is a translation  as an element of the jacobian $J_s$ of the curve $\E_s$ of genus $1$. We have moreover the\smallskip

 {\bf Theorem of Poncelet}: {\it The translation $T$ depends only on the caustic, not on $p$.} \smallskip

 This follows again from the general theory of curves of genus $1$, the proof by Poncelet\footnote{Poncelet considered actually a more general dynamics involving arbitrary pairs of non-tangent conics} however predating such facts. 
  See \cite{T}, p. 58 (especially Cor. 4.5),  for a somewhat different description of the translation, without invoking elliptic curves, and described as an element of $S_1$. This is consistent because the connected component of the identity in the group of {\it real} points of an elliptic curve over $\R$ is a circle. The present description covers also the case of complex coordinates. (Compare with Remark \ref{R.real} above.)  
  
  {\tt An invariant}.   It is also proved in an elementary geometric way that  (in our notation) the function 
 $(1-c^2)xv_1+yv_2$ is invariant for the billiard map (i.e. is an {\it integral}) for $v=(v_1,v_2)$, and indeed its square  is found to be $(1-c^2)(1-s)$. This corresponds to the constancy of a certain (easily written down)  rational function on the above elliptic curves.  We also note that there is an    area form, invariant   for the billiard map, on the phase space, given by $\sin\alpha\cdot  \d\alpha\wedge \d t$, where $t$ is an arc-length parameter on $C$ and, for $(p,v)$ in the phase-space,  $\alpha$ is the angle between $v$ and the positive tangent to $C$ at $p$ (this was discovered by Birkhoff \cite{Bi}, see also \cite{T}, p. 33).  
 
 These very interesting invariances can be used to prove many results without invoking elliptic curves. However (as we shall remark with explicit examples), the latter  description  seems indispensable to deal with other conclusions  of this paper, especially the present finiteness theorems.


\subsection{Billiard shots and elliptic schemes with sections}  We have recalled that each billiard shot defines a caustic, and in turn a curve $\E_s$  of genus $1$ and its Jacobian $J_s$, which is an elliptic curve. If we disregard the ground field, $J_s$  is isomorphic to the curve $\E_s$ itself once we have chosen  a point on it as origin.  If we work over $\Q(c,s)$ where $s$ is a variable this procedure increases the ground field. Note that this viewpoint  is like having a scheme over the $s$-line whose fibers are elliptic curves. We indicate briefly two distinct (but equivalent) ways of dealing more explicitly with this.

In the first place, it is relevant  to consider the intersection $C\cap C_s$, consisting of  the ramification points of the projection of $\E_s\subset C\times\widehat {C_s}$ to $C$.  One easily finds that
\begin{equation}\label{E.ccs}
C\cap C_s=\{(\pm x_0,\pm y_0)\}, \qquad x_0^2={s \over c^2},\quad y_0^2={(1-c^2)(c^2-s)\over c^2}.
\end{equation}
If these points are pairwise distinct, which amounts to $s \neq 0, c^2,\infty$, then $\E_s$ is indeed smooth of genus $1$.  Of course to choose these points we have to perform a base change from the $s$-line, defined by the equations on the right of \eqref{E.ccs}.

Note that, considering the presentation \eqref{E.Ys}, we have a natural projection $\E_s\to C$. For a point $p\in C$ the inverse image consists of the tangents from $p$ to $C_s$. These tangents coincide precisely when $p\in C_s$, i.e. when $p$ is one of the points in question. Hence these points are precisely the ramification points of the said projection. We remark that knowledge of these points determines the curve $\E_s$ up to isomorphism (over an algebraic closure of the ground field).

 We must also exclude for the moment $s =1$, i.e. $C_s=C$; however soon we shall see that on considering other models we may allow this value as well, in the sense that it yields an elliptic curve, though not associated to the billiard. (When $C=C_s$ one might also  conceive a billiard shot as acting only on balls on $C$ and as producing no effect on the ball.)

We finally note that , for real $s\neq 0,1,c^2$, these points are real precisely when $s< c^2$,  
 which happens when $C_s$ is a hyperbola. This is clear also geometrically. When $s>c^2$ the caustic is an ellipse and these points have real $x_0$ and purely imaginary $y_0$.


\subsection{A Legendre model} \label{SS.legendre}  We now want to construct an explicit isomorphism of the curve $\E_s$ of genus $1$ given by \eqref{E.Ys} (where $s$ is a variable) with an elliptic curve in Legendre form; this will be convenient for some calculations.

Using the map $z$ of \eqref{E.par} giving an isomorphism $C\cong \P_1$,  a cross-ratio  of the above four points in \eqref{E.ccs}, denoted here $\lambda$,  is  is easily calculated (as obtained below in \eqref{E.lambda}) to be 
\begin{equation}\label{E.cr}
\lambda=x_0^2={s \over c^2}.
\end{equation}
This already shows that the $j$-invariant of $J_s$ is a non-constant  rational function of $s $, so our scheme is not isotrivial.

To construct an isomorphism with a Legendre curve, let now $\zeta\in \mathrm{PGL}_2$ send three of the four points \eqref{E.ccs} to $0,1,\infty$; we can take  $\zeta$  to be   defined   over $K(x_0,y_0)$, where we put $K:=\Q(c,\sqrt{1-c^2})$.  Let us carry out explicitly this construction, for which we have no bibliographical reference.

Choosing $x_0,y_0$ as in  \eqref{E.ccs}, let us put  for this argument $P_1=(x_0,y_0)$, $P_2=(-x_0,y_0)$, $P_3=(x_0,-y_0)$, $P_4=(-x_0,-y_0)$, and let us denote $z_i:=z(P_i)$, so for instance $z_1=y_0/(x_0-1)$. 

The cross ratio (for a certain ordering)  is given by
\begin{equation}\label{E.lambda}
{(z_1-z_2)(z_3-z_4)\over (z_1-z_4)(z_3-z_2)}=x_0^2=\lambda.
\end{equation}

A homography $\zeta$ as required, putting $\zeta_i:=\zeta(z_i)$, is given by
\begin{equation}\label{E.zeta}
\zeta(z)={z-z_2\over z-z_4}\left({z_3-z_4\over z_3-z_2}\right)={(x_0+1)z+y_0\over (x_0+1)z-y_0}x_0,\qquad  \zeta_1=\lambda,\ \zeta_2=0,\ \zeta_3=1,\ \zeta_4=\infty.
\end{equation}

\bigskip

Then  the function $\zeta\circ z\circ \pi$ from $\E_s$ to $\P_1$ (where $\pi$ is the projection $\E_s\to C$)  will send 
 the four points \eqref{E.ccs}  resp. to $\lambda, 0,1,\infty$. Note that $\zeta(z)$ is a cross-ratio of $z,z_2,z_3,z_4$.

  It is not difficult after  elimination from \eqref{E.Ys} to realize $\E_s$ as a ramified quadratic cover  of $\P_1$ through this function, branched  exactly above  $0,1,\infty,\lambda$ (in the above order), of index $2$.
  
  Explicitly, on elimination we find the equation for $t$ over $\Q(c,s,x)$ given by
  \begin{equation*}
  \left(c^2(s-1)x^2+(1-c^2)s\right)t^2+2(c^2-s)xt+\left(s-1+(1-c^2)x^2\right)=0.
  \end{equation*}
  The  discriminant is given by 
  \begin{equation*}
 { {\rm Discriminant}\over 4}=c^2(1-c^2)(1-s)(x^4+(1-{s\over c^2})x^2+s)=c^2(1-c^2)(1-s)(x^2-1)(x^2-{s\over c^2}).
  \end{equation*}
 Note that $(1-c^2)(1-x^2)=y^2$, and we may express $x$ in terms of $z$ in the last factor, which is $(x-x_0)(x+x_0)$.  After a few calculations we see that we may  thus recover $t$  (and $u$)  in the quadratic extension of   $\Q(c,s)(z)$ given by $\Q(c,s)(z,w)$, where 
 \begin{equation}\label{E.w}
 w^2=(s-1)\left(z^2-{y_0^2\over (x_0+1)^2}\right)\left(z^2-{y_0^2\over (x_0-1)^2}\right),
 \end{equation}
  and conversely, we may recover $w$ from $z,t$. 
 Indeed, we have
 \begin{equation*}
\left(c^2(s-1)x^2+(1-c^2)s\right)t=(s-c^2)x+ {cy_0\over z^2+1-c^2}w.
 \end{equation*}

We also  have $w=\pm \sqrt{s-1} \sqrt{\prod(z-z_i)}$, and  the function field of  the curve $\E_s$,  over the field  $\Q(c, \sqrt{1-c^2}, \sqrt{s-1},x_0,y_0)$, is obtained by adding  the element $\sqrt{\prod(z-z_i)}$. 

\medskip

   In conclusion, we find that  $\E_s$, and also $J_s$,  becomes isomorphic to the Legendre curve  
 \begin{equation}\label{E.leg}
 \L_s=L_\lambda: \qquad Y^2=X(X-1)(X-{s\over c^2}),\qquad s=c^2\lambda.
 \end{equation}
 where we use for the moment capital letters to avoid confusion with $x,y$ on $C$.  
 
  By the above formulae, a  field of definition of an isomorphism is found to be the extension 
 $\Q(c,\sqrt{1-c^2})(\sqrt{s-1},x_0,y_0)=K(\sqrt{s-1},x_0,y_0)$, where actually the $X$-coordinate is defined over $K(x_0,y_0)$.

 More explicitly, a point $(x,y)\times (t,u)$ on $\E_s$ in the model \eqref{E.Ys} goes under this isomorphism  to the point $(\zeta(z(x,y),\eta(x,y,t,u))\in\L_s$, where $\eta $ is a certain rational function  obtained as follows. We have
 \begin{equation*}
 \zeta(z)(\zeta(z)-1)(\zeta(z)-\lambda)={x_0(x_0-1)(x_0-\lambda)\over (z-z_4)^4(s-1)}w^2,
 \end{equation*}
 where $w$ is given by \eqref{E.w}. On using that  $x_0(x_0-1)(x_0-\lambda)=-x_0^2(x_0-1)^2$ we find
 \begin{equation}\label{E.expl.leg}
  \zeta(z)(\zeta(z)-1)(\zeta(z)-\lambda)=\left({\pm x_0(x_0-1)\over (z-z_4)^2\sqrt{1-s} }w\right)^2,
 \end{equation}
 so we may take $\eta$ as the expression under brackets on the right (after a choice of sign).

 By construction, the point at infinity for the Legendre curve corresponds to one of the four points \eqref{E.ccs}, namely $P_4=(-x_0,-y_0)$ which is defined over $K(x_0,y_0)$.   Of course, we have a wide choice for an origin on each $\E_s$, i.e. of a {\it zero section}, so that $\E_s$ becomes an elliptic curve (scheme). The present choice seems to us natural, because it depends on the points \eqref{E.ccs} and is also such  that we fall into the familiar framework provided by a Legendre form.

 \medskip
 
 {\tt About the field of definition}. As said, we have an isomorphism $\phi:\L_s\to \E_s$ defined over $K(\sqrt{1-s},x_0,y_0)$.  If $c$ is a given (real)  number and $s$ is a variable (over $\C$), this last field is a Galois  extension of $K(s)$ of degree $8$, with Galois group, denoted here $G$,  isomorphic to $(\Z/2)^3$, whereas $K(s)$  is a field of definition for  the curves $\E_s,\L_s$. Now, for $g\in G$, we can consider the automorphisms  of $\L_s$ given by $\phi_g:=(\phi^g)^{-1}\circ \phi$.  These form a $1$-cocycle for $G$ with values in Aut$(\mathcal{L}_s)$. Looking at the action on the points of order $2$ of $\L_s$, which correspond through  $\phi$ to  the four ramification points considered above, it is not difficult to prove that  the $\phi_g$ are precisely those automorphisms of the form $x\mapsto \pm x+t$, for $x\in \L_s$, where $t$ varies among the four points of order $2$. Also, if $g_0\in G$ fixes $K(x_0,y_0)$ but moves $\sqrt{1-s}$ then $\phi_{g_0}(x)=-x$. (See \cite{CZPonc}.) 
 
 We also note that the Jacobian of $\E_s$ over say $\C(s)$ would be an elliptic curve defined over $\C(s)$, so the appearance of $\sqrt{1-s}$ may appear strange. Indeed, it may be checked that the Jacobian is not the Legendre curve, but the twist of it through $\sqrt{1-s}$, given in practice by the same Legendre equation, but multiplied by $1-s$ on the left. However we have preferred to work with the standard Legendre model, at the cost of enlarging the field.

 \medskip
 
 \subsection{Review of generalities about the Legendre curve}  In this subsection we  review a few facts about the Legendre curve from a complex viewpoint, namely we analyse Weiestrass functions and periods, as $\lambda$ varies in $\P_1(\C)-\{0,1,\infty\}=\C-\{0,1\}$. We shall refer to Silverman's \cite{SilEC}
 and especially Husemoller's \cite{Hus} books. 
  
 \subsubsection{Weierstrass equations}\label{SSS.weier}   First, let us put the Legendre curve in {\it pure}  Weierstrass form.  Setting $U=X-{\lambda+1\over 3}$, $V=2Y$, we find from \eqref{E.leg} (recalling $s=c^2\lambda$),
 \begin{equation*}
 V^2=4U^3-{4\over 3}(\lambda^2-\lambda+1)U-{4\over 27}(\lambda-2)(\lambda+1)(2\lambda-1).
 \end{equation*}
 For $\lambda\in\C-\{0,1\}$ this is  indeed the equation of a complex elliptic curve in Weierstrass form and there is a unique lattice $\Lambda=\Lambda_\lambda\subset \C$ such that the corresponding Weierstrass coefficients $g_2(\Lambda),g_3(\Lambda)$ yield precisely the given equation, which can then be parametrized by $U=\wp_\Lambda(\mu)$, $V=\wp_\Lambda'(\mu)$ for a complex variable $\mu$.
 
 \subsubsection{Differentials}  Note that the differential $\d\mu$ corresponds to $\d \wp/\wp'=\d U/V$, and in turn to $\d X/2Y$  on the Legendre model \eqref{E.leg}. This is consistent with \cite{Hus}.
  
  \subsubsection{Caustics} For our billiard purposes we shall be mainly interested in the case of real $\lambda$. Actually,  our {\it real} caustics  will occur for $0<s<1$, corresponding to $0<\lambda<1/c^2$.  For $0<\lambda<1$ the caustic will be a hyperbola whereas for $1<\lambda<1/c^2$ it will be an ellipse. 
 
There is a degenerate case at $s=c^2$, i.e. $\lambda=1$, when the caustic becomes the horizontal segment  connecting the foci $(-c,0), (c,0)$.    It is very easy to realize that a billiard trajectory passing through some focus will alternatively pass through the two foci and will tend to become horizontal; it will become horizontal only when it is originally horizontal, with period $2$.

At $s=0,1$ again there is a degenerate situation: in the latter case the caustic coincides with the ellipse and the billiard trajectories will degenerate to points. The situation is more complicated in the former case, as will be seen in the Appendix.

 \subsubsection{Periods}\label{Periods}  We borrow mainly from \cite{Hus}, Ch. 9 (and see also the authors's book in progress \cite{CZPonc}). One may take generators for the lattice which depend locally analytically on $\lambda$, but have monodromy as we move $\lambda$ in $\C-\{0,1\}$.  In the region $\max(|\lambda|,|1-\lambda|)<1$,  we may pick the periods as given by
 \begin{equation*}
 \omega_1(\lambda)=i\pi F({1\over 2},{1\over 2},1,1-\lambda),\qquad \omega_2(\lambda)=\pi F({1\over 2},{1\over 2},1,\lambda),
 \end{equation*} 
 where $F({1\over 2},{1\over 2},1,\lambda)=\sum_{n=0}^\infty{-1/2\choose n}^2\lambda^n$ is a hypergeometric function. Note that for real $\lambda$ in the region the second period is real positive  whereas the first is purely imaginary; the possibility of choosing such periods depends (as is not difficult to prove) on the fact that the real points of the corresponding elliptic curve have two components. We have the formulae (\cite{Hus}, Ch. 9, Thm. 6.1) 
 \begin{equation}\label{E.periods}
 \omega_1(\lambda)=\int_{-\infty}^0{\d x\over y}\in i\R, \qquad \omega_2(\lambda)=\int_1^\infty  {\d x\over y}\in \R,
 \end{equation}

where we choose as $y$  the  square root: $y=\sqrt{x(x-1)(x-\lambda)}$ with positive imaginary resp. positive real part, resp. in the first and second case.  

For $|\lambda|<1$ we also have the formula $\omega_2(\lambda)=2\int_0^{\pi/2}(1-\lambda\sin^2\theta)^{-1/2}\d\theta$ (see \cite{Hus}, Thm. 5.9).  For $|1-\lambda|<1$ this yields $\omega_2(1-\lambda)=2\int_0^{\pi/2}(\sin^2\theta+\lambda\cos^2\theta)^{-1/2}\d\theta$, whence after approximating $\sin\theta,\cos\theta$ by $\theta,1$ near $\theta=0$, we easily obtain 
$\omega_2(1-\lambda)=-\log\lambda+O(1)$ as $\lambda\in\R$ tends to $0^+$.

The periods satisfy the Gauss-Legendre differential equation
\begin{equation*}
\Gamma(\omega)=\lambda(1-\lambda)D^2\omega+(1-2\lambda)D\omega-{1\over 4}\omega=0,\qquad D:={\d\over \d\lambda},
\end{equation*}
which also gives a possibility to obtain the analytic continuation. 

Using this differential equation we obtain that $D(\omega_1/\omega_2)=c(\omega_2^2\lambda(1-\lambda))^{-1}$, and by integration we find  $\omega_1(\lambda)= -\omega_2(\lambda)\log\lambda/\pi+$ analytic function in a neighbourhood of $0$, where the constant $c$ is found  by using the above asymptotic.  This yields  both the monodromy and the asymptotic at $0$ and similarly at $1$.  For instance the monodromy at $0$ is given by $(\omega_1,\omega_2)\to (\omega_1+2\omega_2,\omega_2)$.

\medskip

When $1<\lambda <1/c^2$ we can still use this representation, after a change of variables which we now explain. As in \cite{Hus}, the curve with parameter $\lambda $  is isomorphic to  the one with parameter  $\lambda':=1/\lambda$ under the isomorphism 
$X'=X^{-1}$, $Y'=Y(X^2\sqrt\lambda)^{-1}$, in the sense that the dashed coordinates satisfy the  equation with the dashed parameter if the former satisfy \eqref{E.leg}.

\subsubsection{Torsion points of order $2$, real points} \label{SSS.torsion} The torsion points of order  (exactly)  $2$ on the Legendre curve (with $\lambda\in\R$) are those with $X=0,1,\lambda$ (and $Y=0$). In a fundamental parallelogram for $\Lambda$ these  points  correspond to the $\omega_1/2, \omega_2/2$ and their sum, in some order. These points together with $0$ determine  a rectangle, of which the horizontal sides are sent to real points in the algebraic model, and the vertical sides to points with real abscissa. The   upper side must correspond to the compact connected component of real points in the affine plane, not meeting the origin, and this determines the picture. For instance, for $0<\lambda<1$ the lower right vertex $\omega_2/2$  is sent to $1$, and the upper vertices  $\omega_1/2, (\omega_1+\omega_2)/2$ are sent resp. to $0,\lambda$ (by connecting $1$ with $\lambda$). Instead, for $1<\lambda$, $\omega_2/2$ corresponds to $\lambda$.

 \medskip

\subsection{The billiard section on $\L_s$}   We can also give an explicit form to the billiard map, as a translation on the Legendre model \eqref{E.leg}. Let $\psi:\E_s\to\L_s$ be the inverse to $\phi$, and let us put 
 \begin{equation}\label{E.[iota]}
 [\iota]=\psi\circ \iota\circ\phi,\qquad  [\iota^*]:=\psi\circ\iota^*\circ\phi,\qquad [T]=\psi\circ T\circ\phi=[\iota^*]\circ [\iota].
 \end{equation}
 So, these are the previous maps, but viewed on the Legendre model. Therefore $[\iota],[\iota^*]$ are involutions on $\L_s$, each with some fixed point, hence of the form $[\iota](x)=\kappa-x$, $[\iota^*](x)=\kappa^*-x$ for $x\in\L_s$ and some $\kappa=\kappa_s,\kappa^*=\kappa^*_s\in\L_s$ (independent of $x$). 
 
 To find this explicitly, note that $\kappa=[\iota](0)$. On the other hand, $\phi(0)=(-x_0,-y_0)\in C$,   is one of the four points in $C\cap C_s$, so $x_0^2=s /c^2$, $y_0^2=(c^2-1)(s -c^2)/c^2$.   We have that $\iota(\phi(0))$ is $(x_0',y_0')\times \ell_0$, where $\ell_0$ is  (for the present argument) the tangent to $C_s$ at $(-x_0,-y_0)$.

   After some easy calculations one finds
\begin{equation*}\label{E.x0'}
x_0'=x_0\cdot {(1-2c^2)+c^2s\over (1+(c^2-2)s },
\end{equation*}
whereas $y_0'$ is found from the equation for $\ell_0$:  $(x_0/s )x+(y_0/(s -c^2))y=1$:
\begin{equation*}\label{E.y0'}
y_0'=y_0\cdot {1-c^2s\over 1+(c^2-2)s }.
\end{equation*}

   We have $\kappa=\psi((x_0',y_0')\times \ell_0)$.

Things are even simpler for $\kappa^*$:  in fact, thinking of a pair $(p,l)\in \E_s$, the involution $\iota^*$  does not act on the point $p$ on $C$, but exchanges the lines $l,l'$ through $p$,  tangent   to $C_s$. Now, since the relevant point $\phi(0)$  lies in $C\cap C_s$, the tangent line $\ell_0$ is unique, hence $\iota^*(\phi(0))=\phi(0)$ and $[\iota^*](0)=0$, which implies $\kappa^*=0$. 

In conclusion, the Billiard Map is represented by translation by $-\kappa=-\psi((x_0',y_0')\times \ell_0)$ on $\L_s$.  This is a section for our scheme, it  may be  also  thought as the linear equivalence-class of the divisor $(p',l')-(p,l)$ on $Y_s$ (which is independent of $(p,l)$ by the theorem of Poncelet).

\medskip

Let us now compute explicit coordinates for this section on $\L_s$, using the formulae given above.  We have no knowledge of such an explicit formula in the literature.

First, using the above equations, we find $z(x_0',y_0')=y_0'/(x_0'-1)=y_0(1-c^2s)(x_0(1-2c^2+c^2s)-1-(c^2-2)s)^{-1}$.  In turn, using \eqref{E.zeta}, after some calculations we find  that the $X$-coordinate of the billiard map (as a function of the caustic)  is given by
\begin{equation}\label{E.bmapX}
h(\lambda):=\zeta(z(x_0',y_0'))={(1-c^2)s\over c^2(1-s)}={(1-c^2)\lambda\over 1-c^2\lambda}.
\end{equation}
The $Y$-coordinate, denoted here  $k(\lambda)$,  may be found again using the above equations, or also directly from the Legendre equation. Since $h(\lambda)-1=(\lambda-1)/(1-s)$, $h(\lambda)-\lambda=c^2\lambda(\lambda-1)/(1-s)$, we find, 
\begin{equation}\label{E.bmapY}
k(\lambda)=\pm c\sqrt{1-c^2}{\lambda(1-\lambda)\over (1-s)\sqrt{1-s}},
\end{equation}
for some choice of the sign. As to this,  note that we may choose the sign arbitrarily for the $Y$-coordinate  in the choice of the isomorphism, and  such  choice reflects in the choice here. 

We shall put in the sequel $B(\lambda):=(h(\lambda),k(\lambda))\in L_\lambda$.

\begin{rem}\label{R.bmap} {\tt Shape of the billiard section}. 
The billiard map, represented by the billiard section $B$,  has turned out to have a very simple shape on the Legendre model. Indeed, on realising $\zeta$ as a cross-ratio one could predict {\it a priori} that $h(\lambda)$ had degree $1$, since the only pole could occur at $s=1$. Also, it turns out that $B(\lambda)$ may be expressed even more simply,  in terms of a section with {\it constant} $X$-coordinate.\footnote{Any such sections may be called  (a) `Masser  section', because Masser referred to them in the original formulation of problems which led to the results in \cite{MZ}, \cite{MZ2}, and subsequent papers.}  Indeed, we may easily check that
\begin{equation}\label{E.masser}
B(\lambda)=\pm ({1\over c^2}, {\sqrt{1-c^2}\over c^3}\sqrt{1-c^2\lambda})+(\lambda,0),
\end{equation}
where the sum is on $\L_s$ and the choice of the sign is at our disposal (but depending on the analogous  choice for $B(\lambda)$). Note that $(\lambda,0)$ is a section of order $2$; this formula then also simplifies some calculations with the Betti map (done in the Appendix). 

We give three quick  arguments for   the following easy but important fact:

\begin{prop}\label{P.nontors}
The billiard section  is non torsion.
\end{prop}

{\it First argument}. Let us start with any point $p_0\in C)$; when the (elliptic) caustic tends to the ellipse $C$, i.e. for $s\to 1^-$, the billiard map tends to fix the point, so by continuity it is not possible that a fixed number of  iterations  of the map produces the identity, since the map itself is not the identity. \smallskip

{\it Second argument}.    Take two coprime integers $m,n\geq 3$ and consider an $m$-gon of maximal length among the $m$-gon inscribed in $C$; also consider an $n$-gon of maximal length. Both correspond to periodic billiard trajectories, and so are circumscribed around caustics. This implies  that the billiard map takes torsion values of distinct, actually coprime, orders $m,n$. Now, if the billiard section would be torsion, say of order $k$, then at every caustic it would give rise to periodic trajectories of length dividing $k$ (actually of exact order $k$, but this is not   needed).  From this contradiction we obtain that the section is non-torsion, and actually we obtain that it assumes values which are real  torsion points of arbitrary prescribed order $>1$.

\smallskip

{\it Third argument}. From the explicit equation \eqref{E.bmapY} it is apparent that the section can be defined only over a ramified extension of the base field, while if it were be torsion it would be defined over an unramified extension.
\qed

\smallskip

One may also show (using the Shioda-Tate formula or usual descent) that, provided $c^2\neq 0,1$,  the Mordell-Weil group of the Legendre curve over the field $\C(\lambda,\sqrt{1-c^2\lambda})=\C(\sqrt{1-s})$ has rank $1$ and    is generated modulo the $2$-torsion  by the Billiard section (or the Masser section). 
\end{rem}

 \subsection{Real points on $\E_s$ read on $\L_s$}   Let us fix a real $s$, $0<s<1$ and $s\neq c^2$. The curve $\E_s$ is then of genus $1$, and the real points $\E_s(\R)$ have two connected components.  Recall also that there is a group of four automorphisms of $\E_s$, without fixed points, acting on $\E_s(\R)$.

To describe the real points of $\E_s$,  let us distinguish between the cases when $\E_s$ is an ellipse and a hyperbola.

{\bf Case 1}: $C_s$ is an ellipse, i.e. $c^2<s$.  

Each real point is represented  by a pair $(p,l)$ where $p\in C(\R)\cap l$ and where $l$ is a line through $p$,  tangent to $C_s$. There are exactly two such  lines, corresponding to two points $(p,v)$ in the phase space; we cannot choose algebraically among these two lines, but over the reals we can select the line such that $v$  forms a smaller angle 
with $v_p$, where $v_p$ is the vector at $p$  tangent to $C$ and going in the clockwise direction. Alternatively, we may prescribe that the caustic lies on the right of $v$.  This choice yields the two real components.  Denoting, for given $v$, by $p'$  the other intersection of $l$ with $C$, this choice also yields a self-map of the real points of $C$ (depending on $s$): 
\begin{equation}\label{E.TC}
T_C=T_{C,s}:C(\R)\to C(\R):\qquad T_C(p)=p'.
\end{equation}

Note that if $\pi:\E_s\to C$ is the first projection, we have $T_c(\pi(x))=\pi(T(x))$, where $T$ as above is the billiard map.  
We may equivalently describe this as follows: the billiard map acts on $\E_s(\R)$ and it acts in fact on each of the two components (as follows from the above definitions). Each of the two components projects homeomorphically onto $C$ through $\pi$.  Restricting to one component this yields the map $T_C$ (and there is a similar map looking at the other component). 

We now want to locate the image of  the real points on $\E_s$ under the isomorphism $\psi:\E_s\to \L_s$. For this we recall equation \eqref{E.zeta} and the fact that, in the present Case 1, $x_0$ is real whereas $y_0$ is purely imaginary. Recall also that $z:C\to \P_1$ is an isomorphism defined over $\R$.  It follows that  $\zeta$ sends $\E_s(\R)$ precisely to the points in $\L_s$ having $X$-coordinate of modulus $|x_0|=\sqrt\lambda$.  \footnote{Note that we can see {\it a priori} that these points form two connected components, exchanged by the map $-1$ on $\L_s$, since $1<|x_0|<\lambda$ so   $\sqrt{u(u-1)(u-\lambda)}$ is well defined as $u$ travels through a circle of radius $|x_0|$.}  It is remarkable that these points form a group, and that their set is stabilised by translations by points of order $2$ and by translation by the billiard map. The interested reader will verify that on the torus picture for $\L_s$  these points are located, modulo the lattice, on the lines 
$\pm (\omega_1/4) +\R\omega_2$; this also corresponds to an identity $\wp(\mu)\wp(\mu+(\omega_1)/2)=|x_0|^2$, where now $\wp$ is the `translated' Weierstrass function for the $X$-coordinate in the Legendre form.  (That the points are located on horizontal lines follows at once on observing that  translation by the billiard map preserves these points.) 

\medskip

{\bf Case 2}: $C_s$ is a hyperbola, i.e. $s<c^2$.

Things are somewhat different now. First, $x_0,y_0$ are both real, and $P_1,\ldots,P_4$ are real points, dividing $C(\R)$ into four parts. Only the northern and southern part are relevant for the billiard shots having such a caustic.  These two parts lift  (through $\pi$) to two connected components of $\E_s$ and the billiard map now switches these components.   The shots which point right or left now do not give rise to distinct connected components: this would happen if we considered only the affine caustic, whereas in its projective completion the asymptotes constitute points which attach these components. 

Also, $\E_s(\R)$ now corresponds to the points in $\L_s$ with real abscissa, and we have already seen how they are represented in the torus picture.

\subsection{An invariant measure}  Given a caustic $C_s$, $s=c^2\lambda$, and assuming $s$ of good reduction, we have an associated elliptic curve $\L_s=L_\lambda$, corresponding to a complex torus $\Lambda_\lambda$. On this torus the complex differential $\d\mu $ induces a measure invariant by translation, that in turn induces an invariant  measure on the real points. We want to transport this last measure  to the ellipse $C$, and we can do this through the map $\zeta\circ z\circ p$ computed above. Up to a constant, the measure on $L_\lambda$  is $\d X/Y$, which equals 
$\d\zeta/\eta$, where $\eta^2=\zeta(\zeta-1)(\zeta-\lambda)$.  We can compute this from the above formulae \eqref{E.zeta}, \eqref{E.w} and find the measure, up to a constant (depending on $s$, i.e. on the caustic), given by
\begin{equation}\label{E.meas}
{\d z\over \sqrt{(z-z_1)(z-z_2)(z-z_3)(z-z_4)}}={\d z\over \sqrt{(z^2-{y_0^2\over (x_0+1)^2})(z^2-{y_0^2\over (x_0-1)^2})}}.
\end{equation}
This measure is invariant by the billiard map on $C$ (relative to a given caustic $C_s$).  Note that the expression (under square root) on the right involves only real numbers, and in fact one may check it is of constant sign for the relevant values of $z$: if the caustic is an ellipse then this is automatically positive,  for $y_0^2$ is real negative and $x_0$ is real; if the caustic is a hyperbola, this is of constant (negative) sign on the part of the ellipse in between the two branches of the hyperbola. Thus we may choose a sign so  that the measure is continuous. Also, the measure of $C$ becomes finite and nonzero, so that one may normalise   so that $C$ has measure $1$ (where in the case of a hyperbolic caustic we restrict to the part of $C$ mentioned above, i.e. the part touched by any billiard shot with that caustic).

\subsection{Other sections} Say that we have chosen a point $p_0\in\T$. Then a choice of a (unit) vector $v$ determines a billiard shot $(p_0,v)$. In turn, this corresponds to a caustic. Hence we have a map which associates to $v$ a caustic, and therefore a curve in the elliptic family. So we obtain a section of our family, now viewed over  the base which is  the space of possible vectors $v$. Of course if we look at the {\it real} billiard, this $v$ will be restricted to the unit circle, but it will be convenient to work with an algebraic base, for instance $\P_1$, regarded as the set of complex slopes of $v$.  Also, a choice of the direction will correspond to a quadratic equation, so if we want this to be well-defined in the algebraic sense, our  space will be in fact a quadratic cover of the former  base.
 
 Here is some more detail for this. Denote  $p_0=(a,b)\in\T$ and say  that our  billiard shot $(p_0,v)$ has slope $\xi$. 
Then the line $l: y=\xi(x-a)+b$ will be tangent to the relevant caustic.\footnote{The intersection with the caustic will occur at infinity in case this line  passes through the origin.}  By  the formulae above  the corresponding caustic shall be  $C_s$ where 
\begin{equation}\label{E.ssigma}
s ={c^2+(\xi a-b)^2\over \xi^2+1},\qquad (s -a^2)\xi^2+2ab\xi+(s -b^2-c^2)=0.
\end{equation} 

To express the coordinates of the points $p,p'$ of intersection of the line $l$ with $C$, note that the equation for $C$ combined with $y=\xi(x-a)+b$, gives $(1-c^2)x^2+(\xi(x-a)+b)^2=1-c^2$, so  the abscissas of $p,p'$ are given by
\begin{equation}\label{E.PP'}
(1-c^2+\xi^2)x^2+2\xi(b-\xi a)x+(b-\xi a)^2+c^2-1=0.
\end{equation}
which has discriminant given by $4(1-c^2)(\xi^2-(\xi a-b)^2+1-c^2)=4(1-c^2)((1-a^2)\xi^2+2ab\xi+1-c^2-b^2)=4(1-c^2)(1-s)(\xi^2+1)$.

Choosing  a solution  determines   the first point  where the  billiard trajectory meets $C$. This is not defined over $\C(\xi)$ but over  a quadratic extension  expressed by the equation. We obtain a section from the space of slopes to  our $\E$, actually defined in fact only over a quartic extension of $K_1(s)$, where $K_1$ may be taken as $K(a,b)$.  We shall explore this field extension in more detail during the proofs; we note that its degree (over $K_1(s)$) being equal to four is due to the fact that given an interior point $p$ and a caustic $C_s$, there are in general four shots from $p$ giving rise to a tangent to $C_s$: indeed, the tangent can be chosen in  two ways and for each tangent there is still a choice for the orientation of the shot.

\medskip

Note that these sections will give rise to sections of the Legendre model \eqref{E.leg}, however the field of definition will increase, so in practice the base will become a corresponding cover of $\P_1$.  On composing with the billiard map, we shall obtain sections (one for each point $p_0$), as functions of the {\it slopes} rather than the caustic (quantities that of course are related through \eqref{E.ssigma}).

\subsection{The phase space as an elliptic surface}\label{SS.phase-space}
  We present now an alternative way of obtaining an elliptic scheme, and actually a compactification (elliptic surface), this time  without extending the ground field (but extending scalars from $\R$ to $\C$). More details are given in Chapter 9 of our forthcoming book \cite{CZPonc} (but the essentials of the proofs of this paper do not require them).

We can identify the phase space using pairs $(p_1,p_2)\in C\times C$ instead of the pairs $(p,v)$ where $p\in C$ and $v$ is a (unit) vector: of course, $v$ is obtained from the pair $(p_1,p_2)$ by taking the normalized vector connecting $p_1$ to $p_2$. Whenever $p_1=p_2$, such a vector is the tangent vector. 
We obtain a rational map
\begin{equation*}
C\times C \dashrightarrow \P_1
\end{equation*}
associating to every pair $(p_1,p_2)$ the unique caustic tangent to the line connecting $p_1$ to $p_2$; the uniqueness  of this caustic follows from the linearity in $s$ of equation \eqref{E.dual} (in geometric terms: the fact that the dual family of caustics is a {\it pencil}  of conics).

The fibers of the projection are curves of bidegree $(2,2)$ on $C\times C\simeq \P_1\times \P_1$, so, whenever they are smooth, they have genus $1$. 

The above map is undefined at four points, namely the points $(p,p)$ where $p$ is one of the four complex points
\begin{equation}\label{E.4points}
\left( \pm \frac{1}{c}, \pm i\, \frac{1-c^2}{c}\right).
\end{equation}
Indeed, for each such point $p$, the tangent at $p$ to the billiard $C$ is tangent also to every other caustic $C_s$ (at points depending on $s$).

Blowing up these four points, all situated  on the diagonal of $C\times C$, we obtain another (smooth projective) surface, containing four distinguished exceptional divisors. The projection to $\P_1$ is still undefined at four new points, one on each of the four exceptional divisors. 
Blowing them up again, we obtain a new smooth projective surface $\mathcal{X}$, endowed with a well defined projection to the line, fitting in the diagram
\begin{equation}\label{E.diagramma}
\begin{tikzcd}
& \mathcal{X} \arrow[dl,swap] \arrow[dr] \\
C\times C &\dasharrow    & \P_1
\end{tikzcd}
 \end{equation} 
 The four exceptional divisors produced in the last blowing-up procedure are images of sections 
 $\P_1\to\mathcal{X}$. One can take one of them to be the zero section; then the other three are sections of order two. 
 
 Let us describe more concretely these sections. First note that on $C\times C$ one can replace  pairs $(p_1,p_2)$ by pairs $(p,l)$, where $p=p_1$ and $l$ is the line joining $p_1$ to $p_2$. While working with the four points \eqref{E.4points} which have been blown up on $C\times C$, it is better to use this second point of view. 
 Letting $p_0$ be one of the points \eqref{E.4points} of $C$, and   $l_0$ be the line through $p_0$ which is a common tangent to all the caustics, the zero-section can be described as the section associating to each caustic the pair $(p_0,l_0)$. 
 
 Note that by the adjunction formula the self-intersection of the zero section is $-1$, and this must be true of all the (images of) sections. (Indeed the canonical divisor on an elliptic surface lies in the class of a multiple of a fiber, in our  case the negative of a fiber, so has constant intersection product with every section; this proves that in the adjunction formula $2g(\Theta)-2= \Theta^2+K\cdot \Theta$, where $\Theta$ is the image of a section, $g(\Theta)$ its genus, i.e. the genus of the base, and $K$ a canonical divisor on the surface, the term $K\cdot \Theta$ is independent of $\Theta$ and so also the self-intersection of $\Theta$ must be constant.)
 
 \smallskip
 
The map $P\mapsto -P$ on $\mathcal{X}$, deriving from the group law on the fibers, corresponds on $C\times C$ to permuting the coordinates; in other words, to inverting the direction of the motion; this is the same as in the Legendre model. 

Translation by elements of order two is viewed in $C\times C$ as the effect of applying the symmetries $(x,y)\mapsto (\pm x,\pm y)$, where $x,y$ are the coordinates in the affine plane containing the ellipse $C$ of equation \eqref{E.C}. 
 
 \medskip

The fiber of $s=1$ in $C\times C$ is the diagonal, counted with multiplicity $2$ (recall that the fibers have bidegree $(2,2)$). The four blown-up points belong to this diagonal. The fiber of the same point $s=1$ viewed on the surface $\mathcal{X}$ is a five component divisor of the form $F=2\Phi+\Phi_0+\Phi_1+\Phi_2+\Phi_3$, where $\Phi$ is the strict transform in $\mathcal{X}$ of the diagonal of $C\times C$ and $\Phi_0,\ldots,\Phi_3$ are the strict transforms of the exceptional divisors occuring in the first blow-up. This configuration of divisors is named $I_0^*$ in Kodaira's classification \cite{Kod}.

The other singular fibers occur over the points $s=0,c^2,\infty$ and each of them consists in the union of two rational curves intersecting at two points, thus forming a configuration of type $I_2$ in the mentioned  Kodaira's classification.

Let us analyze in detail the case $s=c^2$. Recall that for $s$ real in the interval $0<s<c^2$, the caustics are hyperbolae, while for $c^2<s<1$ the caustics are ellipses. In the middle case $s=c^2$, the `caustic' defined by equation \eqref{E.caustic} becomes the horizontal line $y=0$. In this case, the billiard game consists in sending the ball alternatively to the two foci $(-c,0)$ and $(c,0)$. The genus one curve $\E_s\subset C\times C$ degenerates in the two rational curves consisting of pairs $(p_1,p_2)\in C\times C$ such that $p_1,(-c,0),p_2$ (resp. $p_1,(c,0),p_2$) are collinear.

To see the shape of the two remaining singular fibers, namely $s=\infty$ and $s=0$, note that the four lines which are common tangents to all the caustics, i.e. those drawn from the points \eqref{E.4points},  intersect in six points, naturally coupled in three pairs of points: one such pair is the pair of foci; the remaining two pairs (consisting in complex non-real points) give rise to two more pairs of rational curves defined in exactly  the same way.  

\smallskip 

The billiard map $\kappa:\P_1\to\mathcal{X}$ is a well defined section, generating  the Mordell-Weil group up to torsion (indeed, $2$-torsion). The image $-\kappa(\P_1)\subset \mathcal{X}$ can be described as follows: in the phase space consisting of pairs $(p,l)$ where $p$ is a point on the ellipse and $l$  a line containing $p$ as above,  the section is the curve $(p_0,l(s))$, $s\in \P_1$, where for each caustic $s$ the line $l(s)$ is the second tangent to $C_s$ drawn from $p_0$ (the first one beeing the common tangent $l_0$).

\smallskip

\subsubsection{The elliptic surface associated to an interior point}\label{SSS.ell-surface}
Take now a point $p$ in the interior of the billiard, not on the $x$-axis nor in the $y$-axis (in particular, not a focus). As already remarked, the algebraic curve $Y_p\subset C\times C \approx \mathcal{X}$ made of pairs $(p_1,p_2)\in C\times C$  with collinear $p_1,p_2,p$  is not a (rational) section, but rather a multi-section, precisely a degree four algebraic section: indeed, for a generic caustic there are two possible choices for the tangents drawn from $p$ and for each tangent, two possible directions.
To obtain a rational section from this curve, we must perform a   degree four (ramified) cover of the base. This cover can be described as follows:
recalling that $Y_p$ is sent by a degree four morphism to the curve parametrizing caustics (isomorphic to $\P_1$), we can form the fiber product
\begin{equation*}
\begin{matrix}
\mathcal{X}_p &\dashrightarrow  &\mathcal{X} \\
\downarrow &{}& \downarrow\\
 Y_p&\longrightarrow &\P_1
\end{matrix}
\end{equation*}
 The top arrow is not a morphism, because in taking a relatively minimal model for the elliptic surface $\mathcal{X}_p\to \mathcal{X}$ we have to contract the four pre-images of the curves $\Phi_0,\ldots,\Phi_3$. It turns out that now the fiber of $s=1$ becomes smooth, as it is the case for Legendre's model. In other terms, $s=1$ is a place of potentially good reduction, although of bad reduction, for the fibration $\mathcal{X}\to \P_1$. 
 
 \smallskip

The elliptic surface  can be birationally described as the set of quadruples $(p_1,p_2,q_1,q_2)\in C^4$ where $q_1,q_2,p$ are aligned and the line containinig them is tangent to the same caustic which is tangent to the line joining $p_1,p_2$. The projection to $\mathcal{X}\ni (p_1,p_2,q_1,q_2)\mapsto (q_1,q_2)\in L_p$ provides the elliptic fibration. 

By construction, this new elliptic surface $\mathcal{X}_p\to Y_p\simeq \P_1$ has one more section, say $\sigma_p$, associating to every point $(q_1,q_2)\in Y_p$ of the base  (recall that the caustics are now parametrized $4$ to $1$ by $Y_p$, i.e. every caustic appears four times)  the quadruple $(q_1,q_2,q_1,q_2)\in C\times C$. 

Although this is not obvious, the elliptic surface admits also  another section, denoted by $\tilde{\sigma}_p:\, L_p\to \mathcal{X}_p$, ``changing the tangent'' with respect to $\sigma_p$: this  follows from the fact that the covering $L_p\to \P_1$ is Galois, as explained in \S 9.5.3 of our book \cite{CZPonc} (see in particular  Theorem 9.5.6 therein).

From  the two sections $\sigma_p,\tilde{\sigma}_p$ one can produce two more sections, interchanging $p_1$ with $p_2$; these new sections just coincide with   $-\sigma_p$ and $-\tilde{\sigma}_p$, with respect to the group law just defined.
 
The  surface $\mathcal{X}_p$ is a K3-surface. Its Picard group may be checked to be  of rank $17$; the fibration $\mathcal{X}_p\to L_p$  possesses $12$ singular fibers, so by Shioda-Tate's formula its rank is $3$. Its Mordell-Weil group is generated up to torsion by the sections $\sigma_p, \tilde{\sigma}_p$ and the billiard section $\kappa$ (see Theorem 9.5.11 in \cite{CZPonc}).

Note that by changing the excentricity of the ellipse (i.e. the parameter $c$) and the position of the point $p$ we obtain a three-dimensional family of $K3$ surfaces of rank $\geq 17$. From the general theory of $K3$ surfaces (see e.g. Griffiths and Harris book \cite{GH}, chapter 4, page 590), this family is a full irreducible component in the relevant moduli space.

\subsection{Auxiliary results} In this subsection we list a few results used in  our proofs. Some of them  may be considered as special cases of the {\it Pink-Zilber conjectures} (see \cite{Z} and  \cite{Z3}) whereas the last one is taken from the theory of $S$-unit equations (see \cite{B-G}).  We believe that this recall may be helpful for some readers; some of  these theorems go back to some time ago, whereas others are more recent.

Our first result is taken from the paper \cite{CMZ} of  D. Masser with both of the authors. (As recalled in \cite{Z3}, it was first formulated by Shou-Wu Zhang as a question.) 

\begin{thm}\label{CMZ}  [\cite{CMZ}, Thm. 1.2] Let $\mathcal A\to {\mathcal C}$ be an abelian-surface scheme over a complex (affine) curve $\mathcal C$ and let $\sigma:\mathcal C\to\mathcal A$ be a section
whose image is not contained in any proper group subscheme. Then there are only finitely many points $x\in\mathcal C$ such that $\sigma(x)$ is torsion on the fiber $\mathcal A_x$.
\end{thm}

Previous results concerned only the case of $\overline\Q$ as a ground field. After intermediate progress,  this result, if we limit to the ground field $\overline\Q$, was recently extended by F. Barroero and L. Capuano, to cover not merely torsion points but also linear relations. They prove in \cite{B-C}, Thm. 1.1, a result which immediately implies the following:

\begin{thm} \label{B-C} [\cite{B-C}, Thm. 1.2] Let $\mathcal A\to {\mathcal C}$ be an abelian  scheme over a(n affine) curve $\mathcal C$ defined over a number field and let $\sigma:\mathcal C\to\mathcal A$ be a section
whose image is not contained in any proper group subscheme. Then  the intersection of $\sigma({\mathcal C})$ with the  union of all subgroup schemes of  $\mathcal A$ of codimension $\ge 2$ is a finite set.
\end{thm}

By {\it subgroup scheme} we tacitly mean that it is surjective onto  the base $\mathcal C$. Contrary to the previous result, here the field of definition is assumed to be $\overline\Q$. It is probable that the methods of \cite{CMZ} (or other methods) allow to replace this with $\C$, but this has not yet been formally proved. So in applying this resul we shall tacitly work over $\overline\Q$. 

It is worth mentioning that before this general theorem, a similar but  weaker conclusion (together with  other results immaterial here) was proved (with partially independent methods) as Theorem 1.1 of the paper \cite{GHT} of Ghioca, L. Hsia and Tucker: this assumed  $\Aa$ equal to a (fiber) product of two elliptic schemes,  and moreover  restricting to certain special subgroup  schemes,  but would be sufficient for our applications  here (though probably not for generalisations of them). 

\medskip

The two previous results have their origin in the following one, previously a conjecture by Lang, proved in  the sixties by Ihara, Serre and Tate (independently). 

 \begin{thm}\label{T.torsion}
 Let $C\subset\G_{\mathrm m}^2$ be an irreducible algebraic curve and $\Gamma\subset\G_{\mathrm m}^2(\C)$ be the torsion subgroup of the torus. If $C\cap \Gamma$ is infinite, then $C$ is a translate of a subtorus by a torsion point.
 \end{thm}
 
Viceversa, it is clear that if $C$ is a subtorus, or more generally a torsion translate of a subtorus, then $C$ contains infinitely many points with coordinates in $\Gamma$.

The above theorem was generalized to curves in abelian varieties, and later to subvarieties of semi-abelian varieties. Theorem \ref{CMZ} can be viewed as the `relative' version of the generalization of Theorem \ref{T.torsion} to abelian surfaces.
\medskip

We state one further result in this context, due to M. Laurent and useful for the proofs of  Theorem \ref{T.P_2} and for some results on circular billiard.. It arose from results of Siegel, Mahler, Lang, Schmidt, together with the previous ones on torsion points. It solves  the   {\it Mordell-Lang conjecture for algebraic tori}:

\begin{thm}\label{T.tori}
Let $\Gamma\subset \G_{\mathrm{m}}^n(\C)$ be a multiplicative group of finite $\Q$-rank. Let $\Sigma\subset \Gamma$ be any subset. The Zariski-closure of $\Sigma$ in $\G_{\rm m}^n$ is a finite union of translates of algebraic subgroups.
\end{thm}
\smallskip

An important instance  of the above theorem is represented by the special case of a finitely generated group $\Gamma$. This is the case  appearing in the proof of Theorem \ref{T.P_2}. The case when $\Gamma$ coincides with  the torsion subgroup generalizes Theorem \ref{T.torsion} to higher dimensions and is used in the proof of Theorem \ref{T.buca} for circular billiards. This torsion case was proved independently by Sarnak-Adams \cite{SA}.
\smallskip

Our last auxiliary result is due to E. Bombieri, D. Masser and the second author:

\begin{thm}\label{T.BMZ}
Let $C\subset\G_{\mathrm{m}}^n$ be an irreducible algebraic curve. Suppose it is not contained in any translate of an algebraic subgroup of the torus. Then the union of the sets of the form $C\cap H$, where $H\subset \G_{\mathrm{m}}^n$ is a codimension $2$ subgroup is finite.
\end{thm}

 This result was improved by Maurin \cite{Maurin}, who replaced the hypothesis on $C$ by the weaker (and optimal) one that $C$ is not contained in any {\it torsion} subgroup. Se also \cite{BHMZ} by  the authors of \cite{BMZ} with Ph. Habegger for a different proof. 
 
 The crucial case of the above theorem, and the one which is needed in the present work, is the case of a curve in $\G_{\mathrm{m}}^3$: the theorem can be rephrased by saying that given three rational functions on any curve, multiplicatively independent modulo constants, the set of points on the curve where the values of the three functions satisfy two independent multiplicative dependence relations is finite.

 \medskip

We spend just a few words on the proofs of these results. The first two depend on an analytic description of the abelian varieties which appear, namely as complex tori (which vary). Correspondingly, expressing  the values of the sections as linear combinations  in a basis of periods for the tori, gives rise to coefficients which are real-valued functions; these are called ``Betti coordinates". 
Now, the relevant relations correspond to relations with integer coefficients among these coordinates; in turn,  one uses counting theorems for rational points in transcendental varieties  to prove that  if these relations hold for values at points of large degree, then they must come from a geometric relation. (See \cite{Z3} for much more on this.) 

The proof of Theorem \ref{T.torsion} is more elementary and still uses, albeit in a simpler way, the Galois action on torsion points. This theorem will be needed to treat circular billiards. 

The last theorem in the above list, Theorem \ref{T.BMZ}, is proved by combining a height estimate for points on a curve satisfying  multiplicative dependence relations with lower bounds for the Mahler measure of algebraic points in tori.
\smallskip

On the other hand, as we already mentioned, the proof of Theorem \ref{T.tori}  needs the Schmidt Subspace Theorem in Diophantine Approximation; this is a deep result, but of rather different nature compared to the former theorems. See  the book \cite{B-G} by E. Bombieri and W. Gubler for a proof, and see \cite{Z3}  for a description of related results and evolutions. 

So, the context shows the peculiarity that completely analogous statements admit completely different treatments.

\medskip

One can ask if it is possible to compute the relevant solutions, when they are finite in number. Now, the Subspace Theorem is presently ineffective and there is to date little hope to obtain an effective proof.  The results alluded to for the former theorems were ineffective as well, but conceptually the obstacles to effectivity were considered of more moderate nature compared to the Subspace Theorem. Indeed, Binyamini \cite{Bin} recently found many effective proofs in this realm.  It appears that these should allow to make effective some of the present results. Concerninig Theorem \ref{T.BMZ}, to be used only for circular billiard, at present it is still ineffective.


\section{Proofs}\label{S.proofs}

 For the proof of  Theorem \ref{T.esiste} we refer to the Appendix, where we shall add several other discussions and conclusions.

\subsection{Proof of Theorem \ref{T.angolo}} 
 This time we have two billiard shots $(p_0,v)$, $(p_0,v')$, such that $v,v'$ (in some order)  form an angle $\alpha\in (0,\pi)$ given in advance. Also, we denote as above $p_0=(a,b)$ where we suppose to be in a {\it real} billiard, so that $a,b$ are  real numbers  (a restriction which can be eliminated).  If  $\xi,\xi'$ are the respective slopes of $v,v'$, then, setting $t_0:=\tan\alpha\in (-\infty, +\infty)$, we have
\begin{equation}\label{E.tan}
\xi'={\xi+t_0\over 1-t_0\xi}=:g(\xi),
\end{equation}
and $\xi'=\frac{-1}{\xi}$ if $t_0=\infty$,
where for this proof we denote by $g\in \mathrm{PGL}_2$ the homography defining $\xi'$. Note that it has the fixed points $\pm i$.

As before, each billiard shot corresponds to a section of $\L$ over a base whose function field is a quadratic extension of $\C(\xi)$ resp. $\C(\xi')$ (we do not mind extending to $\C$ the constants in this case).  We may view the pair of shots as giving a section of the fiber product of $\L$ with  another copy of $\L$ over the $\xi$-line, with respect to the map defined by \eqref{E.tan}.  That is, the second copy equals the first, however with $\xi'$ in place of $\xi$.  Hence we obtain an abelian  scheme $\Aa$ over a finite cover $B$ of the $\xi$-line, where the fibers are products of two elliptic curves, in short $\Aa=\Aa_1\times_B\Aa_2$, where $\Aa_i$ are elliptic schemes. 

Recall now that each scheme is of Legendre type, with parameter $\lambda=s/c^2$ (where $C_s$ is the relevant caustic), and that $s$ is given in terms of the slope $\xi$ by equation \eqref{E.ssigma}, i.e. 
\begin{equation*}
s={c^2+(\xi a-b)^2\over \xi^2+1}=: R(\xi),
\end{equation*}
 where, again for this proof, we denote by $R$ the present  rational function of degree $2$ expressing $s$ in terms of $\xi$.

We start by proving that the two schemes obtained as above are not (generically and geometrically)  isogenous, that is, the corresponding elliptic curves do not become isogenous over any extension of $\C(s)$. In principle there are several methods for checking this, and we choose the following one. If the curves were isogenous, their $j$-invariants $J,J'$ would satisfy some modular equation $\Phi_n(J,J')=0$. Recall that such equations are over $\Z$, irreducible over $\C$, symmetric  and monic in both variables. Therefore $J'$ is integral over $\Z[J]$ and conversely. In particular, if $J,J'$ are rational functions on a certain complete smooth curve, they have exactly  the same poles. 

Now, the $j$-invariant of the Legendre  elliptic curve with parameter $\lambda$ is given by 
\begin{equation*}
j(\lambda)=1728{(\lambda^2-\lambda+1)^3\over \lambda^2(1-\lambda)^2}.
\end{equation*}
So this has poles of order $2$ at $\lambda=0,1,\infty$. If we take $\lambda=s/c^2$, this corresponds to 
$s=0,c^2,\infty$. In turn, if we use the above formula relating $s$ and $\xi$ this corresponds respectively  to 
\begin{equation*}
\xi= {b\pm ic\over a}, \qquad {\pm b\over c\pm a},\qquad \pm i,
\end{equation*}
at any rate for $a\neq 0, \pm c$, which we suppose for the moment. 
Consider now the image $\xi'$ of these points under the map $g$ appearing in   \eqref{E.tan}. The last two points are fixed by $g$, while under our present assumption the whole set of these points has to be stabilized by $g$. So the set of the first  four points has to be stabilized by $g$.  Since however $g$ is defined over $\R$ in fact it has to send the set of the first (resp. second)  two points into itself. Since $g$ has only $\pm  i$ as fixed points, $g^2$ must then fix all   points, and then it has to be the identity,  forcing $t_0=\infty$ (since $t_0=0$, corresponding to  $\alpha=0$, is excluded.  But then $g(\xi)=-\xi^{-1}$ and from $g((b+ic)/a)=g((b-ic)/a)$ we get $ b^2+c^2=-a^2$ which is impossible.

It remains to consider the cases $a=0$ and $a=c$.  

If $a=0$ the first two points are replaced by $\infty$ and  we obtain that the set $\{\infty,b/c,-b/c\}$ is acted on by the automorphism $g$, which, we recall, fixes $\pm i$. This is possible if only if $t_0=\pm \sqrt{3}$, i.e. $\alpha=\pm \pi/3$,  $p=(0,c/\sqrt{3})$ and the three point set is $\{\infty,1/\sqrt{3},-1/\sqrt{3}\}$.   
In this case, we have
\begin{equation*}
\lambda(\xi)=\frac{4 }{3(\xi^2+1)},\qquad \lambda(\xi')=\frac{3\xi'^2+2\sqrt{3}\xi'+1}{3(\xi'^2+1)}.
\end{equation*}
From these relations it follows that, although the two functions $j(\lambda(\xi))$ and $j(\lambda(\xi'))$ have the same pole set, the corresponding multiplicities do not coincide: for instance $\xi=\infty$ is a pole of order $4$ for $j(\lambda(\xi))$ while $\xi'=\infty$ is a double pole for $j(\lambda(\xi'))$. Since the modular polynomials are symmetric, the two functions $j(\lambda(\xi))$ and $j(\lambda(\xi'))$, for $\xi'=g(\xi)$, cannot be related by such an equation.

In the last case to consider, namely $a=c$, an argument of the same type as above shows that $g$ is an involution, so $\alpha$ is a right angle and $p$ lies in a focus. Then the assertion of the theorem is then trivial, since, as we previously noticed, there is only one periodic trajectory passing through the foci.

We note that the whole argument we have used can be rephrased in terms of bad reduction: unless $p$ lies in a focus, there always exist an angle $\gamma$ giving rise to bad reduction while $\gamma+\alpha$ corresponds to a non-degenerate caustic.

\medskip

This proves that the two elliptic schemes are not isogenous (up to one possible exception, when the theorem is already proved). But then the present scheme $\Aa$ has no proper group subschemes other than product or torsion subschemes, and any such subscheme projects to a torsion subscheme on at least one component.   But then,  if our product section has image contained into one of these, then one of the sections would be torsion. But the billiard section is not torsion, as we have shown  in Proposition \ref{P.nontors}.  Then we may apply Theorem \ref{CMZ}, which concludes the argument. 

\bigskip

It remains to treat the case of circular billiards, where both the geometric and the diophantine tools are very different. 

We can suppose the circle is defined by the equation
\begin{equation*}
x^2+y^2=1
\end{equation*}
and the point $p_0$ from which the ball is shot has coordinates $(u,0)$, with $-1<u<1$, $u\neq 0$. 
The line containing the first segment of the trajectory has an equation of the form
\begin{equation*}
\frac{x}{u}+by=1
\end{equation*}
for some $b\in \R$. Let $C_b$ the only circle centered at the origin which is tangent to that line. The billiard trajectory associated to $b$ will be periodic if and only if the angle formed by the tangent to $C_b$ drawn from a point on the border of the billiard with the diameter is commensurable with $\pi$. We can consider the point $(-1,0)$, so that this tangent line will have an equation of the form
\begin{equation*}
 -x+b' y=1.
\end{equation*}
 The fact that this two lines are tangent to a same circle centred in the origin amounts to the quadratic relation
 \begin{equation}\label{E.rel-b,b'}
 u^2+b^2=1+b'^2.
 \end{equation}
Letting $\gamma$ be the angle formed by the first line with the horizontal diameter, and $\beta$ the angle formed by the same diameter with the second line,
we have
\begin{equation*}
\tan(\gamma)=(bu)^{-1},\qquad \tan(\beta)=b^{-1}.
\end{equation*}
 The relation \eqref{E.rel-b,b'} is easily seen to be equivalent to the relation
 \begin{equation*}
 u^2\sin^2\gamma=\sin^2 \beta.
 \end{equation*}
By obvious symmetries, we can suppose that $u\sin\gamma=\sin \beta$. Now, let us suppose that the shots of angles $\gamma$ and $\gamma+\alpha$, where $\alpha$ is fixed, both give rise to periodic orbits. We then obtain that $u\sin\gamma=\sin\beta$ and $u\sin(\gamma+\alpha)=\sin\beta'$   for two angles  $\beta,\beta'$ which are commensurable with $\pi$. Writing 
\begin{equation*}
\sin\gamma=\frac{t-t^{-1}}{2i},\qquad \sin(\gamma+\alpha)=\frac{te^{i\alpha}-t^{-1}e^{-i\alpha}}{2i}
\end{equation*}
\bigskip
for a suitable complex number $t$ in the unit circle, and correspondingly for the  sines of $\beta,\beta'$, we arrive at the system of algebraic equations
\begin{equation}\label{E.system-G_m^3}
\left\{\begin{matrix}
u\cdot (t-t^{-1})&=&\xi-\xi^{-1}\\
u\cdot (te^{i\alpha}-t^{-1}e^{-i\alpha})&=&\eta-\eta^{-1}\end{matrix}\right.
\end{equation}
to be solved in $(t,\xi,\eta)\in \G_{\mathrm m}^3$ where $\xi,\eta$ are roots of unity.
We let $\mathcal{X}\subset \G_{\mathrm m}^3$ be the algebraic curve defined by the above system, $\pi:\, \G_{\mathrm m}^3\to \G_{\mathrm m}^2$ be the projection to the $(\xi,\eta)$ coordinates and $\mathcal{Y}=\pi(\mathcal{X})$ the projected algebraic curve. We must prove that $\mathcal{Y}$ cannot contain infinitely many torsion points.

We shall first verify  that if $e^{i\alpha}\neq \pm 1$, which we have supposed,  $\mathcal{X}$ and $\mathcal{Y}$ are isomorphic under $\pi$. We will then check that if $u\neq 0,\pm 1$, which we are supposing, $\mathcal{X}$ (and so $\mathcal{Y}$) is not a rational curve. This will allow the application of Theorem \ref{T.torsion} completing the proof.\smallskip

For the first assertion, we must check that for every $(\xi,\eta)\in\mathcal{Y}=\pi(\mathcal{X})$ there is just a point $t$ satisfying both equations in the system \eqref{E.system-G_m^3}. Indeed, these equations are quadratic in $t$ and the product of the two solutions to the first equation is $-1$, while the product of the solutions to the second equation is $e^{-2i\alpha}$, excluding that the two sets of solutions coincide.
\smallskip

As to the second assertion, we can view $\mathcal{X}$ as a fiber product of the two quadratic coverings of the $t$-line given by the individual equations of the system \eqref{E.system-G_m^3}. These equations define (isomorphic) elliptic curves, since the discriminant of the quadratic equation satisfied by $\xi$ (resp. by $\eta$) is $u^2(t-t^{-1})^2+4$ (resp. $u^2(te^{i\alpha}-t^{-1}e^{-i\alpha})^2+4$) which are rational functions on the line with four simple zeroes and two double poles. Hence the corresponding curve is a quadratic cover of the line ramified over four points.
Now, since $\mathcal{X}$ dominates an elliptic curve it cannot be rational. 

\qed

\subsection{Discussion on parallelogram billiards}\label{SS.altri}  We start with proving Theorem \ref{T.angolo-parallelogramma}, i.e. the analogue of Theorem \ref{T.angolo} for parallelogram billiard (e.g. rectangular ones), as promised in Remark \ref{R.altri}.

\begin{proof}[Proof of Theorem \ref{T.angolo-parallelogramma}]  

Let us normalise the lattice by complex dilation (which does not affect the issue), so to  assume that  $L=\Z\tau+\Z$, $Im(\tau)>0$. We shall write elements $\lambda\in L$ as linear combinations $a\tau+b$ of $\tau, 1$ with integer coefficients,  indicating this  with $F_\lambda(\tau)$, and we denote by $F_\lambda(x)$ the  polynomial $ax+b$  in the indeterminate $x$ of degree $\le 1$ having those same coefficients (so $\lambda=F_\lambda(\tau)$ is consistent). 

 The condition of periodicity of a shot with direction $v$ amounts to $v\in \R \cdot L:=\{t\lambda: t\in\R, \lambda\in L\}$ for the direction. So we want that $v\in\R\cdot L$ and also $v'=e^{i\alpha}v\in \R\cdot L$, and here $\alpha\in (0,\pi)$ is given.
 A solution of this amounts to an equation   $e^{i\alpha}=t\lambda/\delta$, where $t\in\R^*$ and $\lambda,\delta\in L-\{0\}$, and where two solutions have to be considered equivalent if the respective $\lambda,\delta$ are  the same up to a factor in $\R$ (which in fact should then lie in $\Q$); this amounts to the directions being the same. 

If $\C/L$ has $CM$ then  the $\Q$-vector space generated by the lattice is a field and, for an infinity of $\alpha$,  we obtain an infinity of inequivalent solutions starting from any single solution: indeed, from a solution $e^{i\alpha}=t\lambda/\delta$ and any $\eta\in L, n\in\N$, we obtain another solution $e^{i\alpha}=t\lambda'/\delta'$, where $\lambda'=n\lambda \eta$, $\delta'=n\delta\eta$ both belong to $L$ for suitable $n>0$. By taking an infinity of  pairwise non-$\Q$-proportional elements $\eta\in L$, we obtain infinitely many pairs of periodic solutions satisying the conditions of Theorem \ref{T.angolo-parallelogramma}.

For the converse assertion, let us then suppose to have four  essentially distinct solutions (i.e. with  non proportional $v$), denoted as $t_j,\lambda_j,\delta_j$, $j=1,2,3,4$. Putting $g_j(x)=F_{\lambda_j}(x)$, $h_j(x)=F_{\delta_j}(x)$.  We have $e^{i\alpha}=t_j\lambda_j/\delta_j$, so
\begin{equation*}
\lambda_i\delta_j=(t_j/t_i)\lambda_j\delta_i=t_{ij}\lambda_j\delta_i,\qquad t_{ij}=t_j/t_i\in\R^*.
\end{equation*}
We may write uniquely $\lambda_i\delta_j=g_i(\tau)h_j(\tau)$ as a quadratic $a_{ij}\tau^2+b_{ij}\tau+c_{ij}$ with integer coefficients in such a way  that the same holds on replacing $\tau$ with $x$.   For $1\leq i<j\leq 4$, set  $\v_{ij}=(a_{ij},b_{ij},c_{ij})$; these are six vectors in $\Q^3$. 

Let $\tau^2+a\tau+b=0$ be the minimal equation of $\tau$ over $\R$; we have to prove that $a,b\in\Q$. 

In any case,  the vector $\u:=(1,a,b)$ is proportional to $\v_{ij}-t_{ij}\v_{ji}$, provided this last vector is nonzero. Now, if $\v_{ij}-t_{ij}\v_{ji}=0$ for some pair $(1\leq i<j\leq 3$ then we have identically $g_i(x)h_j(x)=t_{ij}g_j(x)h_i(x)$.  But $g_k(x),h_k(x)$ cannot be proportional, no matter $k$, since otherwise $e^{i\alpha}$ would be real. So, the equation implies that  $h_j(x)=ch_i(x)$, $g_j(x)=ct_{ij}^{-1}g_i(x)$, for a rational constant $c$, and $t_{ij}$ must also be rational.  But then the two equations for $e^{i\alpha}$ are essentially the same, i.e. obtained just by multiplying the lattice elements by rational constants. We may assume this is not the case, so $\v_{ij}-t_{ij}\v_{ji}\neq 0$ for all $1\leq i <j\leq 3$, and actually the same argument proves that $\v_{ij}$ and $\v_{ji}$ are linearly independent.  This already shows that $\u$ lies in the plane spanned by them, i.e. the rational plane orthogonal to $\v_{ij}\wedge\v_{ji}\neq 0$, so $1,a,b$ are linearly dependent over $\Q$. If there exist two of these planes which are distinct, then $\u$ would lie in their intersection, which is a rational line, so $\u$ must be rational as wanted. So suppose that all these planes are equal; then the polynomials $g_i(x)h_j(x)$, $i\neq j$  generate a vector space of dimension $2$ over $\Q$.  In particular, $g_1(x)h_2(x)$, $g_1(x)h_3(x)$  generate this space, or $h_2(x),h_3(x)$ are proportional. This last assumption is impossible, since otherwise the second and third solutions would be essentially equal. Hence any $g_s(x)h_k(x)$, $s\neq k$, is a linear combination of $g_1(x)h_2(x)$, $g_1(x)h_3(x)$, and hence is a multiple of $g_1(x)$. In particular, $g_1|g_2h_3$ and $g_1|g_2h_4$, so, since   $g_1$ cannot divide $g_2$, for the same reason why $h_s$ cannot be proportional to $h_k$ for $s\neq k$, we must have $g_1|h_3$ and $g_1|h_4$; but this implies that   $h_3$ and $h_4$ are proportional. This is  contradiction  concludes the argument.

\end{proof}

\begin{proof}[Discussion of the assertion in Remark \ref{R.altri} concerning Theorem \ref{T.buca}] For Theorem \ref{T.buca} things are again  elementary, and we only add a few words. It is easy to see that, for instance in a rectangular billiard, taking three points $p_i=(a_i,b_i)$, and setting $\alpha_{ij}=a_i-a_j$, $\beta_{ij}=b_i-b_j$, the trajectories from $p_1$ passing through $p_2$ and $p_3$ correspond to integer solutions $(x,y,z,w)$ of $(\alpha_{13}+x)(\beta_{12}+y)=(\alpha_{12}+z)(\beta_{13}+w)$.  For instance if $p_1,p_2,p_3$ are rational points with common denominator $N$ and if, putting $A_{ij}=N\alpha_{ij}$, $B_{ij}=N\beta_{ij}$, we have $A_{13}B_{12}\equiv A_{12}B_{13}\pmod N$, there are infinitely many integer solutions: it suffices to find `many'  integers $m>0$, $m\equiv A_{13}B_{12}\pmod N$, and  having two divisors congruent modulo $N$ resp. to  $A_{13}$ and $A_{12}$. 

We leave it to the interested readers to discover the exact assumptions that have to be imposed for obtaining an analogue of Theorem \ref{T.buca} (and possibly Theorem \ref{T.ritorno}) for parallelogram billiards. 
\end{proof}

\subsection{Proof of Theorem  \ref{T.buca} and the finiteness part of Theorem \ref{T.ritorno}}
We now prove Theorem \ref{T.buca} distinguishing the  elliptic and the circular cases; the proofs are somewhat different.  The technique used to treat the elliptic case turns out to be useful in the proof of the finiteness part in Theorem \ref{T.ritorno}.

\subsubsection{Proof of Theorem \ref{T.buca} in the elliptic case} Given the three points $p_1,p_2\in \T^o$ and $h\in C$, consider the sections $\sigma_{p_1}, \sigma_{p_2}, \sigma_{h}$ associated to these points, as explained at the end of sub-section \ref{SS.phase-space}.  


Before going on, we recall  that these sections are not well-defined over the base  (parametrising the caustics)  of our elliptic scheme, but $\sigma_{p_1}$ and $\sigma_{p_2}$ are defined over (possibly different)  quartic extensions of the base; these extensions ramify over the points of the base corresponding to the two caustics passing through the relevant point and to the caustic $C$.  On the contrary, the section $\sigma_h$ is defined over a quadratic extension: so, although    there are still four choices  for a shot from $h$ with given caustic,  one can canonically choose the direction of the shot for each tangent (and we shall be interested in the direction pointing to $h$). Hence $\sigma_h$ can be defined over a degree two  extension of the base, which ramifes over the (hyperbolic) caustic passing through $h$ and (again) the caustic $C$. 

We can then define three elliptic schemes $\mathcal{X}_{p_1}, \mathcal{X}_{p_2}, \mathcal{X}_h$, each of them derived from the billiard scheme $\E\to \P_1-\{4\, \mathrm{points}\}$ by base change as explained in the sub-section  \ref{SSS.ell-surface}, each endowed with a new section, namely  $\sigma_{p_1}, \sigma_{p_2},\sigma_h$. These sections associate to each point of the  base of $\mathcal{X}_{p_1}$ (resp. $\mathcal{X}_{p_2}, \mathcal{X}_{h}$)  a `shot' passing through $p_1$ (resp. $p_2$, $h$) 
 We can also let $B$ be the compositum of all the base changes, and define an abelian scheme $\mathcal{A}\to  B$ by taking the fibre products of the ellipic schemes $\mathcal{X}_{p_1}, \mathcal{X}_{p_2}, \mathcal{X}_h$. 

The billiard shots we are considering in Theorem \ref{T.buca}, i.e. those sending $p_1$ to $p_2$ after $m$ bounces and then $p_2$ to $h$ after $n$ further bouncings, correspond to the values of $s\in B$ such that
\begin{equation}\label{E.buca1}
\left\{\begin{matrix}
(\sigma_{p_1} +m\kappa)(s)&=& \sigma_{p_2}(s)\\
(\sigma_{p_2}+n\kappa)(s)&=& \sigma_h(s)
\end{matrix} 
\right.
\end{equation}
where $\kappa$ is the billiard section viewed on $B$. More precisely, if (contrary to what is claimed in Theorem \ref{T.buca}) there existed infinitely many shots from $p_1$ sending the ball to $p_2$ and eventually to the hole $h$, then for a suitable choice of the sections associated to $p_1,p_2,h$, the system of equations \eqref{E.buca1} would have admit infinitely many solutions $s\in B$, $m,n\in\Z$. 

Let us write this system in the form
\begin{equation}\label{E.buca2}
\left\{\begin{matrix}
(\sigma_{p_2}-\sigma_{p_1})(s) &=& m\kappa(s)\\
(\sigma_{h}-\sigma_{p_2})(s) &=& n\kappa(s) 
\end{matrix} 
\right.
\end{equation}
The three sections $\sigma_{p_2}-\sigma_{p_1},\sigma_{p_2}-\sigma_h,\kappa$ give rise to a section $\sigma:B\to \mathcal{A}$ to the three-dimensional abelian scheme just defined. The solutions $s\in B$ to the above system give  rise to points where $\sigma(s)$ is contained in a subgroup scheme of codimension $2$.

Before going on, we note that in the present situation, unlike that of the proof of Theorem \ref{T.angolo}, the four sections $\sigma_{p_1}, \sigma_{p_2}, \sigma_h,\kappa$ we are considering can be algebraically defined over the same scheme $\mathcal{L}\to\P_1$ (the Legendre scheme). Hence in principle there can be linear relations among them (and indeed there are some, in very special cases). So some extra work is needed to exclude linear relations of a certain type, which might prevent an application of Theorem \ref{B-C}. That is, the image of our section could be {\it identically} lie in a subgroup-scheme of codimension $2$, and then of course this would continue to hold for each point $s\in B$.

More precisely, if the three sections $\sigma_{p_2}-\sigma_{p_1},\sigma_{p_2}-\sigma_h,\beta$ are linearly independent,  the curve $\sigma(B)$ is not contained in any proper subgroup scheme of $\mathcal{A}\to B$, and Theorem \ref{B-C} applies, assuring the finiteness of the solutions to \eqref{E.buca2}.

Hence we have to investigate these possible dependencies.

\medskip

We first show that the sought independence  holds {\it generically} and then we shall treat the special cases. 

\medskip

{\tt Claim}. {\it Let $p_1,p_2$ be interior points outside the axes of the ellipse. Unless the two caustics containing $p_1$ coincide with the two caustics containing $p_2$, the four sections $\sigma_{p_1},\sigma_{p_2},\sigma_h,\kappa $ are linearly independent. In particular, the three sections $\sigma_{p_2}-\sigma_{p_1},\sigma_{p_2}-\sigma_h,\kappa$ are linearly independent.}
\medskip

The principle of the proof is the following: if some algebraic sections $\sigma_1,\ldots,\sigma_k$ of an elliptic scheme can be rationally defined on a base which is unramified over a certain place $s_0$ while another section $\sigma_{k+1}$ cannot, then no multiple of this last section can belong to the group generated  by the previous ones. In particular, if $\sigma_1,\ldots,\sigma_k$ are proved to be independent, then also $\sigma_1,\ldots,\sigma_{k+1}$ will be independent.

\smallskip

{\it Proof of the Claim}. 
 Suppose first that the elliptic caustics passing through $p_1$ and $p_2$ are different.
 Recall that the billiard section is defined over a quadratic extension of the base of the Legendre scheme which ramifies only over $s=1$ and $s=\infty$; since the minimal field of definition of $\sigma_h$ ramifies  over the hyperbolic caustic containing $h$, the section $\sigma_h$ cannot be dependent with $\kappa$. 
 Consider now $\sigma_{p_1}$; observe that it is defined on a base which ramifies over the elliptic caustic passing through $p_1$, while  $\kappa,\sigma_h$ can be defined on a base which is unramified over such a caustic. Hence no multiple of $\sigma_{p_1}$ belongs to the subgroup generated by $\sigma_h,\kappa$. For the same reason, looking at the elliptic caustic passing through $p_2$, we deduce that no  multiple of $\sigma_{p_2}$ can be generated by $\kappa,\sigma_h,\sigma_{p_1}$.
 
 If, on the contrary, the two points $p_1,p_2$ are contained in a same elliptic caustic, but not in a same hyperbolic caustic, we argue as follows: certainly at most one of the two hyperbolic caustics containing $p_1,p_2$ can contain $h$; suppose for instance that the hyperbolic caustic containing $p_2$ does not contain $h$ (nor $p_1$). Then we proceed as before to prove that $\kappa,\sigma_h,\sigma_{p_1}$ are linearly independent, and to conclude the argument we use the ramification over this hyperbolic caustic (containing $p_2$ but not $p_1$).

\smallskip 

It remains to treat the special case when $p_1,p_2$  lie at the intersection of the same caustics or one of them lies on the axis; we shall see that in this case $\sigma_{p_1},\sigma_{p_2}$ can indeed be linearly dependent (and they always  are so for a suitable choice of the sections). For simplicity, we shall still suppose that $p_1,p_2$ do not lie on the axes,  leaving  to the reader that special (and easier) case.  

In the sequel we shall treat this special case, but we stress that the theorem is already proved in the `generic' case when $p_1,p_2$ are not on the same caustics, so that the  arguments we are using below are  needed only in that special case.

We shall use a modification of Theorem \ref{B-C}, which can be formally deduced from the general results of  \cite{B-C} and was formulated earlier by D. Ghioca, L.C. Hsia and T. Tucker. Here is the statement:

\begin{thm}\label{T.Ghioca-Hsia-Tucker}
Let $\rho_1,\rho_2,\kappa$ be sections of a complex elliptic scheme $\E\to B$. If there are infinitely many points $s\in B$ such that the system
\begin{equation*}
\left\{\begin{matrix}
m\kappa(s)&=&\rho_1(s)\\
n\kappa(s)&=&\rho_2(s)
\end{matrix}\right.
\end{equation*}
admits a solution $(m,n)\in\Z^2$ then either there exists an index $i\in\{1,2\}$ and an integer $l\in\Z$ such that $l\kappa=\rho_i$,  or $\rho_1,\rho_2$ are linearly dependent.
\end{thm}

After re-writing \eqref{E.buca2} in the more symmetric  way
\begin{equation}\label{E.buca3}
\left\{\begin{matrix}
(\sigma_{p_1}-\sigma_{h})(s) )&=& (m-n)\kappa(s)\\
(\sigma_{p_2} -\sigma_h)(s) &=& n\kappa(s) 
\end{matrix} 
\right.
\end{equation}
we set $\rho_1=\sigma_{p_1}-\sigma_{h}, \rho_2=\sigma_{p_2}-\sigma_{h}$ and shall apply the above Theorem. 

To proceed in the proof of Theorem \ref{T.buca} we  need to prove the following two lemmas:

\begin{lem}\label{lem1}
Let $p,h\in \T$  be distinct real points, not foci. For no choice of sections $\sigma_p,\sigma_h$ associated to them a relation of the form $\sigma_p-\sigma_h=l\kappa $ can hold.
\end{lem}

In this  lemma, to be proved below together with the next one,  $p$ and $h$ can be either interior points of the billiard table $\T$ or border points; in our application here $p$ is an interior one and $h$ lies on the border $C$.
Note that geometrically this means that, whatever the interior point $p$ and the hole $h$ be fixed, there exists no integer $l$ such that shooting the ball from $p$ in any direction, after $l$ bounces the ball ends in the hole. From this fact it follows that the set of directions from $p$ sending the ball to the hole is at most countable. (In Theorem \ref{T.esiste} we also show that it is indeed an infinite countable set and provide an estimate for the number of suitable directions in term of the number of bounces.)

\begin{lem}\label{lem2}
Suppose that for some choices of three distinct points $p_1,p_2, h$, of which $h$ is on the border, $p_1,p_2$ are interior and not foci, and some choice of sections $\sigma_{p_1},\sigma_{p_2}$ and $\sigma_h$ the two sections $\sigma_{p_1}-\sigma_{h}$, $\sigma_{p_2}-\sigma_{h}$ are linearly dependent. Then the two caustics passing through $p_1$ coincide with those passing through $p_2$   and a minimal  linear relation reads $2(\sigma_{p_2}-\sigma_h)=2(\sigma_{p_1}-\sigma_h)$ i.e 
\begin{equation}\label{E.only-relation}
2\sigma_{p_1}=2\sigma_{p_2}.
\end{equation}
\end{lem}

Let us conclude the proof of Theorem \ref{T.buca} assuming these lemmas (to be proved in a moment). The application of Theorem \ref{T.Ghioca-Hsia-Tucker} provides the sought finiteness result unless $\sigma_{p_1},\sigma_{p_2}$ satisfy the above linearly dependence relation (since the first conclusion of Theorem \ref{T.Ghioca-Hsia-Tucker} cannot hold in view of Lemma \ref{lem1}). In that case, looking at  the system \eqref{E.buca2} we obtain that the solutions $s$ to the system are torsion points for both $\beta$ and $\sigma_{p_2}-\sigma_h$. But $\sigma_{p_2}-\sigma_h$ is non-torsion by the above lemma, and the billiard section $\beta$ is also non-torsion, as we have already remarked. Then by Theorem \ref{CMZ} the set of such points $s$ is finite.
\medskip

It remains now to prove the two lemmas.

\smallskip

{\it Proof of Lemma \ref{lem1}}. We prefer to use the phase space  model of the elliptic scheme described in paragraph \ref{SS.phase-space}.  Recall that the compactification of the total space of the basic billiard surface is the surface $\mathcal{X}$ obtained by suitably blowing up eight times over the surface $C\times C\simeq \P_1\times\P_1$. The billiard map can also be viewed as an automorphism $\beta:\mathcal{X}\to\mathcal{X}$ of that surface. The sections $\sigma_p,\sigma_h$ are indeed multi-sections in this model, and correspond to curves, say $L_p,L_h$,  on the surface $\mathcal{X}$. The lemma we are proving amounts to saying that no power of $\beta$ sends the curve $L_p$ to the curve $L_h$. 

Consider the images $\overline{L}_p$ (resp. $\overline{L}_h$)  of $L_p$ (resp. $L_h$) on $C\times C$ (see diagram \eqref{E.diagramma}).  
Observe that $\overline{L}_h$ intersects the diagonal of $C\times C$ at the point $(h,h)$, while  $\overline{L}_p$  intersects the diagonal at the complex points of the form $(q,q)$ where $q\in C$ is such that the tangent at $q$ passes through $p$. Now, since $p$ is internal, $q$ cannot be real, so in particular $q\neq h$; this implies that $(q,q)$ does not belong to $\overline{L}_h$. Also, since $p$ is not a focus, $q$ is not one of the four points to blow up in the construction of $\mathcal{X}$. It follows that the automorphism $\beta$, viewed as a rational automorphism of $C\times C$, is well defined at these points $(q,q)$ (there are two of them) and fixes them. Since $(q,q)$, which is fixed by $\beta$,  does not belong to $\overline{L}_h$, no power of $\beta$ can send $\overline{L}_p$ to $\overline{L}_h$.
\qed

\smallskip

Note that in the case of our concern, i.e. when $h$ lies on the boundary and $p$ is an interior point, a simpler proof is available: take an elliptic caustic $C_s$ so that $p$ lies inside it; then every trajectory starting from $h$, which is made of segments of lines tangents to $C_s$, cannot pass through the point $p$. This proves that for such an $s$ and for every positive integer $l$, $\sigma_h(s)+l\kappa(s)\neq \sigma_p(s)$.

\smallskip

Concerning Lemma \ref{lem2}, we first note that it can happen that $\sigma_{p_1},\sigma_{p_2}$   satisfy relation \eqref{E.only-relation}. Indeed, suppose that $p_2$ is the image of $p_1$ under one of the three non-identical automorphisms $(x,y)\mapsto(\pm x,\pm y)$ of $C\times C$; this is precisely the case when the caustics passing through $p_1$ are the same as those passing through $p_2$. We already remarked in paragraph \ref{SS.phase-space} that these symmetries  lift on the surface $\mathcal{X}$ to the automorphisms of translations by points of order $2$. The corresponding multi-sections $L_{p_1},L_{p_2}$ are interchanged by one of these order two automorphisms. This implies that after performing a base change so that $L_{p_1}\subset\mathcal{X}$  gives rise to a rational section  $\sigma_{p_1}$,  there is a suitable choice  for the rational section  $\sigma_{p_2}$, depending on the choice of $\sigma_{p_1}$ (recall that there are four choices for each section above $L_{p_2}$), such that the difference $\sigma_{p_2}-\sigma_{p_1}$ is of order two, and so relation \eqref{E.only-relation} holds.

\smallskip

{\it Proof  of Lemma \ref{lem2}}. The lemma has already been proved in the generic case, i.e. with the possible exceptions when $p_1,p_2$ lie on the same caustics and $h$ lies in the hyperbolic caustic containing $p_1,p_2$. Let us then consider only this case. Denote by $\pm \sigma_{p_1}, \pm\tilde{\sigma}_{p_2}$ the four sections associated to the interior point $p_1$ as in \ref{SSS.ell-surface} and by $\sigma_h$ one of the two sections associated to $h$, ``pointing towards $h$''; the other such section is $-\sigma_h-\kappa$. Then the four sections associated to the second interior point $p_2$ are of the form $\pm \sigma_{p_1}+\rho, \pm\tilde{\sigma}_{p_1}+\rho$, where $\rho$ is a torsion section of order $2$.
To prove the lemma, it suffices to prove the independence modulo torsion of the two sections $\sigma_{p_1}-\sigma_h$, $\tilde{\sigma}_{p_1}-\sigma_h$ as well as the independence  of the two sections $\sigma_{p_1}-\sigma_h, -\sigma_{p_1}-\sigma_h$. It then remains to consider only the trivial case  of the pair of sections both  equal to $\theta_{p_1}-\theta_h$ which leads to the relation \eqref{E.only-relation}.

The linear independence of the sections $\sigma_{p_1}-\sigma_h, -\sigma_{p_1}-\sigma_h$ follows simply by the already proven independence of the sections $\sigma_{p_1},\sigma_h$. 

As to the remaining case, this
 can be achieved by computing the canonical height quadratic form in the lattice generated by $\sigma_{p_1}, \tilde{\sigma}_{p_1},\sigma_h$. The calculations follow by a systematic procedure in the theory of elliptic surfaces, so we do not perform them here, and refer to   Chapter 9 of \cite{CZPonc}. It turns out that the intersection matrix of the N\'eron-Tate height bilinear form is  (dropping the index we write $\sigma_p$ for $\sigma_{p_1}$)
 \begin{equation*}
 \begin{matrix}
 {} & \sigma_p & \tilde{\sigma}_p & \sigma_h \\
 \sigma_p & 1 & 0 & 1/2  \\
   \tilde{\sigma}_p & 0 & 1 &  {-1}/{2} \\
 \sigma_h &  {1}/{2} &  {-1}/{2} & 1
 \end{matrix}
 \end{equation*}
 From these data the independence modulo torsion of the three sections $\sigma_p,\tilde{\sigma}_p,\sigma_h$ follows immediately, so in particular the follows    differences $\sigma_{p_1}-\sigma_h$, $\tilde{\sigma}_{p_1}-\sigma_h$ turn out to be independent modulo torsion. \qed
 \smallskip

\medskip

\subsubsection{Proof of the finiteness statement in Theorem \ref{T.ritorno}}. The proof of this statement is similar to that just given for  Theorem \ref{T.buca}.  
Given an interior  point $p\in\T^0$, we can define four algebraic sections $\sigma_p, \tilde{\sigma}_p, -\sigma_p,-\tilde{\sigma_p}$ corresponding to the four possible shots from $p$ with given caustic. Once $\sigma_p$ is chosen, $-\sigma_p $ corresponds to changing the orientation of the trajectory, while $\tilde{\sigma}_p$ and $-\tilde{\sigma}_p$ correspond to the other choice for the tangent. These four sections become rational after a single quartic extension of the base.

Suppose that for a certain shot, i.e. for a point  $s$ of the base, $\sigma_s$ gives rise to a trajectory of type (2) and (3), namely passing through $p$ two more times, once with the same tangent but opposite orientation and once with different direction. Then $s$ will be a solution to the system
\begin{equation}
\left\{\begin{matrix}
\sigma_p(s)+m\kappa(s)&=& -\sigma_p(s)\\
\sigma_p(s)+n\kappa(s)&=& \tilde{\sigma}_p(s).
\end{matrix}\right.
\end{equation}
where $\kappa$ still denotes the billiard section.
This system can be written in the form
\begin{equation}\label{E.ritorno}
\left\{\begin{matrix}
2\sigma_p(s) &=& -m\kappa(s)\\
\sigma_p(s)- \tilde{\sigma}_p(s) &=&  n\kappa(s).
\end{matrix}\right.
\end{equation}
similiarly to the system \eqref{E.buca3}. We have already proved that $\sigma_p,\beta$ are linearly independent. It can be  proved, e.g. looking at the canonical height, that, if $p$ is not on an axis of the ellipse, the three sections $\sigma_p,\tilde{\sigma}_p,\kappa$ are also linearly independent (see again Chapter 9 of \cite{CZPonc}, where we calculated the intersection matrix of the images of these sections, from which independence follows at once). From this fact, an application of Theorem \ref{B-C} provides the sought finiteness.

If, however, the point $p$ lies on an axis but not on a focus, then $\tilde{\sigma}_p$ can be obtained from $\sigma_p$ by applying a symmetry of the ellipse $C$, which corresponds to translation by a point of order $2$. Hence the   linear dependence relation between $\sigma_p,\tilde{\sigma}_p$ reads $2\sigma_p=\pm 2\tilde{\sigma}_p$ and every solution $s$ of \eqref{E.ritorno} gives rise to a torsion point for the two sections $\sigma_p,\kappa$, which enables an application of Theorem \ref{CMZ}. 

Suppose now that the trajectory for a given point $s$ of the base is at the same  time of type (1) and (2), so it is periodic and passes through $p$ (infinitely many times) with both orientations. This corresponds to a system
\begin{equation}
\left\{\begin{matrix}
\sigma_p(s) +m\kappa(s) &=& \sigma_p(s)\\
\sigma_p(s) + n\kappa(s) &=& \tilde{\sigma}_p(s) 
\end{matrix}\right.
\end{equation}
which also reads as
\begin{equation}
\left\{\begin{matrix}
m\kappa(s) &=& 0\\
n(\sigma_p - \tilde{\sigma_p})(s) &=&  0.
\end{matrix}\right.
\end{equation}
We have already proved that the two sections $\sigma_p,\tilde{\sigma}_p$ are independent, so in particular their difference is not torsion; also, the billiard section $\kappa$ is not torsion, so we  can  apply Theorem \ref{ThmComplex},  which provides the finiteness of the solutions $s\in B$ of the above system.

In the case a trajectory is of types (1) and (3), we reduce to the system
\begin{equation}
\left\{\begin{matrix}
(\sigma_p +m\kappa)(s) &=& \sigma_p(s)\\
(\sigma_p  +n\kappa)(s) &=& -\sigma_p(s)\\
 \end{matrix}\right.
\end{equation}
which again implies that both $\sigma_p(s)$ and $\kappa(s)$ are torsion points in $\E_s$. Again Theorem \ref{ThmComplex} leads to finiteness.
This concludes the proof. \qed

\medskip

\subsubsection{Proof of Theorem \ref{T.buca} in the  circular case}.
It remains to treat the case of circular billiards in the context of Theorem \ref{T.buca}. Let then $C$ be the unit circle in the complex plane, so that $C$ is the multiplicative group of complex numbers of modulus $1$. Clearly, the billiard map consists in a rotation, in the following sense; if a segment of the trajectory goes from point $\zeta\in C$ to point $\zeta\xi$, for some $\xi\in C$, then  the orbit consists of the points of the form $\zeta\xi^n$, for $n\in\Z$. We can then parametrize the phase space as pairs $(\zeta,\xi)\in C\times C$, where $\xi$ determines the rotation corresponding to the billiard map.

Let us fix two interior points $p_1,p_2$, both distinct from the center; without loss of generality, we can suppose that the hole is represented by the point $1\in C$.   The shots passing through $p_1$ are parametrized by pairs $(\zeta,\xi)\in C\times C$ such that the real line joining $\zeta$ with $\zeta\xi$ contains $p_1$. If $p_1$ is represented by the complex number $\alpha$ with $0<|\alpha|<1$, then the condition that $\zeta,\zeta\xi,\alpha$ are aligned corresponds to the condition that the solution $t$ to the equation $t\zeta+(1-t)\zeta\xi=\alpha$ be real. Since $t$ is given by $t=(\alpha-\zeta\xi)/(\zeta-\zeta\xi)$, such a condition amounts to the relation
\begin{equation*}
\frac{\alpha-\zeta\xi}{\zeta-\zeta\xi} = \frac{\bar{\alpha}-\bar{\zeta}\bar{\xi}}{\bar{\zeta}-\bar{\zeta\xi}}.
\end{equation*}
Since $\bar{\zeta}=\zeta^{-1}$ and $\bar{\xi}=\xi^{-1}$, the above equation can be written as an algebraic equation which becomes, after some simplifications, 
\begin{equation*}
\xi=-\frac{1-\alpha\zeta^{-1}}{1-\bar{\alpha}\zeta}.
\end{equation*}
 Setting $u=\alpha\zeta^{-1}$ the above equation becomes 
 \begin{equation}\label{E.alpha}
 \xi=-\frac{1-u}{1-|\alpha|^2u^{-1}}.
 \end{equation}
 The condition that after $m$ bounces the ball hits $p_2=:\beta$ amounts to a similar equation referred to the pair $(\zeta\xi^m,\xi)$, namely 
 \begin{equation}\label{E.beta}
 \xi=-\frac{1-v}{1-|\beta|^2v^{-1}}.
 \end{equation}
where $v=\beta\zeta^{-1}\xi^{-m}$.

Finally, the ball will fall into the hole after $n$ bounces if and only if $\zeta\xi^n=1$.  This gives the relations
\begin{equation}\label{E.u,v}
u=\alpha\xi^n,\qquad v=\beta\xi^{n-m}.
\end{equation}

The system of equations \eqref{E.alpha}, \eqref{E.beta} defines a curve $\mathcal{X}\subset \G_{\mathrm{m}}^3$, while the relations \eqref{E.u,v} imposes two multiplicative dependence conditions. In other words, we are interested in solutions $(\xi,u,v)$ such that $\alpha^{-1}u,\beta^{-1}v$ lie in the multiplicative subgroup generated by $\xi$. 

\smallskip

We first note that the algebraic curve $\mathcal{X}$, birational to a cubic curve in the plane, is irreducible of genus $1$ if $|\alpha|\neq|\beta|$, while in the special case $|\alpha|=|\beta|$ it is the union of a conic and one line. \smallskip

Let us first consider  the generic case, when $|\alpha|^2\neq |\beta^2|$. We claim that in this case the curve $\mathcal{X}\subset\G_{\mathrm{m}}^n$ is not contained in any translate of an algebraic subgroup. Indeed, first note that $v$ is quadratic over $\C(u)=\C(u,\xi)$, while $u$ is quadratic over $\C(v,\xi)=\C(v)$. Now, starting from an equation of the form $\xi^a\cdot u^b\cdot v^c=\mathrm{const.}$, supposing e.g. $c\neq 0$,  we would deduce that the conjugate $v'$ of $v$ differs multiplicatively from $v$ by a root of unit. Writing $v'=\zeta v$, with $\zeta^c=1$, we obtain that its trace is $(1+\zeta)v\in \C(u)$, forcing $\zeta=-1$ and its trace being zero. However, from the quadratic equation \eqref{E.beta} follows that the trace is $1+\xi\neq 0$. 

We can then apply Theorem \ref{T.BMZ}, which provides the finiteness of the points $(\xi,u,v)\in\mathcal{X}$ satisfying the dependence conditions \eqref{E.u,v}. This concludes the proof in the case $|\alpha|^2\neq |\beta^2|$.
\smallskip

Let us consider now the special case $|\alpha^2|=|\beta^2|=:r$.  In this case,   we have
\begin{equation}\label{E.degenere}
u=v \qquad \mathrm{or} \qquad uv+r-r(u+v)=0.
\end{equation}
 
The first possibility $u=v$ leads to a curve contained in a proper algebraic subgroup of $\G_{\mathrm{m}}^3$. Hence Theorem \ref{T.BMZ} does not apply. In this case, however, the relations \eqref{E.u,v} are equivalent to
$\xi^m=\beta/\alpha$. We have then to solve equation \eqref{E.alpha} (or the equivalent equation \eqref{E.beta}) under the condition $u=\alpha\xi^n$, knowing that $\xi$ is an $m$-th root of $\beta/\alpha$. If $\beta/\alpha$ is a root of unity, we can apply directly Theorem \ref{T.torsion}, since the curve defined in $\G_{\mathrm{m}}^2$ by equation \eqref{E.alpha} is not an algebraic subgroup. Otherwise we can either apply the full Laurent's Theorem, using for $\Gamma$ the division group of the cyclic group generated by $\alpha/\beta$ or use the following trick: given a solution $(\xi,u)=(\xi,\alpha\xi^n)$ with $\xi^m=\alpha/\beta$ we obtain other solutions conjugating over $\C(\alpha,\beta)$. The conjugate solutions are of the form $(\theta\xi,\theta^n\xi)$ for some root of unit $\theta$. After taking two conjugates, we obtain two more equations. We then obtain three algebraic equations in $(\xi,u,\theta,\theta')$ to be solved in roots of unity $\theta,\theta'$. Eliminating $\xi,u$, one obtains a single algebraic equation in $\theta,\theta'$ to which Theorem \ref{T.torsion} applies. A third way to conclude in this case is applying a theorem of A. Schinzel on irreducibility of lacunary polynomials \cite{S}, which constitutes a particular case of Theorem \ref{T.BMZ} (see Remark after Theorem 2 in \cite{BMZ}).

Let us now consider the second possibility in the degenerate case $|\alpha|=\beta|$, i.e. the second equality in \eqref{E.degenere}. In that case, after rewriting \eqref{E.alpha}, \eqref{E.beta} and the second equality in \eqref{E.degenere} we arrive at the system of equations 
\begin{equation*}
\left\{\begin{matrix}
\xi&=&\frac{u-1}{u-r}\cdot u\\ \xi&=&\frac{v-1}{v-r}\cdot v \\ u&=& \frac{v-1}{v-r}\cdot r
\end{matrix}
\right.
\end{equation*}
The above equations imply $uv\xi^{-1}=r$, so our curve is contained in a non-torsion translate of an algebraic subgroup. In this situation we can either conclude by applying the full Maurin's theorem, or argue more simply in a way similar to the previous one, since  from the last equality and relations \eqref{E.u,v} we can deduce that $\xi^{2n-m-1}=r\alpha^{-1}\beta^{-1}$.
\qed

\section{Proof of Theorem \ref{T.P_2}} \label{S.P_2}

We recall the notation in the statement of Theorem \ref{T.P_2}. We are given three lines $L_1,L_2,L_3$ in the plane and we are looking for the (complex) points $P\in L_1$  such that for some pair $(m,n)\in\Z^2$, $\beta^m(P)\in L_2, \beta^n(P)\in L_3$. It is useful to write this condition in the form
$P\in L_1\cap \beta^{-m}(L_2)\cap \beta^{-n}(L_3)$. In view of our assumption that $L_1,L_2,L_3$ lie in different orbits of $\beta$, the three lines $L_1,\beta^{-m}(L_2),\beta^{-n}(L_3)$ are pairwise distinct and we shall have $L_1\cap \beta^{-m}(L_2)\cap \beta^{-n}(L_3)=\{P\}$. 
\smallskip

The  automorphism $\beta$ is given by a $3\times 3$ invertible matrix, defined  up to scalars. The powers of $\beta$ form a cyclic subgroup of $\mathrm{PGL}_3(\C)$; its Zariski-closure is an algebraic  sub-group of $\mathrm{PGL}_3$; we shall denote by $G_\beta$ the connected component of this algebraic group containing the identity.

For the algebraic group $G_\beta$, as a linear algebraic group, there are six possible isomorphism classes: $\{1\}, \G_a,\G_m,\G_a\times \G_m, \G_m^2$. The first case arises when $\beta$ has finite order; the last case is the generic one, holding outside a countable union of proper Zariski-closed subsets of $\mathrm{PGL}_3$, and corresponds to a diagonalizable matrix which, in suitable normalized form, has eigenvalues $1,\lambda_1,\lambda_2$ with $\lambda_1,\lambda_2$ multiplicatively independent.

\smallskip

The aim of this sub-section is the proof of the following proposition, of which Theorem \ref{T.P_2} is an immediate consequence:

\begin{prop}\label{P.P_2}
Let $L_1,L_2,L_3$ be three lines in $\P_2$, $\beta\in\mathrm{PGL}_3(\C)$ be a projective automorphism. Suppose that the three lines $L_1,L_2,L_3$ belong to distinct  orbits for the group generated by $\beta$.  Denote by $\mathcal{O}=\mathcal{O}(L_1,L_2,L_3;\beta)$ the set of orbits for the action of $\beta$ on $\P_2(\C)$ which intersect each of the three lines $L_1,L_2,L_3$. Let $G_\beta\subset\mathrm{PGL}_3$ be the algebraic group determined by $\beta$ as above.

Then:

\begin{enumerate}
\item If $G_\beta=\G_m^2$ then $\mathcal{O}$ is finite.

\item If $G_\beta=\G_m$ and $\mathcal{O}$ is infinite, then at least one of the three lines, say $L_1$, contains a fixed point for $\beta$. Also, the set of pairs 
$(m,n)$ such that there exists a point $P\in L_1$ such that $\beta^m(P)\in L_2, \beta^n(P)\in L_3$ is the union of finitely many cosets of subgroups in $\Z^2$. 

Also, either each of the three lines contains a fixed point and there exists an automorphism $\gamma\in\mathrm{PGL}_3(\C)$ such that the  group $\Gamma=<\beta,\gamma>$ is commutative and $\mathcal{O}$ is contained in finitely many orbits for $\Gamma$, or only one line contains a fixed point. In this last case, the set of points $P$ on such a line whose orbit lies in $\mathcal{O}$ is given, in suitable affine coordinates,  by the values of a binary linear recurrence sequence.

\item If $G_\beta=\G_a\times\G_m$ and $\mathcal{O}$ is infinite, then
two among the three   lines $L_1,L_2,L_3$, say $L_1,L_2$,  contain  a fixed point for $\beta$. Also, the set of pairs $(m,n)\in\Z^2$ such that there exists a point $P\in L_1$ such that $\beta^m(P)\in L_2, \beta^n(P)\in L_3$ is the union of a finite set and an infinite set of the form $m=a\lambda^n+bn$, where $a,b\in\Q^*$. Also, there exists an automorphism $\gamma\in\mathrm{PGL}_3(\C)$ such that the  group $\Gamma=<\beta,\gamma>$ is commutative and $\mathcal{O}$ is contained in finitely many orbits for $\Gamma$.

\item If $G_\beta=\G_a$ and $\mathcal{O}$ is infinite, then the set of fixed points for $\beta$ contains a line, so each of the three lines $L_1,L_2,L_3$ contains a fixed point. All the pairs $(m,n)\in \Z^2$ such that there exists a point $P\in L_1$ such that $\beta^m(P)\in L_2$ and $\beta^n(P)\in L_2$ lie on a single line in $\Z^2$. Also, there exists an automorphism $\gamma\in\mathrm{PGL}_3(\C)$ as in case (3).     

\end{enumerate}

\end{prop}

In each case, there is an effective procedure to detect whether the set $\mathcal{O}$ is infinite: when it is so, $\mathcal{O}$ is the union of finitely many infinite families, which can be explicitly described, and a finite set.

Examples of infinite families in cases (2), (3) and (4) will be constructed.

Let us begin the proof of Proposition \ref{P.P_2}, by fixing the notation which will be used.

In projective coordinates, the three lines $L_1,L_2,L_3$ are defined by the vanishing of a linear form in three variables, which will also be denoted by $L_1,L_2,L_3$. Let us write the three linear forms as
\begin{equation*}
\begin{matrix}
L_1(x_1,x_2,x_3)&=&a_1x_1+a_2x_2+a_3x_3\\
 L_2(x_1,x_2,x_3)&=&b_1x_1+b_2x_2+b_3x_3\\
  L_3(x_1,x_2,x_3)&=&c_1x_1+c_2x_2+c_3x_3
  \end{matrix}
\end{equation*}

Before proceeding to the proof, let us make three useful remarks, which will be applied several times throughout the proof.

\begin{small}\begin{rem}\label{R.P,m,n}
 We are looking for triples $(P,m,n)$ with $P\in L_1$, $m,n\in\Z$, such that
\begin{equation}\label{E.P,m,n}
P\in L_1,\quad \beta^{m}(P)\in L_2,\quad \beta^n(P)\in L_3.
\end{equation}
We are assuming that $L_1,L_2,L_3$ have distinct $\beta$-orbits. Hence every pair $(m,n)\in\Z^2$ can give rise to at most one triple $(P,m,n)$ satisfying the above relation. Also, if $m,n$ or $m-n$ bounded there are at most finitely many points  $P\in\P_2$ giving rise to a solution $(P,m,n)$  of \eqref{E.P,m,n} .  
\end{rem}\end{small}

\begin{small}\begin{rem}\label{R.inv}
If the first line $L_1$ is invariant under $\beta$, then there can be at most one   $\beta$-orbit intersecting $L_1$, $L_2$ and $L_3$. Indeed, if there is one such orbit, up to changing $L_2$, $L_3$ by suitable pre-images of themselves under powers of $\beta$, we can suppose that $L_1\cap L_2\cap L_3=\{P\}$ is non-empty. Then  the only orbit intersecting all the three lines is the orbit of $P$.  Of course, the same remark applies when $L_2$ or $L_3$ is invariant. 
\end{rem} \end{small}
\smallskip  

\begin{small}\begin{rem}\label{R.fix}
If two of the three lines intersect in a point $P$  which is fixed for $\beta$, then the  orbit of $P$ is the only orbit which  can intersect all the three lines, and actually $P$ would be the unique point, if any, belonging to $L_1$ and possessing images both in $L_2$ and in $L_3$.   
\end{rem}\end{small}

To prove Proposition \ref{P.P2}, we shall distinguish several cases, according to the nature of the group $G_\beta$, which in turn depends on the Jordan form of a matrix associated to $\beta$.

\smallskip

{\it Case 1: $G_\beta=\G_m^2$}. As mentioned, this is the generic case. 
It means that the automorphism is diagonalizable with pairwise distinct eigenvalues $\lambda_1,\lambda_2,\lambda_3$ such that $\lambda_1/\lambda_3, \lambda_2/\lambda_3$ are multiplicatively independent. We can choose suitable coordinates so that the matrix is in diagonal form, and also normalize it so that one eigenvalue equals  $1$, so the matrix of $\beta$ will be of the form 
\begin{equation}\label{E.beta-diagonal}
\left(\begin{matrix} \lambda_1 & 0& 0\\ 0& \lambda_2 & 0\\ 0&0&1\end{matrix}\right).
\end{equation}
Given a pair of positive integers $m,n$, 
the existence of a point $x=(x_1:x_2:x_3)\in\P_2(\C)$ such that  $L_1(x)=L_2(\beta^m(x))=L_3(\beta^n(x))=0$ amounts to the relation
\begin{equation*}
\det \left(\begin{matrix} a_1 & a_2 & a_3 \\ b_1\lambda_1^m & b_2 \lambda_2^m & b_3  \\ c_1\lambda_1^n & c_2\lambda_2^n & c_3
\end{matrix}\right)=0
\end{equation*}
Let us then consider the subvariety $V$ of $\G_m^4$ defined by the equation
\begin{equation}\label{E.V}
V:\qquad \det \left(\begin{matrix} a_1 & a_2 & a_3 \\ b_1 x_1 & b_2 x_2 & b_3  \\ c_1 y_1 & c_2 y_2 & c_3
\end{matrix}\right)=0
\end{equation}
We are looking for points $(x_1,x_2,y_1,y_2)=(\lambda_1^m, \lambda_2^m,\lambda_1^n,\lambda_2^n)\in V\cap \Gamma$, where $\Gamma\subset\G_m^4(\C)$ is the multiplicative group generated by the points $(\lambda_1,\lambda_2,1,1)$ and $(1,,1,\lambda_1,\lambda_2)$. 
\medskip

Suppose first that the variety $V$ is the whole of $\G_m^4$. Then in particular for every choice of $m,n\in \Z$, the three lines $L_1, \beta^{-m}(L_2),\beta^{-n}(L_3)$ are concurrent. This can happen only if their intersection point is fixed for the group, hence for $\beta$, and we conclude via Remark \ref{R.fix}.  
\smallskip

We can then suppose that $V$ is a  hypersurface of $\G_m^4$.  

Let $W\subset V$ be a positive dimensional irreducible component of the Zariski-closure of $V\cap\Gamma$ in $V$. By Theorem \ref{T.tori}, it is a translate of a  subtorus of $\G_m^4$. In particular, $W$ is contained in a translate of a maximal proper subtorus, so its points satisfy a monomial relation of the form
\begin{equation*}
x_1^ax_2^b y_1^cy_2^d=\xi
\end{equation*}
for some non-zero complex number $\xi\in\C^*$ and some non-zero vector $(a,b,c,d)\in\Z^4$. By our assumptions $\Gamma\cap W$ is dense in $W$, so in particular we have infinitely many solutions $(m,n)\in\Z^2$ to the equation
\begin{equation}\label{E.dip-m,n}
\lambda_1^{am+cn}\lambda_2^{bm+dn}=\xi.
\end{equation}
Since $\lambda_1,\lambda_2$ are multiplicatively independent, the above relation implies a pair of relations of the form
\begin{equation*}
\left\{\begin{matrix}
am+cn &=& e\\
bm+dn &=& f
\end{matrix}\right.
\end{equation*}
for suitable fixed integers $e,f\in\Z$ and all $m,n$.   

Then either the above relations are independent, so there exists at most one pair $(m,n)$ satisfying them, or  we can parametrize all the pairs $(m,n)$ as $m=m_0-ct, n=n_0+at$
for suitable $m_0,n_0\in \Z$, where $t$ varies in $\Z$. Put  $b_1^*=b_1\lambda_1^{m_0}, b_2^*=b_2\lambda_2^{m_0}, c_1^*=c_1\lambda_1^{n_0} $ and $c_2^*=c_2\lambda_2^{n_0}$. Note that $b_1^*, b_2^*,b_3$ (resp. $c_1^*,c_2^*,c_3$) are the coefficients of the linear form vanishing on $\beta^{-m_0}(L_2)$ (resp. on $\beta^{-n_0}(L_3)$. Hence by our assumptions the line vectors $(b_1^*,b_2^*,b_3), (c_1^*,c_2^*,c_3)$ are linearly independent.  

By Remark \ref{R.P,m,n}, if $a$ or $b$ vanishes, we are done. Then suppose none of them vanishes, so we obtain 
 that identically (for $(x_1,x_2)\in\G_m^2$)
\begin{equation}\label{E.matrix-x_1,x_2}
\det \left(\begin{matrix} a_1 & a_2 & a_3 \\ b_1^*x_1^{-bt}  & b_2^* x_2^{-bt} & b_3  \\ 
c_1^*x_1^{at} & c_2^*x_2^{at}  & c_3
\end{matrix}\right)=0.
\end{equation}

Now, if the third column vanishes, the point $(0:0:1)$, which is a fixed point for the group, is the intersection point $L_1\cap \beta^{-m}(L_2)\cap \beta^{-n}(L_3)$, hence of $L_1\cap L_2\cap L_3$ and by Remark \ref{R.fix} we are done. Analogously, if the second column vanishes, then $(0:1:0)$, which is also a fixed point for the group, would be the intersection point of $L_1\cap \beta^{-m}(L_2)\cap \beta^{-n}(L_3)$, while the vanishing of the first column would give $L_1\cap \beta^{-m}(L_2)\cap \beta^{-n}(L_3)=(1:0:0)$, and again by Remark \ref{R.P,m,n} we are done.

Suppose then that no column in the   matrix in \eqref{E.matrix-x_1,x_2}  vanishes. 

In view of Remark \ref{R.inv}, we can also suppose that $(a_2,a_3)\neq (0,0)$.

If the second column in \eqref{E.matrix-x_1,x_2} is a multiple of the third one for every $x_2\in\G_m$, then since, as we remarked,  we cannot have neither $(a_2,a_3)=0$ nor the vanishing of the second column, we shall have $a_2\neq 0, a_3\neq 0$ and  so $b_2^*x_2^{-bt}$ as well as $c_2^*x_2^{at}$ would be constant. This implies $b_2^*=c_2^*=0$, so $\beta^{-m}(L_2)\cap \beta^{-n}(L_3)$ would be the fixed point $(0:1:0)$. But this point would then belong to $L_2\cap L_3$ and by Remark \ref{R.fix} we are done.

Finally, we can suppose that the last two columns are generically independent. Then, since the determinant in \eqref{E.matrix-x_1,x_2} is identically zero, the first column is, for every $x_1\in\G_m$,    a multiple of a constant vector. This can happen only if two of the three coefficients $a_1,b_1^*,c_1^*$ vanish, or if $a_1$ vanishes and $a=b$. In the first case two lines would meet in a fixed point for the group, and once again we apply Remark \ref{R.fix}. In the second case we have $a-b=0$, i.e. $(m,n)=(m_0,n_0)+(t,t)$ and by Remark \ref{R.P,m,n} we obtain the finiteness of the relevant points $P$.

This achieves the proof in Case 1.
\smallskip

{\it Case 2: $G_\beta=\G_m$}. In this case the Zariski closure of the cyclic group generated by $\beta$ might be a disconnected algebraic group.  Observe, however, that the conclusion does not change if we replace $\beta$ by one of its powers with non-zero exponent, and replace the three lines by suitable elements in their $\beta$-orbit. Hence we can and shall suppose that the Zariski closure of the cyclic group generated by $\beta$ itself is isomorphic to $\G_m$, so that $\beta$ can be expressed in suitable coordinates by a matrix of the form
\begin{equation*}
\left(\begin{matrix} \lambda^u &0&0\\ 0&\lambda^v &0 \\ 0&0&1\end{matrix}
\right)
\end{equation*}
where $\lambda\in\C^*$ is not a root of unity and $u,v$ are non-zero integers. 

We repeat the argument developed in Case 1 (and the notation for the equations of the three lines): consider the variety $V\in\G_m^2$ defined by the equation
\begin{equation*}
\det \left(\begin{matrix} a_1 & a_2 & a_3 \\ b_1 x_1^u & b_2x_1^v & b_3\\ c_1 x_2^u &c_2x_2^v &c_3 \end{matrix}
\right)
\end{equation*}
By the same argument as in Case 1, if the determinant vanishes identically, the above matrix contains at least two zeros in one single row or one single column, so we conclude  by applying one of the Remarks \ref{R.fix}, \ref{R.inv}. So we assume that $V$ is a curve in $\G_m^2$ and look for its points of the form $(x_1,x_2)=(\lambda^m,\lambda^n)$. In principle, the equation of $V$ involves monomials of the form $x_1^ix_2^j$ where for the exponents $(i,j)$ six pairs are possible, namely
\begin{equation*}
(u,0),\, (v,0)\, (0,u)\, (0,v),\, (u,v),\, (v,u).
\end{equation*}
It is easy to see that if $u$ or $v$ or $u-v$ vanishes, then $V$ is indeed a translate of a subgroup, but the resulting equation in $(\lambda^m,\lambda^n)$ leads to a bound on either $m$ or $n$ or $n-m$ and we conclude by Remark \ref{R.P,m,n}. 

So we can suppose that $u,v,u-v$ are all non-zero, which implies that the six pairs of integers displayed above are pairwise distinct. Then, 
in order that the determinant equation involves only two monomials, so that it defines a translate of a sub-torus of $\G_m^2$, it is necessary and sufficient that three coefficients in the above matrix vanish; as we already remarked, these coefficients must belong to different lines and rows. 

Then, up to permuting the coordinates we can suppose that the vanishing coefficients are $a_3,b_2,c_1$ and we obtain the equation $b_1c_3a_2x_1^u=+c_1a_2b_3x_2^u=0$. This equation can have infinitely many solutions of the form $(x_1,x_2)=(\lambda^m,\lambda^n)$. In this case the pair $(m,n)$ satisfies the equation 
\begin{equation*}
mu+n(u-v)=d
\end{equation*}
 where $d\in\Z$ is fixed. It then follows that the pairs $(m,n)$ are parametrized as
 \begin{equation*}
 m=m_0+(v-u)t,\, n=n_0+ut
 \end{equation*}
for suitable $m_0,n_0$ and varying $t$. 
Then, after replacing $L_2$ by $\beta^{m_0}(L_2)$ and $L_3$ by $\beta^{n_0}(L_3)$, we obtain an infinite family of solutions $(P,(v-u)t,ut)$. The solution corresponding to $t=0$ provides a point $P$ in the intersection $L_1,\cap L_2\cap L_3$. In our coordinates, $L_1$ contains the fixed point $(1:0:0)$, $L_2$ th fixed point $(0:1:0)$ and $L_3$ the fixed point $(0:0:1)$. After change of coordinates, which will not affect the matrix of $\beta$, we can suppose that $P=(1:1:1)$. Hence $L_1$ (resp. $,L_2,L_3$) must be defined by the   equation  $y=z$  (resp. $x=z$, $x=y$). The points of $L_1$
\begin{equation*}
P_t:=(1:\lambda^{(vu-u^2)t}:\lambda^{(vu-u^2)t})
\end{equation*}
have the properties that for $m=t(v-u)$ and $n=tu$, $\beta^m(P_t)\in L_2$ and $\beta^n(P_t)\in L_3$.
Note that all the points $P_t$ lie in a same orbit for the action of a diagonal matrix, hence for an automorphism commuting with $\beta$. 

It remains to consider the other situation,  of an equation involving more than two monomials and nevertheless  defining a translate of an algebraic subgroup. Namely, the determinant might be a reducible polynomial  in $\C[x_1^{\pm 1},x_2^{\pm}]$,  one of its factor being a binomial. This  case can occur for instance when $u=-v$ (see Example \ref{Ex.u=-v}) below. In this case too the conclusion that the relevant pairs $(m,n)$ belong to finitely many lines in $\Z^2$ still hold. We omit the proof that in this case the corresponding points $P\in L_1$ form a binary linear recurrent sequence. This fact will be clear from the concrete example presented in \ref{Ex.u=-v}.

\smallskip

{\it Case 3: $G_\beta=\G_a\times\G_m$}. In this case we can suppose that the matrix of $\beta$ takes the form
\begin{equation*}
\left(\begin{matrix}
1& 1&0\\ 0&1 & 0\\ 0&0&\lambda
\end{matrix}\right)
\end{equation*}
for some $\lambda\in\C^*$ which is not a root of unity. Its powers are
\begin{equation*}
\left(\begin{matrix}
1 & n &0\\ 0& 1 & 0\\ 0&0&\lambda^n
\end{matrix}\right).
\end{equation*}
The presence of a triple $(P,m,n)$ satisfying \eqref{E.P,m,n} is equivalent to the equation
\begin{equation*}
\det \left(\begin{matrix} a_1 & a_2 & a_3 \\ b_1  &  b_1 m +b_2  & b_3\lambda^m  \\ 
c_1   &  c_1 n +c_2     & c_3 \lambda^n
\end{matrix}\right)=0.
\end{equation*}

Then the above equation reads as
\begin{equation}\label{E.separata}
\lambda^m P(n)-\lambda^n Q(m)=R(m,n)
\end{equation}
for suitable polynomials $P(X),Q(Y),R(X,Y)\in\C[X,Y]$  of degree $\leq 1$. 

 Suppose first that $b_3$ and $c_3$ are both non-zero, so that we can take them to be equal to $1$. Then 
 \begin{equation*}
  P(X)=a_1c_1X+a_1c_2-c_1a_2, \quad  Q(X)=a_1b_1X+a_1b_2-a_2c_1.
 \end{equation*}
 If one of the two polynomial vanishes then one couple among  $(a_1,a_2)$, $(a_1,b_1)$ and $(a_1,c_1)$ must vanish. Then we can apply either  Remark \ref{R.inv} (in the first case) or Remark \ref{R.fix} (in the second and third cases).

  Then we can suppose that the two polynomials $P(X),Q(X)$ are both non-zero. We claim that then either $m$, or $n$ or $m-n$ is bounded, which permits to conclude via Remark \ref{R.P,m,n}. Indeed, suppose that an infinite sequence of solutions $(m,n)$ to equation \eqref{E.separata} exists with $m,n,|m-n|$ all tending to infinity. We can extract a subsequence for which $m>n$ and $m/n$ converges in $[1,+\infty]$. Dividing by $\lambda^n$ we get the equation 
  \begin{equation*}
  P(n)\lambda^{m-n}=Q(m)+R(m,n)\lambda^{-n}.
  \end{equation*}
 Now, since $\lambda$ is not a root of unity, there exists an absolute value $|\cdot|_\nu$ of a field containing $\lambda$ and all the involved coefficients with  $|\lambda|_\nu >1$ (if $\lambda$ is transcendental this is  possible since it can be sent by a possibly discontinuous automorphism of $\C$ to any other transcendental number; if $\lambda$ is algebraic, this follows from  the  fact that its height is $>1$).   Now, if $m/n\to 1$, we can bound $|R(m,n)|_\nu\ll 1+n$ and so clearly $R(m,n)\lambda^{-n}\to 0$, while $|Q(m)|_\nu \ll m$ so that the above relation cannot hold for large $n$ and large $m-n$. If on the contrary  $m/n$ does not tend to $1$, then $m-n\gg m$ and again the above relation cannot hold for large $m$. 
  \smallskip
  
It remains then to consider the case in which one between $b_3$ and $c_3$ vanishes. Note that we can suppose that exactly  one of them vanishes, otherwise the two lines $L_2,L_3$ would meet in the fixed point $(0:0:1)$ and we would conclude via Remark \ref{R.fix}. For the same reason we can also assume  that $a_3\neq 0$.

Suppose then that  $b_3=0$, $c_3=a_3=1$.  The equation \eqref{E.separata} reads
\begin{equation*}
\lambda^n(a_1b_2-a_2b_1+a_1b_1m)+(b_1c_2-b_2c_1+b_1c_1(n-m))=0
\end{equation*}
which can be written as
\begin{equation*}
m=\frac{A\lambda^n+Bn+C}{D\lambda^n+E}
\end{equation*}
where $A=a_2b_1-a_1b_2, B=-b_1c_1, C=b_1c_2-b_2c_1, D=a_1b_1, E=-b_1c_1$. Recall that we are supposing that the above ratio on the right-hand side, depending only on $n$, takes infinitely many integral values. By taking a subsequence of these integer numbers,  and looking at its limit, we conclude that $D=0$. Then $m$ is of the form $A\lambda^n+Bn+C$, as predicted by the proposition we are proving. Note also that $B\neq 0$, otherwise either $b_1=0$, so $L_2$ would be invariant for $\beta$, or $b_1\neq 0$, $c_1=0$ and, in view of $D=0$,   we would have $a_1=c_1=0$, so that $L_1\cap L_3$ would be a fixed point. Again, Remarks \ref{R.inv} or \ref{R.fix} would provide finiteness.

 \smallskip
 
 {\it Case 4}: $G_\beta=\G_a$. In this case the automorphism $\beta$ can admit either just one fixed point, or infinitely many (forming a line).  In the first case the matrix of $\beta$ is conjugate to the matrix
 \begin{equation*}
\left(\begin{matrix} 1 &1&0\\ 0&1 &1 \\ 0&0&1\end{matrix}
\right)
\end{equation*}
 so that its $n$-th power reads
 \begin{equation*}
\left(\begin{matrix} 1 &n&{n}\choose{2}\\ 0&1 &n \\ 0&0&1\end{matrix}
\right)
\end{equation*}
We now prove that in this situation the set $\mathcal{O}$ is finite.
With the above notation, the determinant equation becomes
\begin{equation*}
\det\left(\begin{matrix} a_1 &a_2&a_3\\ b_1& b_1m+b_2 & b_1 {m \choose 2} +b_2m+b_3 \\ c_1& c_1n+c_2 & c_1 {n \choose 2} +c_2n+c_3
\end{matrix}
\right)=0
\end{equation*}
which reads
\begin{equation}\label{E.cubica}
a_1b_1c_1(m^2n-mn^2)=P(m,n)
\end{equation}
for a polynomial $P(X,Y)\in\C[X,Y]$ of degree $\leq 2$. 

If $a_1b_1c_1=0$ then we can suppose, by symmetry, that $c_1=0$. Then the equation becomes of the form
\begin{equation*}
Q(m)=nF(m)
\end{equation*}
for   polynomials $Q(X), F(X)$ with $\deg Q\leq 2$, $\deg F\leq 1$. Now, either the solutions $(m,n)$ of the above diophantine equation have $m$  bounded, or the two polynomials $Q(X),F(X)$ admit a common integral root $m_0$; in that case, however, the line $\beta^{-m_0}L_2$ would coincide with  $L_1$, contrary to our assumptions.

We can then suppose that the coefficient $a_1b_1c_1$ in equation \eqref{E.cubica} does not vanish, so that the equation defines a cubic curve with three (rational) points at infinity. Now, if this curve is irreducible, then by Runge's theorem it has only finitely many integral points, which can be easily found. In the reducible case,   the only components which can contain infinitely many integral points are lines; now, for such a line, looking at its point at infinity one sees that its integral points $(m,n)$ would have either $m$, or $n$ or $m-n$ bounded, concluding the proof in this sub-case.

\smallskip

 Let us now consider the case in which $G_\beta=\G_a$ and $\beta$ admits infinitely many fixed points.  In suitable coordinates, the matrix for $\beta$ reads  
\begin{equation*}
\left(\begin{matrix}
1& 1&0\\ 0&1 & 0\\ 0&0&1
\end{matrix}\right)
\end{equation*}
The relevant determinant equation becomes 
\begin{equation*}
\det\left(\begin{matrix} a_1 &a_2&a_3\\ b_1& b_1m+b_2 &  b_3 \\ c_1& c_1n+c_2 &  c_3
\end{matrix}
\right)=0
\end{equation*}
which is either unsolvable or is the equation of a line.

\qed
 \medskip

We now show how to produce infinite families of solutions in each of the cases (2), (3) and (4). \medskip

\begin{example}  Here is a first example of an infinite set of orbits intersecting all the lines in the setting of Proposition \ref{P.P_2}.
 The three lines are given by the equations:
\begin{equation*}
L_1:\, y-z=0  \quad L_2:\,  x+y=0 \quad
L_3:\, x+y+z=0
\end{equation*}
Thake for $\beta$ the  automorphism  represented by the matrix
\begin{equation*}
\left(\begin{matrix} 1&1&0\\ 0&1&0\\ 0&0&2
\end{matrix}\right)
\end{equation*}
so that $G_\beta=\G_a\times\G_m$. Note the fixed points $(1:0:0)\in L_1$ and $(0:0:1)\in L_2$. 

The sequence of points in $L_1$
\begin{equation*}
P_m=(-m-1:1:1)
\end{equation*}
has the property that for each 
  $m\in\Z$, $\beta^m(P_m)\in L_2$; also, if  $m=2^n+n$ then $\beta^n(P_m)\in L_3$.
  
  Note that letting $\gamma$ denote the automorphism defined by the matrix
  \begin{equation*}
\left(\begin{matrix} 1&1&0\\ 0&1&0\\ 0&0&1
\end{matrix}\right)
\end{equation*}
we have $\gamma\circ\beta=\beta\circ\gamma$ and $P_m=\gamma^{-m}(P_0)$. 
\end{example}

\begin{example}\label{Ex.u=-v}
Here is another example of an infinite family of solutions for $G_\beta=\G_m$.
Given non-zero complex numbers $a,b$, the three lines are defined by the equations
\begin{equation*}
L_1:\, x-y=0\quad
L_2:\, ax+by=z\quad
L_3:\, bx+ay=z
\end{equation*}
The automorphism is defined by the matrix
\begin{equation*}
\left(\begin{matrix} \lambda&0&0\\ 0&\lambda^{-1}&0\\ 0&0&1
\end{matrix}\right)
\end{equation*}
for any $\lambda\in\C^*$, not a root of unity, so  that $G_\beta=\G_m$. The line $L_1$ contains the fixed point $(0:0:1)$, while the other lines contain no fixed point.

Consider the sequence $P_m\in L_1$ defined by
\begin{equation*}
P_m=(1:1:a\lambda^m+b\lambda^{-m}).
\end{equation*}
(Observe that the  sequence $m\mapsto a\lambda^m+b\lambda^{-m}$ is a binary recurrent sequence.)
Now  $\beta^m(P_m)\in L_2$ and $\beta^{-m}(P_m)\in L_3$. Note that the automorphism $F\in \mathrm{PGL}_3(\C)$ sending $(x:y:z)\mapsto (y:x:z)$ restricts to the identity on $L_1$, interchanges $L_2$ with $L_3$ and satisfies $F\circ \beta \circ F^{-1}=\beta^{-1}$. From these facts it follwos that  if a point $P\in L_1$ satisfies $\beta^m(P)\in L_2$, then $\beta^{-m}(P)=F(\beta^m(P))\in L_3$.

\end{example}

\begin{example} It is easy to construct examples where the set $\mathcal{O}$ is infinite when $G_\beta=\G_a$. In this case, one can take for $\beta$ an affine translation, say defined by $\beta(x,y)=(x+1,y)$. 
Observe that for every three lines $L_1,L_2,L_3$ in general position on the plane, defined over $\Q$, none of which is horizontal, there are infinitely many horizontal lines cutting $L_1,L_2,L_3$ at points having integral distances.  This gives an infinite family of $\beta$-orbits intersecting $L_1,L_2,L_3$. However, not every example   is given by  lines and endomorphisms definable (in suitable coordinates) over the rationals.
\end{example}

We end this section by classifying the cases when infinitely many points of a single orbit lie on a line, and this orbit intersects the other two lines. This completes the classification of the cases when infinitely many points on a line have their orbit intersecting the other two lines. 
We prove the following 

\begin{prop}\label{P.SML}
Suppose the lines $L_1,L_2,L_3\subset\P_2$ have distinct orbits under an automorphism $\beta$ of $\P_2$. Suppose that a point $P\in L_1$ has the property that its orbit intersects    $L_1$ at infinitely many points and it also intersects $L_2$ and $L_3$. Then the orbit of the line  $L_1$ is finite and at most finitely many points of the orbit of $P$ can intersect $L_2 \cup L_3$.
\end{prop}

\begin{proof}
We shall apply the celebrated Skolem-Mahler-Lech theorem concerning zeroes of linear recurrent sequences. In appropriate homogeneous coordinates, the line $L_1$ is expressed by the equation $z=0$ and the point $P$ has homogeneous coordinates $(1:0:0)$. Letting $T\in\mathrm{GL}_3(\C)$ be a matrix representing the automorphism $\beta$, the condition $\beta^m(P)\in L_1$ reads
\begin{equation*}
u_m:=(0,0,1)\cdot T^m \cdot \left(\begin{matrix}1\\ 0\\ 0\end{matrix}\right)=0.
\end{equation*}
The left-hand side is a sequence satisfying a linear recurrence of the form $u_{m+3}=au_{m+2}+bu_{m+1}+cu_m$, where $a,b,c$ are the coefficients of the characteristic polynomial of $T$, i.e. $\det(T-xI_3)=-x^3+ax^2+bx+c$. 
By hypothesis the above equation has infinitely many solutions $m\in \Z$. By the mentioned Skolem-Mahler-Lech theorem there is an arithmetic progression $n\mapsto hn+k=m$, for suitable integers $k\geq 0$ and $h\geq 1$,  such that for each $m$ in that set, $u_m=0$. This amounts to saying that 
$\beta^k\circ (\beta^h)^n (P)\in (L_1)$ for all $n=0,1,\ldots$, which can be written as
\begin{equation*}
(\beta^h)^n (P)\in \beta^{-k}(L_1)
\end{equation*}
for all $n=0,1,\ldots$. But then, since the orbit of $P$ is infinite, so that $P$ cannot be a fixed point for $\beta^h$, the line $\beta^{-k}(L_1)$ must be invariant under $\beta^h$. This implies  that $L_1$ too is invariant under $\beta^h$ and that the $\beta$-orbit of $L_1$ is finite. The second conclusion of the proposition follows immediately.
\end{proof}

\section{Some remarkable formulae}\label{S.formula} 

In this section we comment about a series of rather surprising formulae  pointed out in the paper \cite{RGK}, and verified subsequently by A. Akopian, M. Bialy, R. Schwartz  and Tabachnikov in private communication to the authors of \cite{RGK} (see that paper for references). Let us start with a brief summary about this, focusing on only one of several similar phenomena.

Suppose that we have a sequence of distinct  points $p_1,\ldots ,p_n\in C(\R)$ defining a periodic billiard trajectory of exact period $n$.  Several    things were known, e.g. on the perimeter of the corresponding polygon: see e.g. \cite{ConnesZagier} or \cite{RGK}. In \cite{RGK} a new property was pointed out: 

\smallskip

{\bf Claim}: {\it Let  $\alpha_i$ be the angle of the trajectory at $p_i$. 
Then, keeping fixed the caustic (tangent to all the segments),   the sum  $\sum_{i=1}^n\cos \alpha_i$ is constant, that is,  it is  independent of $p_1$.} 

\medskip

Our purpose here is twofold: first, we wish to prove this property using the viewpoint of elliptic curves, showing moreover how to obtain in a sense `all' properties of this type. Second, we will see how this viewpoint allows to say something also when the trajectory is not periodic. For instance we shall prove that this sum, again as a function of $p_1$ for a fixed caustic, cannot attain any  value more than twice (taking into account obvious symmetries), which is a best possible conclusion.

The proofs by the authors mentioned above that we have alluded to  allow to analyse also the case of more general billiard curves; however it seems to us that hardly such approach can extend to the case of non-periodical orbits. 

\medskip

Since this topic is not a main one in this paper, we shall proceed somewhat briefly. 

We have remarked above that the function given on the phase space by  $(1-c^2)xv_1+yv_2$ depends only on the caustic. Then it is easy to see that up to a quantity depending only on the caustic, we may replace $\cos\alpha_i$ by the reciprocal of the squared norm of the gradient of $(1-c^2)x^2+y^2$ at $p_i$, i.e. we may replace $\cos\alpha_i$ by $h(p_i)$, where  $h$ is the function of degree $4$ on $C$ given by 
\begin{equation*}
h(p)=h(x,y)={1-c^2\over (1-c^2)^2x^2+y^2}={1\over 1-c^2x^2 }={1\over 2}\left({1\over 1+cx}+{1\over 1-cx}\right).
\end{equation*}
 This function on $C$ has simple poles at the four non-real points $(\pm 1/c, \pm i(1-c^2)/c)\in C(\C)$, denoted $v_i$, $i=1,...,4$,  for this section (we have already considered these points; see formula \eqref{E.4points}). As already noted, these points are significant since the tangent to $C$ at $v_i$, denoted $l_i$ in this section,  is tangent to every caustic.   
 
 As before, we may express $h$ as a function of $z=y/(x-1)$, which provides an isomorphism $z:C\to\P_1$.  The poles become the values $z(v_i)$, i.e. $\pm i(1\pm c)$.  A partial fraction decomposition then yields 
 \begin{equation*}
h={1\over 1-c^2}+\sum  {\epsilon\over c}\cdot {(1+\eta  c)\over z\pm\epsilon  i(1\pm \eta c)},
\end{equation*}
the sum running over all choices of $\epsilon,\eta=\pm 1$.

 Further, we let $h^*$ be the rational  function on $\E_s$ which is the pullback of $h$ via $\pi$, namely, if $\xi:=(p,l)\in \E_s$, so $p\in l$ and $l$ is tangent to $C_s$, and if $\pi(\xi):=p$, we put 
 \begin{equation*}
h^*=h\circ \pi:\E_s\to\P_1,\qquad h^*(\xi)=h^*(p,l)=h(p).
\end{equation*}
The poles of $h^*$ are the points in $\pi^{-1}(v_i)=:\{\nu_i,\nu_i'\}$, say.  Since $l_i$ is tangent to every caustic, we may put $\nu_i=(v_i,l_i)$, whereas $\nu_i'=(v_i,l_i')$, say, with $l_i'$ depending on $s$ and being the other tangent to $C_s$ from $v_i$. 

Let us compute the billiard map $T=\iota^*\circ\iota$ on $\E_s$, when applied to $\nu_i$. Since $l_i$ is tangent to every caustic, it is tangent to $C$, hence $l_i$ meets $C$ only in $v_i$ with multiplicity $2$, hence $\iota$ fixes $\nu_i$: $\iota(\nu_i)=\nu_i$. Therefore $T(\nu_i)=\iota^*(\nu_i)=\nu_i'$, by the very definition of $\iota^*$. Hence the billiard map sends $\nu_i$ to $\nu_i'$  and $\pi^{-1}(v_i)=\{\nu_i, T(\nu_i)\}$. 

\medskip

For an integer $n>0$, let us now put
 \begin{equation}\label{E.H}
H(\xi)=h^*(\xi)+h^*(T(\xi))+\ldots +h^*(T^{n-1}(\xi)),\qquad H:\E_s\to\P_1.
\end{equation}

\begin{proof}[Proof of Claim.] Verifying the above Claim  amounts to prove that $H$ is constant if $T^n=$identity on $\E_s$.

For this, let now $\chi=\chi_s$ be a nonzero differential of the first kind on $\E_s$ and let us consider  the differential $H\cdot \chi$, which is the sum of the differentials $W_r:=(h^*\circ T^r) \chi$, $0\le r\le n-1$.  Since $\chi$ is invariant by $T$ (which is a translation on an elliptic curve corresponding to $\E_s$), we have  $W_r=W_0\circ T^r$. 

Since $\chi$ has no zeros and no poles, the divisor of $W_0$ equals the divisor of $h^*$, hence has simple poles at the points in $\bigcup_{i=1}^4\{\nu_i, T(\nu_i)\}$. Moreover, by the above partial fraction decomposition, the residues of $W_0$ are resp. constant multiples of the residues of the differentials obtained through the  functions of degree $1$ on $C$ which appear in the decomposition,  pulled back to $\E_s$ and multiplied by $\chi$. The four  such functions have  simple poles resp. at $\{\nu_i, T(\nu_i)\}$, and since the sum of the residues is zero, the same happens for $h^*\cdot \chi$.  

Let then $r_i$, resp. $-r_i$ be the residue of $h^*\cdot \chi$ at $\nu_i$ resp. $T(\nu_i)$.  Then $W_r=W_0\circ T^r$ will have simple poles at $T^{-r}(\nu_i)$ and $T^{-r+1}(\nu_i)$ with residues resp. $r_i,-r_i$, for $i=1,2,3,4$.  

Hence we see that the poles of $H\cdot \chi$  occur at at  the $T^s(\nu_i)$,  $s=1,0,\ldots ,-n+1$, with residues which cancel except possibly for $s=1,-n+1$, where anyway the residues are opposite. But if $T^n$ is the identity, then these poles are the same and the residues cancel as well. Hence $H\cdot \chi$ has no poles and is therefore a regular differential on $\E_s$, which must be a constant multiple of $\chi$, proving that $H$ is constant, as wanted.
\end{proof}

\subsubsection{Other formulae?}  

One may ask for other rational functions $g$ on $C$  behaving like $h$, i.e. having the following 

\medskip

{\bf Property}: {\it For a  fixed caustic $C_s$  for which $T$ has finite order $n$ on $\E_s$,  the sum $g(p_1)+\ldots +g(p_{n-1})$ is constant as a function of $p_1$. }

\medskip

The same argument as above shows that in place of $h$ we may tale any linear combination of the functions occurring in its partial fraction decomposition, i.e. constants and the functions $(z\pm\epsilon  i(1\pm \eta c)^{-1}$ on $C$ ($\epsilon,\eta\in\{\pm1\}$).

Actually,  restricting to caustics of {\it even} period $n$  for $T$, there are many other such functions. Indeed, let  $g$ be a (rational)  function   on $C$  which is odd  with respect to the automorphism $p\to -p$ on $C$, i.e. $g(-p)=-g(p)$. Now, note that if $T$ has even period $2m$ on  an  {\it elliptic}  caustic $C_s$, then the points $p_1,\ldots ,p_n$  in a periodic orbit can be grouped into $m$ pairs $p_i,p_{i+m}$, which satisfy $p_{i+m}=-p_i$: indeed, we have that $T^m$ is an automorphism of period $2$, which can be checked to correspond to $p\to -p$ on $C$. (This may be verified e.g. on using the elliptic  picture, in which passing from $i$ to $i+m$ with respect to the caustic $C_s$ corresponds to adding $mB(\lambda)$ in $L_\lambda$, and the above analysis shows that this is a point of order $2$ coresponding to the said automorphism. We note that the case of hyperbolic caustics is slightly more complicated and the   period considered modulo $4$ plays a role.)  Thus $\sum g(p_i)=0$.

\medskip

In general, these functions are essentially the only exceptions.  Indeed assuming for instance that $g$ has only simple poles, and restricting to odd period, we may prove the following assertion:

\begin{prop}[Converse Claim] {\it  Suppose that a function  $g\in \C(C)$, with only simple poles   satisfies  the Property merely for infinitely many caustics, but where we require that the period is odd. Then $g$   is a linear combination of 
the functions $1$ and $z\pm\epsilon  i(1\pm \eta c)^{-1}$ on $C$.}
\end{prop}

We remark that the restriction to odd periods can be eliminated, allowing the other  functions described above,  at the cost of a more complicated argument. Since this relies exactly on the same ideas, and especially since this matter is not the main one on this paper, we confine to the given statement, which illustrates all the principles (and we shall be brief in the proofs).

On the other hand, we remark that the finiteness assertions  at the basis of many  results of this paper (especially Theorem \ref{CMZ}) play a role also in this proof, and this is one more reason for including such result here.  After the proof we shall observe that replacing `{\it infinitely many}' with `{\it at least $K$}' would lead to a false assertion no matter how  large $K$ is. 

\begin{proof}[Sketch of proof]
 Let us consider a given caustic $C_s$ so that the assumptions holds for it, so $T$ is periodic on $\E_s$ of odd (exact) period $n=n_s$.
 
 Let $V$ be the set of poles of $g$ on $C$. Each $v\in V$ lifts to $\le 2$ points $(v,l)\in \E_s$, related by $\iota^*$, whose set we denote by $\tilde V$. These poles are the ones of $g^*:=g\circ\pi$ and they are simple except when $v$ is one of the four ramification points of $\pi$, namely the $P_i$. But this holds only for finitely many caustics at most, so let us assume this is not the case, so each $v\in V$ lifts to exactly two points of $\tilde V$, related by $\iota^*$. 

The  assumption yields   that  the function $G:=g^*+g^*\circ T+\ldots +g^*\circ T^{n-1}$ is constant on $\E_s$. 
 Then, it is not too difficult to perform a similar   analysis as in the previous  proof of the Claim, to  show that a certain cancellation would necessarily occur among the poles of the $(g^*\chi)\circ T^m$, $m=0,\ldots ,n-1$, which entails that the  poles  of $g^*\chi$ may be grouped in sets in the same orbit under $T$, such that moreover the sum of the residues in each such set vanishes. Going to an infinite subset of caustics we may assume that the same grouping occurs and that the corresponding poles of $g$ in each group are the same for all caustics.
 
 Let  $\tilde W$ be one such group, with distinct poles $\nu_1,\ldots ,\nu_r$ of $g^*$, so that $\nu_i=T^{m_i}\nu_1$, where $m_i$ are distinct  integers modulo $n$ depending on $s$ ($m_1=0$), whereas the set $W:=\pi\tilde W$ does not depend on $s$, and the sum of the residues of $g^*\chi$ at these poles is zero. 
 
 We may now interpret this on the Legendre model, through the isomorphism $\psi:\E_s\to \L_s$. In the above notation, we have (recalling  $s=c^2\lambda$),
 \begin{equation*}\label{E.dep_i}
  \psi(\nu_i)=\psi(\nu_1)+m_iB(\lambda),\qquad i=1,\ldots ,r.
 \end{equation*}
 Since the sum of the residues vanishes, we have $r\ge 2$ for each group of points. Since this holds for infinitely many $s$ such that $B(\lambda)$ has finite (odd) order on $\L_s=L_\lambda$, by Theorem \ref{CMZ} we infer that the sections $\psi(\nu_i)-\psi(\nu_j)$ and $B(\lambda)$ are (generically)  linearly dependent for each $i,j\in\{1,\ldots ,r\}$. These sections  correspond to a base which is a cover of the $s$-line unramified except above a certain finite set whereas the minimal  field of definition of $B$ is unramified except above $s=0,c^2,1, \infty$. 
 
 Let us inspect the ramification over $\C(s)$ of  a minimal  field of definition of a section  $\psi(\nu)$, where $\pi(\nu)=v =(v_x,v_y)\in C$. From equation \eqref{E.w}  we see that the ramification occurs when $z(P_0)=z(\gamma \nu)$, where $\gamma$ varies in the group $\Gamma$ of four symmetries of $\E_s$, which lift the symmetries of $C$ obtained from sign changes. A simple calculation yields that the ramification occurs when $\lambda=v_x^2$ or equivalently $1-\lambda=v_y^2/(1-c^2)$.  On the other hand, since the said sections are linearly dependent, we must have that $\psi(\nu_i)-\psi(\nu_j)$ is defined over a field unramified except above $s=0,c^2,1, \infty$. To exploit this information we distinguish among some cases.
 
 Suppose first that in a given group there are two indices $i,j$  such that the corresponding $v_x^2$ are distinct and distinct from $0,1,1/c^2$. Then the corresponding section $\psi(\nu_i)-\psi(\nu_j)$ is defined over a field necessarily ramified above some point outside the said ones, so it cannot be linearly dependent with $B$.  The same holds if one of the $v_x^2$ is distinct from $0,1,1/c^2$ and all the others fall in this set. Therefore we may assume that for all groups and for all pairs of points we have either $v_x^2\in\{0,1,1/c^2\}$ or we have  that the $v_x^2$ are equal. 
 
 Let us then consider the case when in a group $W$ as above  there are points $\nu_1\neq\nu_2$ such that $v_{1x}=v_{2x}\not\in \{0,1,1/c^2\}$. Then $v_1=\gamma(v_2)$ for some $\gamma\in\Gamma$. If $v_1=v_2$ then necessarily $\nu_2=\iota^*\nu_1$, so $\psi(\nu_2)=-\psi(\nu_1)$. Then $\psi(\nu_2)-\psi(\nu_1)=-2\psi(\nu_1)$ is not (generically)  linearly dependent with $B(\lambda)$ since the respective minimal  fields of definition have distinct ramification.  If $v_1\neq v_2$ then $\psi(\nu_2)=\pm\psi(\nu_1)+\tau$ where $\tau$ is a section of order exactly $2$.  If the minus sign holds, then we conclude as before using ramification. If the plus sign holds, then $\psi(\nu_2)-\psi(\nu_1)=\tau$ has exact order $2$, thus there cannot be any $s_0$ (of good reduction) for which it equals a multiple of $B(\lambda_0)$, which has odd order.
 
 So we are reduced to the case when every relevant $v_x$ in a group $W$ is in $\{0,1,1/c^2\}$.  It is easily seen (e.g. with a simple explicit calculation)  that the first two possibilities lead to distinct sections for which  $\psi(\nu)$ has exact order $4$. Therefore the difference 
 $\psi(\nu_2)-\psi(\nu_1)$ of any two distinct ones of them is of (exact) even order, which may be excluded as before. 
 
 The third case instead leads to the poles with $v_x=\pm 1/c$.  The corresponding sections $\nu$ satisfy $2\psi(\nu)=\pm B$.  Indeed, if $v=\pi(\nu)$, we may assume that $\nu=(v,l)$ where $l$ is tangent to {\it all} caustics (this is the distinguished property of the points in question). Hence $\iota$ fixes $\nu$ and therefore $\iota^*(\nu)=T(\nu)$ (we have seen this already in the proof of the Claim). At the level of Legendre curves we have $\psi(\iota^*\nu)=\psi(\nu)+B$. However $[\iota^*]=-1$, which yields the assertion. Now, if a group $W$ contains poles of both types then again we find that the order of $B$ would be even.
 
 But then the function  $g$ has only poles in the said set of four points, and an easy analysis with residues completes the proof. 
 \end{proof}

  \medskip


As a further remark, we note that if we let the function $g$ depend also on $s$, then the space of relevant functions greatly increases, containing e.g. all those of the shape $2f(v)-f(T_C(v))-f(T_C^{-1}(v))$. It may be still of interest to describe more completely this space.
  In any case this shows  that we may construct counterexamples to the conclusion of the Proposition, satisfying however the assumption for any finite set of caustics: indeed, by easy interpolation, given distinct $s_1,...,s_m$, one may construct a function $g(s,p)$ on $\P_1\times C$,  such that $g(s_i,p)$ coincides  with a prescribed function $g_i$ on $C$ (depending on $i$), and we may choose the $g_i$ as above.

\medskip

To go ahead, we remark that, in the opening case of the function $h$ above, we could compute the relevant constant, e.g. integrating with respect to the invariant measure mentioned above, and then obtain another proof of the  complete results of the above cited authors. But instead we show how this approach leads to further conclusions in the case when $T$ is not of finite order.

 \subsubsection{Non periodic orbits}  We now conclude this section by studying the same function (as in the Claim above) $\sum_{i=1}^{n}\cos\alpha_i$, but when the trajectory starting with $p_1,p_2,\ldots$ is not periodic. This issue does not appear in the quoted papers, and certainly the sum cannot be constant on $C$ (in the above sense)  this time. Nevertheless, we shall see that something relevant still can be said, on adopting the `elliptic scheme  viewpoint'. 
 
 By the same remarks as above, for any given caustic $\E_s$, the problem is reduced to the study of the rational function $H$ on $\E_s$ given by \eqref{E.H}.  Indeed, we have remarked that, for a fixed caustic, we have 
 $\sum_{i=1}^n\cos\alpha_i=k(s) H(\xi)$, for a number $k(s)\neq 0$ depending only on the caustic $C_s$ and for $\xi$ a point in $\E_s$ with $\pi(\xi)=p_1$. 
 
 As in the proof of the Claim,  the poles of $H\cdot \chi$  occur at at  the $T^s(\nu_i)$,  $s=1,0,\ldots ,-n+1$, with residues which cancel except possibly for $s=1,-n+1$, where anyway the residues are opposite.  Now we are assuming that $T$ has not finite order (on $\E_s$) so we cannot draw the conclusion that even these last poles cancel. Actually, the converse assertion is true as well. (The same holds if $T$ has finite order not dividing $n$.) 
 
 However  all this says that $H\cdot \chi$ has only eight simple poles, and the same holds for the function $H$ on $\E_s$. In particular, $H$ has degree $8$, is not constant,  and cannot attain any value more than $8$ times.

 Now,  if  $\sum_{i=1}^n\cos\alpha_i$ (as a function of $p_1$ for a given caustic) attains a certain value at $p_1$, then, by the invariance of the map $h$ under the symmetries of the ellipse,  it is easily seen geometrically  that this value is attained by $H$ at one point of $\E_s$ above $p_1$ and above  another point in the same semi-ellipse, namely at the symmetrical of $p_1$ with respect to the $x$-axis. The function $H$ also attains the same value at two further points corresponding to two points of $C$ in the other semi-ellipse. (This is clear also algebraically, say on looking at $\L_s$, since the two points above a given point of the ellipse correspond to changing sign in the elliptic curve, and the four symmetries correspond to addition of points of order $2$.) Then   we also obtain the following

 \begin{prop}\label{P.8}
 If $T^n$ is not the identity on $\E_s$, the sum  $\sum_{i=1}^n\cos\alpha_i$, as a function of $p_1$ (and keeping fixed the caustic),  cannot attain any value more than twice in any semi-ellipse.
 \end{prop}
 
 Note that this function is  `half $C$-periodic',  in the sense that it attains the same values at $p$ and $-p$ (see also below). So, the result is in a sense best-possible: viewed on a semi-ellipse considered as an interval, the function is real and periodic, and non-constant, so the Proposition implies that   it will assume each value at least twice.
 
 \medskip

Let us now see how we can extract even more information. For this it shall be convenient to distinguish between even and odd $n$, and let us say that $n=2m+1$ is odd. We put
\begin{equation*}
H_1(\xi)=H(T^{-m}(\xi))=h^*(T^{-m}(\xi))+\ldots +h^*(T^{-1}(\xi))+h^*(\xi)+\ldots +h^*(T^m(\xi)).
\end{equation*} 
 For a given caustic, this function is a constant times the above sum of the cosines; however considered as a function of $p_{m+1}$, which is the central point in the sequence of $2m+1$ points in the billiard trajectory, we shall see that this function, more symmetrical than the former,  has a very special shape. 
 
  \medskip
  
  Let $D$ be the group of four automorphisms of $\E_s$,  introduced above,  induced by the four natural symmetries of $C$ (they correspond to the translation by points of order $2$ on $J_s$). Note that each element of $D$ commutes with $T$. For $\sigma\in D$, we thus have  $h^*(T^r(\sigma(\xi))=h^*(\sigma(T^r(\xi)))=h^*(T^r(\xi))$, where the last equality follows since $h$ is a function on $C$  invariant by the mentioned four symmetries. Hence $H$ is invariant by $D$. (This gives another explanation of the fact mentioned before.) 

Let also $D_1$ be the group generated by $D$ and $\iota^*$. Note that this group is commutative (as can be very easily checked by direct geometric reasoning), hence isomorphic to $(\Z/(2))^3$. We also note that on the Legendre model $\iota^*$ corresponds to $x\mapsto -x$ so the group is represented by $x\mapsto \pm x+t$ for $2t=0$. 

Now, we have $T\circ \iota^*=\iota^*\circ \iota\circ \iota^*=\iota^*\circ T^{-1}$. Moreover, since $\pi\circ \iota^*=\pi$,  the function $h^*$ is invariant by $\iota^*$.  Therefore $h^*(T^r\iota^*\xi)=h^*(\iota^*T^{-r}\xi)=h^*(T^{-r}\xi)$. This entails that  $H_1$ is also  invariant by $D_1$.

\medskip

Recall from the proof of the Claim that the poles of $h^*$ are simple and consist of the points denoted therein $\nu_i$  ($i=1,...,4$) and $\iota^*\nu_i$. These points $\nu_i$ are distinct and have the property of being  fixed points of $\iota$, so that $\iota^*\nu_i=T\nu_i$. The poles of $H$ turn out to be simple and occurring  at the eigth points $T\nu_i$ and $T^{1-n}\nu_i$, and in turn the poles of $H_1$ occur  at the $T^{1+m}\nu_i$ and $T^{-m}\nu_i$, $i=1,2,3,4$.  (We could show the poles are distinct under the present assumption, but that is in fact not needed.) 

\medskip

Now, $D_1$ acts as a group on the function field of $\E_s$, e.g. over $\C$, the fixed field having therefore index $8$ in $\C(\E_s)$. Since $h^*$ has degree $8$, the fixed field is precisely $\C(h^*)$.  Hence $H_1$ lies in $\C(h^*)$. On the other hand, $H_1$ has degree $\le 8$, hence $H_1$ is a linear fractional transformation of $h^*$:
\begin{equation}\label{E.H_1}
H_1(\xi)=H(T^{-m}(\xi))=h^*(T^{-m}(\xi))+\ldots  +h^*(T^m(\xi))={a_nh^*+b_n\over c_nh^*+d_n},
\end{equation}
where $a_n,b_n,c_n,d_n$ are complex numbers not all zero depending on $n$ and $s$, actually they are real because all functions are defined over $\R$. Note this could give rise to a constant function (if $a_nd_n=b_nc_n$), and this may indeed happen if $T$ has order dividing $n$, as we have seen. The corresponding discussion also shows that these are the only cases. 

\medskip

Note that we may view $H_1$ also as a function $\tilde H_1$ on $C$, because it is invariant under $\iota^*$, hence, setting $p:=\pi(\xi)$, we may write $H_1(\xi)=\tilde H_1(p)$. Recalling the above formula for $h$ and putting $p=(x,y)\in C$, we have
\begin{equation}\label{E.tildeH_1}
\tilde H_1(p)= {a'_nx^2+b'_n\over c'_nx^2+d'_n},
\end{equation}
where $a'_n,b'_n,c'_n,d'_n\in\R'$ again depend (only) on  $n$ and $s$ and are not all zero. This also says that the function $\tilde H_1$ is `a quarter of $C$-periodic', and takes its extremal values for $x=0,\pm 1$, i.e. at the vertices of the ellipse, i.e. the points where $xy=0$. Hence we may improve the previous Proposition  with the following

\begin{thm} If $T$ has not finite order dividing $n$, the function $\tilde H_1$   assumes its maximum and minimum precisely at the  points of $C$ where $xy=0$, and assumes equal values at opposite points. It assumes each value in its range exactly once in each quarter of $C$ (i.e. between any consecutive  two of the said points).
\end{thm}
Clearly this entails a corresponding statement for the sum $\sum_{i=-m}^m\cos \alpha_i$. 

The case of even $n=2m$ is similar, just a little more laborious: it suffices to observe that we may write $T=U^2$ for an automorphism $U$ of $\E_s$ and then the sum once symmetrized becomes $h^*(U^{-m}(\xi))+\ldots +h^*(U^{-1}(\xi))+h^*(U(\xi))+\ldots +h^*(U^m(\xi))$.  We omit the verifications for brevity.

It should also be possible  to express the numbers $a_n,...$ in terms of significant quantities, as has been done for the original of the mentioned authors, but we have not  performed this analysis. 

We suspect that the maximum (or minimum)  is not attained always at the same points, but that it is attained at vertices (on the lines $xy=0$) which alternate (finitely many times) depending on the caustic. However we have not proved this nor made particular effort, so that we do not express any opinion on how difficult this could be.

\medskip

On the other hand, we do not know whether these results admit (simple) proofs not using the elliptic description.

\bigskip

\section{Final remarks}

It is clear how Theorem \ref{T.P_2} is related to the Question formulated in the Introduction, and indeed this represents maybe the simplest issue of it.
Let us now illustrate   how also some of our finiteness theorems on billiards  enter into this frame. 

\smallskip

{\tt Theorem \ref{T.buca} and the Question}. Consider for instance Theorem \ref{T.buca}. The algebraic surface mentioned in the question will be the billiard elliptic surface $\mathcal{X}$, which, we recall, is birationally isomorphic to the product $C\times C \simeq \P_1\times \P_1$, hence also to the plane $\P_2 $. The group of endomorphisms $\Gamma$ is the cyclic group generated by the billiard map $\beta$, viewed as an automorphism of the surface $\mathcal{X}$. The curves $L_1$ (resp. $L_2,L_3$) are defined by the pairs $(x_1,x_2)\in C\times C \approx \mathcal{X}$ such that the line joining $x_1,x_2$ passes through $p_1$ (resp. through $p_2$, through $h$). Theorem \ref{T.buca} may be stated as asserting  that there are only finitely many points of $L_1$ whose orbit under $\Gamma$ intersects both $L_2$ and $L_3$.

Note that, since the surface $\mathcal{X}$ is rational, the automorphism $\beta$ can also be viewed as a rational automorphism of the plane $\P_2$ (a Cremona transformation). As we already remarked, the  methods of proof in these two cases are very different.
\smallskip

We have seen in Proposition \ref{P.P_2}, and in the examples at the end of paragraph \ref{S.P_2}, that, in the case of the plane $\P_2$ and a linear automorphism, the hypothesis that the lines belong to distinct orbits does not guarantee the finiteness of the set of orbits intersecting all of them. Similar counter-examples arise on elliptic surfaces.

Consider for example the case in which $L_1$ is the image of the zero section and the curves $L_2,L_3$ are distinct torsion curves  (images of algebraic torsion sections); taking for $\beta$ the autorphism of the surface induced by translation with respect to another (non-torsion), section we obtain that  infinitely many  torsion  points for $\beta$ have an orbit intersecting all the three curves. 

\smallskip

{\tt Another example of infinitude of orbits intersecting three curves}. We can produce another example,  still concerning elliptic surfaces and related to  Example \ref{Ex.u=-v} (involving  the projective plane).

Let again $\mathcal{X}$ be an elliptic surface with a section $\beta$ of infinite order, let $L_1$ be the zero section and take another arbitrary curve $L_2$ on $\mathcal{X}$. Let $F:\mathcal{X}\to \mathcal{X}$ be the automorphism sending $P\to -P$ (with respect to the group law on the fibers of the elliptic surface). Put $L_3=F(L_2)$. Denoting again by $\beta$ the translation map induced by the section $\beta$, we observe that,   as it was the case in Example \ref{Ex.u=-v}, $F\circ \beta\circ F^{-1}=\beta^{-1}$, and $F$ induces the identity on $L_1$. 
Then for every point $P\in L_1$ and every integer $m\in \Z$ such that $\beta^m(P)\in L_2$, it holds $\beta^{-m}(P)\in L_3$.

\smallskip

{\tt Further links with the Dynamical Mordell-Lang Conjecture}. Let us describe  more formally the link between our results and the so-called Dynamical Mordell-Lang Conjecture, for which we refer  to the book \cite{BGT} by J. Bell, D. Ghioca and T. Tucker. It turns out that our results are instances of a case of the Dynamical Mordell-Lang problem for an automorphism group of rank two.

Let us consider again our Theorem \ref{T.P_2}. To insert this result  into the frame of the Dynamical Mordell-Lang Conjecture, let us consider the algebraic four-fold $\mathcal{X}\simeq \P_2\times \P_2$ parametrizing pairs of lines on the plane. Given an automorphism $\beta$ of the plane, we define an action on $\mathcal{X}$ by the commutative group $\Z^2$ by setting 
\begin{equation*}
\Z^2\times \mathcal{X} \ni ((m,n),(L,L'))\mapsto (\beta^{-m}(L),\beta^{-n}(L'))
\end{equation*}
Given a line $L_1$, consider the hypersurface $\mathcal{Y}_{L_1}=\mathcal{Y}\subset\mathcal{X}$ formed by the pairs $(L,L')\in\mathcal{X}$ such that $L\cap L'\cap L_1\neq\emptyset$. 

Finally, fix two more lines $L_2,L_3$, so that the pair $(L_2,L_3)$ is a point of $\mathcal{X}$.

Given a pair $(m,n)\in\Z^2$, the existence of a point $P\in L_1$ such that $\beta^m(P)\in L_2$ and $\beta^n(P)\in L_3$ amounts  to the condition that $(\beta^{-m}(L_2),\beta^{-n}(L_3))\in \mathcal{Y}$. Hence, the problem treated in Theorem \ref{T.P_2} is equivalent to that of  describing the pairs $(m,n)$ such that the corresponding image of the point $(L_2,L_3)$ lies on the closed proper subvariety $\mathcal{Y}$ of $\mathcal{X}$. 

No general result seems to be known in the context of $\Z^2$-actions; to our knowledge, differently from the one dimensional case of $\Z$-actions,  no general conjecture has been formulated so far.

In the specific example just described, we proved that generically   such pairs $(m,n)$ are finite in number, while in particular cases we found infinite families which either consist  of lines in $\Z^2$ or in `exponential families' of the form $(m,a\lambda^m+bm+c)$, for fixed $a,b,c,\lambda$. 

\medskip

{\tt Dynamical viewpoint on Theorem \ref{T.angolo}}. Concerning  Theorem \ref{T.angolo}, we can view it as a statement about finiteness of periodic points in a subvarieties. Here are the details. Consider the four-fold $\mathcal{X}\times \mathcal{X}$, parametrizing pairs of segments of a billiard trajectory. It is endowed by the diagonal automorphism $(\beta,\beta)$. The pairs $(x,y)$ corresponding to shots from a given point $p_0$ forming a given angle $\alpha$ form a curve $L\subset \mathcal{X}\times \mathcal{X}$. If the two shots $x,y$ are both periodic for $\beta$, then so is the pair $(x,y)$ for $(\beta,\beta)$. Hence Theorem \ref{T.angolo} asserts the finiteness of periodic points lying on $L$. Note that the periodic subvarieties of $\mathcal{X}\times \mathcal{X}$, i.e. those formed by the fixed points of the itarates of $\beta$, are two dimensional (product of two curves in $\mathcal{X}\times\mathcal{X}$, so the expected condition for a subvariety $L\subset \mathcal{X}\times\mathcal{X}$ to contain infinitely many periodic points is that $\dim L \geq 2$. 

\medskip

{\tt Links with other issues}. A very recent work by S. Cantat and R. Dujardin \cite{CantDu} studies orbits for a group of automorphisms of surfaces. Possibly the present methods can be applied in some cases to deduce the finiteness of periodic orbits  outside exceptional cases, which could be classified. A first case to treat might be that of a double elliptic fibration on a surface.
\bigskip

\section{Appendix  - with the collaboration of Julian Demeio}

In this Appendix we shall prove in particular Theorem \ref{T.esiste}. But we shall develop several other results, for completely general sections of elliptic schemes, both in the real and the complex case.

For the case of the real billiard we shall give quite explicit formulae for the constants which express the asymptotics. We remark that, for this issue,  some similar analysis has been carried out in the book \cite{Du} (where a special attention is given to the s-called QRT maps which we do not consider here). However the present treatment is different in several respects: it is direct and  essentially  self-contained and develops formulae in terms of elliptic integrals which we have not found in the existing literature. 

\medskip

\subsection{Proof of Theorem \ref{T.esiste}}\label{SS.esiste} We start by recalling very  briefly the notion of Betti map (for which see especially  \cite{CMZ} and \cite{ACZ}). 

\subsubsection{The Betti map of a section}\label{BettiMap}  Let $\pi:\Aa\to B$ be an elliptic scheme over an (affine) complex smooth curve $B$, and let $\sigma:B\to\Aa$ be a section. Locally for $b\in B(\C)$ we may represent the elliptic curve $\Aa_b\cong \C/\Lambda_b$ analytically as a complex torus, where  $\Lambda_b=\Z\omega_1(b)+\Z\omega_{2}(b)$ is  a lattice and the {\it periods} $\omega_{i}$ vary locally holomorphically on $B$.  Again locally on  disks $U\subset B$ we have exponential maps $\exp_b:\C\to\Aa_b:=\pi^{-1}(b)$ varying holomorphically on $U$, and we may take an elliptic logarithm $\tilde\sigma$ of the section, so $\tilde\sigma:U\to\C$ is holomorphic and of the shape $\tilde\sigma(b)=\beta_1(b)\omega_1(b)+\beta_2(b)\omega_2(b)$, for real-valued real-analytic functions $\beta_i$ on $U$. These are called ``Betti coordinates'' of $\sigma$ and the map $b\mapsto (\beta_1(b),\beta_2(b))$ is called  ``Betti map'' of $\sigma$.  Of course this  holds only locally, the map is determined only up to the addition of integer constants, and  there are monodromy transformations if we perform analytic continuation.  Also, this may be done more generally for schemes of abelian varieties over a complex algebraic base variety. See \cite{Z3}, \cite{CMZ}, and \cite{ACZ} for much more on this. 

\medskip

{\tt Betti map on real points}.  Suppose that $\Aa$ is defined over $\R$, so for $b\in B(\R)$ the group $\Aa_b(\R)$ is not empty and has one or two connected components.  Let  $\sigma$ be a section defined over $\R$ and let $b_0\in B(\R)$,  so  $\sigma(b_0)\in\Aa_{b_0}(\R)$.      In some neighbourhood $U\subset B(\C)$ of $b_0$, the  values of an elliptic logarithm $2\tilde\sigma$  of $2\sigma$ at points $b$ of $U(\R)$, where $\sigma$ is real,  will be of the shape 
$(2tr+m)\omega_1(b)+(2ts+n)\omega_2(b) $ for some real $t=t(b)$ and integers $r,s,m,n$, necessarily constant, by continuity, in a neighbourhood  of $b_0$ which we may assume to be $U$.  Hence the Betti map of $\sigma$ restricted to real points in the neighbourhood will be of the shape $b\mapsto (rt(b)+{m\over 2},st(b)+{n\over 2})$, for some real-analytic function $t$, hence mapping to a segment of a rational line in $\R^2$.

Note that we may always choose the periods $\omega_1, \omega_2$ in such a way that $\overline{\omega_{2}(b)}=\omega_2(b)$ for $b \in U(\R)$. In this case, with the notation above, we may assume that $s=1, r=0$, i.e. the Betti map takes the shape $b \mapsto (\frac{m}{2},t(b)+\frac{n}{2})$ on $U(\R)$. We will then refer to the Betti coordinate associated to the period $\omega_2$ on $U$ as the {\em real} Betti coordinate. The function $U(\R) \ni b \mapsto t(b)+\frac{n}{2}$ will  be referred to as the {\em real} Betti map.

For the Legendre curve things may be described even more precisely. Let us assume that $\lambda$ is real. Then we have recalled that the corresponding lattice $\Lambda_\lambda$ is generated by a purely imaginary period $\omega_1$ and a real period $\omega_2$, that we have expressed explicitly in the region $0<\lambda<1$.

\medskip

Recall now that we have the elliptic scheme $\pi:\L\to\P_1$. 
Namely, a point $(p,l)$ of $\E$ yields a caustic $C_s$ to which $l$ is tangent,   a corresponding point $(p,v)$ on the phase space, and finally a point on $\L$ obtained on applying the isomorphism $\phi$ above.  Here a parameter on $\P_1$ is given by $s^2$, which determines the caustic, so $\pi\phi((p,v))=s^2$ (see also the third of equations \eqref{E.Ys}), and the fibers of $\pi$ are the $\L_s$ (which really depend on $s^2$). 
The billiard map determines 
a section of $\L$, however the base should be extended to the curve with function field $K(s,\sqrt{s^2-1},\sqrt{s^2-c^2})$, $K=\Q(c,\sqrt{c^2-1})$,  after removing the points with $s=0,\pm c,\infty$.   

\smallskip

{\tt Betti map of the billiard section}.  Let $B(\lambda)$ denote the billiard section expressed on the Legendre curve, as in equations \eqref{E.bmapX} and \eqref{E.bmapY}, where $0<s<1$ and we may choose the positive sign and the positive square root in the second formula. In order to express the elliptic logarithm, and hence the Betti map, we shall use \eqref{E.masser}, where the sign is found to be positive (since we have chosen the  positive sign and square root in \eqref{E.bmapY}).

Suppose first that $0<\lambda<1$, so $0<s<c^2$ and  the caustic is a hyperbola.  The Betti map of $(\lambda,0)$ is $(1/2,1/2)$ (up to integer points), since $(\lambda,0)$ corresponds to $(\omega_1+\omega_2)/2$ by the discussion in \S\ref{SSS.torsion}.   The section with constant abscissa $1/c^2$ takes values in the connected component of the identity in $L_\lambda(\R)$, hence its Betti map takes values in  $\Z\times \R$. Also, 
   by the discussion in \S\ref{SSS.weier} an elliptic logarithm of the first section is given by 
$-(1/2)\int_{1/c^2}^\infty \d x/\sqrt{x(x-1)(x-\lambda)}$, where we choose the positive sign of the square root, and where the minus sign is due to the fact that $\wp_\lambda'(\mu)$ is negative in the interval $(0,\omega_2/2)$ whereas $k(\lambda)>0$. Hence, by the second equation in \eqref{E.periods}, and denoting
\begin{equation}\label{E.int}
I_u(\lambda)=  \int_{u}^\infty{\d x\over \sqrt{x(x-1)(x-\lambda)}}, \qquad u\ge 1,
\end{equation}
 the Betti map of the billiard section for $0<\lambda<1$ is given by

\begin{equation}\label{E.betti1}
\beta(\lambda)=(\beta_1(\lambda),\beta_2(\lambda))=\left({1\over 2} , {1\over 2}   -{  I_{1/c^2}(\lambda)\over 2I_1(\lambda)   }
       \right).
\end{equation}

We prove that the function on the right is monotonic increasing  in $\lambda$.  In fact, it suffices to show that $I_{1/c^2}(\lambda)'I_{1}(\lambda)-I_{1/c^2}(\lambda)I_{1}(\lambda)'<0$ where the dash denotes derivative with respect to $\lambda$. The derivative  $I_u'(\lambda)$ is obtained by a similar integral, where however the integrand is multiplied  by $\alpha(x):=(2(x-\lambda))^{-1}$. Then, denoting  for this argument by $f(x)$ the integrand expressing $I_u(\lambda)$, we have 
\begin{equation*}
I_{1/c^2}(\lambda)'I_{1}(\lambda)-I_{1/c^2}(\lambda)I_{1}(\lambda)'=\int\int_A f(x_1)f(x_2)(\alpha(x_1)-\alpha(x_2))\d x_1\d x_2,
\end{equation*}
where $A=(1/c^2,\infty)\times (1,\infty)=((1/c^2,\infty)\times (1/c^2,\infty))\cup((1/c^2,\infty)\times (1,1/c^2))=A_1\cup A_2$, say.  The integral over $A_1$ vanishes since the integrand is anti-symmetric. The integral over $A_2$ is negative since $\alpha$ is a decreasing function and $f$ is positive.

For $\lambda$ tending to $1$, $I_1$ diverges whereas $I_{1/c^2}$ remains bounded so  $\beta$ tends to $(1/2,1/2)$. For $\lambda$ tending to $0$, both integrals converge and we may pass to the limit under the integral sign, so for $u\ge 1$,  $I_u(\lambda)\to \int_u^\infty\d x/(x\sqrt{x-1})=2\int_{\sqrt{u-1}}^\infty\d z/(z^2+1)=\pi -2\arctan\sqrt{u-1}$.  Therefore $\beta$ tends to $(1/2, \pi^{-1}\arctan\sqrt{(1/c^2)-1})$. 

\medskip

We have already commented the behaviour of the billiard map in the case $\lambda=1$, i.e. $s=c^2$; this is a degenerate case, when the caustic degenerates into the segment connecting the foci. Any billiard shot passing through one focus will give rise to a sequence of segments passing alternatively through the foci, and tending to the horizontal segment, without reaching it unless the whole trajectory is horizontal, and periodic of period $2$. This also explains the `$1/2$' in the formula. (Note however that the geometrical intuition is not equally effective when $\lambda=0$.)

\medskip

Let us now consider the case $1<\lambda<1/c^2$, when the caustic is an ellipse.  Now the Betti map of $(\lambda,0)$ is $(0,1/2)$ (still by \S \ref{SSS.weier}).  Again, the section with constant abscissa $1/c^2$ takes values in the connected component of the identity in $L_\lambda(\R)$, hence its elliptic logarithm may be taken in $\R\omega_2(\lambda)$. This logarithm may be again expressed by  $(-1/2)I_{1/c^2}(\lambda)$, by similar considerations as before.

The period $\omega_2$ now equals $\int_\lambda^\infty \d x/y$; however for computing the derivative we prefer to use the alternative formula
\begin{equation*}
\omega_2(\lambda)=\int_\lambda^\infty{\d x\over y}=\int_0^1{\d x\over y},\qquad 1<\lambda,
\end{equation*}
which follows either by substitution $x\mapsto \lambda/x$ or by looking at the corresponding integrals on a representative torus, or by observing that the connected components of real points on the torus  are homologous, thus lead to equal integrals.

The same considerations as above then show that the Betti map of $B(\lambda)$ is given by 
\begin{equation}\label{E.betti2}
\beta(\lambda)=(\beta_1(\lambda),\beta_2(\lambda))=\left(0, {1\over 2}   -{  I_{1/c^2}(\lambda)\over 2\int_0^1{\d x\over y} }
       \right)=\left(0,{ \int_\lambda^{1/c^2}{\d x\over y} \over 2\omega_2(\lambda)}\right).
\end{equation}
With the same notation as above now the derivative of the ratio $I_{1/c^2}/\omega_2$    is found to be positive because equal to $\int\int_Bf(x_1)f(x_2)(\alpha(x_1)-\alpha(x_2))\d x_1\d x_2$, where now $B=[1/c^2,\infty]\times [0,1]$ and where now $\alpha(x_1)\ge 0$ whereas $\alpha(x_2)\le 0$ for all relevant values. 

The limit of $\beta(\lambda)$ for $\lambda\to 1^+$ is (for the same reason as before) $(0,1/2)$, whereas for $\lambda\to {1/c^2}^-$ we have  easily $I_{1/c^2}\omega_2^{-1}\to 1$, hence the limit  of $\beta$ is $(0,0)$. 

Of course one may find various other expressions for these functions, e.g. power series expansions, to approximate  their values rapidly.

\medskip

{\tt Betti map and rotation number}.  We briefly point out an interpretation of the Betti (billiard) map as a rotation number; for this notion we refer to C. Yoccoz's paper in the volume \cite{Wald}. 

For simplicity let us consider only the case when the caustic is an ellipse. So, fix such a caustic  $C_s$, $c^2<s<1$,  and consider the map $f=f_s:C\to C$ defined as follows. For $x\in C$ there are two tangents from $x$ to $C_s$. Choose then the one meeting $C$ in a point $y\neq x$ which comes first on travelling $C$ from $x$ in the clockwise direction.  We put $f(x):=y$. (This may be clearly expressed also on using the billiard map on $Y_s$, but the present definition is more direct. Note also that referring to $Y_s$   would not lead to an algebraic notion because of the orientation.) 

Now, it is clear that $f$ is a bijective map $C\to C$. After identifying $C$ with the circle $S_1$, $f$ may be thought of as a topological automorphism of the circle, and we may iterate it and consider its {\it rotation number}. We only recall from the quoted article that this may be defined as the supremum of the set of fractions $p/q$ such that the   iterates $f^{\circ m}(x)$, $0\le m\le q$,  locate a sequence on $C$ making $p$ tours  through $C$.  If we now think of $Y_s(\R)$ as the  identity component of the real points on an elliptic curve, and of the billiard map as a translation on $Y_s(\R)$,  it is immediate to realize that this rotation number equals $\beta_2(\lambda)$. We leave the easy verifications to the interested readers. 

With this interpretation for instance  it becomes {\it a priori} clear that $\beta_2$ is a decreasing function for $1<\lambda<1/c^2$, and that it tends to $0$ at the upper extreme. Indeed, as $\lambda$ grows the corresponding caustics strictly  increase, hence the values $f_s(x)$ decrease in $s$ for any given point $x$.  When $s\to 1^-$ the caustic $C_s$ approaches the ellipse $C$, so $f_s$ tends to the identity and the rotation number tends to $0$. 
  
\bigskip

\begin{proof}[Proof of Theorem \ref{T.esiste}] We start by recalling an elementary euclidean argument for the first existence assertion. This is clear for $n=1$ so suppose $n\ge 2$. Consider all sequences $x_0:=p_1, x_1,..., x_{n-1}, x_n:=p_2$, where $x_1,\ldots ,x_{n-1}$ lie on $C$.  By compactness there is a choice so that the total length $|p_1-x_1|+|x_1-x_2|+\ldots +|x_{n-2}-x_{n-1}|+|x_{n-1}-p_2|$ (of the  piecewise linear trajectory with $n$ segments and ordered vertices in the sequence) is maximal.  We contend that this is a billiard trajectory. Indeed, let us first assume $1<j<n-1$ and consider  the segments $x_{j-1}-x_j$ and $x_j-x_{j+1}$, that is those not containing $p_1$ or $p_2$.  Consider the line through $x_j$ which   forms equal angles with these segments and has both  $x_{j\pm 1}$ in the same half-plane; if this is not tangent to $C$ at $x_j$ then it meets the ellipse at a point $x\neq x_j$   strictly  between $x_{j-1}$ and $x_{j+1}$.  But then by the Fermat-H\'eron principle, we would  have $|x-x_{j-1}|+|x-x_{j+1}|>|x_j-x_{j-1}|+|x_j-x_{j+1}|$, contradicting maximality.  Similarly if $j=1$ or $n-1$. This proves the claim  except possibly if $x_j=p_1=p_2$ (because maximality is intended with $p_1,p_2$ given). But, taking into account the theorem on caustics,  the argument proves that all segments are tangent to a same caustic, hence the reflexion law holds also at $p_1$. 
	
	Note that (on taking $p_1=p_2$ and $n$ a prime)  this yields another proof of Proposition \ref{P.nontors}.
	
	\medskip
	
	Let us now take a point $p=(a,b)\in \T$ and prove the stated estimates, indicating at the same time how to obtain   formulae for the constants $c_{o}, c_e$. For brevity we treat in full only the case when $p\in\T^o$ and $b>0$, $0<a<c$.  
	
	Let $\xi\in [-\infty,+\infty]$ be the slope of a billiard shot from $p$. By symmetry we can consider only the case when the shot hits $C$ on the right of $p$.  The line $\ell:y=\xi(x-a)+b$ from $p$ will be tangent to the caustic $C_s$ where
	\begin{equation*}
	s={(\xi a-b)^2+c^2\over \xi^2+1}.
	\end{equation*}
	This function of $\xi$ tends to $a^2$ at both $\pm\infty$, is increasing from the left until it reaches its maximum $M=M(a,b,c)$, then decreases until its minimum $m=m(a,b,c)$, and finally increases again indefinitely to the right. Since we are in the case $0<a<c$, for $\xi$ near to $-\infty$ the caustic will be a hyperbola. This will continue until $\ell$ will hit the focus $(0,c)$, when $\xi a-b=\xi c$ and $s=c^2$. In the interval $({-b\over c-a},{b\over c+a})$ the caustic will be an ellipse, and will return to be a hyperbola on the whole right of it. 
	
	We note that  easy geometry shows that the maximum will be attained when $C_s$ is an ellipse passing through $p$, and $\xi<0$, whereas $m$ will be attained when $C_s$ is a hyperbola through $p$, and $\xi>0$, in both cases $\ell$ being tangent to $C_s$ at $p$.  Of course $M,m$ can be very easily expressed explicitly but we omit the not very simple formulae. 
	
	Now let $n>0$ be an integer and suppose that the billiard shot corresponding to $\ell$ has period (dividing) $n$. This will happen if and only if the billiard map corresponds to a torsion point of order $n$ on $\L_s$, or if and only if the value  $\beta(\lambda)$ of the Betti map, where $c^2\lambda=s$, is rational with denominator (dividing) $n$. 
	
	Now suppose first that $n$ is odd. Then by \eqref{E.betti1}, the caustic cannot be a hyperbola, hence we confine our attention to the elliptic caustics, where we require that $n\beta_2(\lambda)$ is integer in \eqref{E.betti2}. By the above considerations this will corespond to the $\xi$ in the interval when the caustic is an ellipse. Since all involved functions are continuous and never locally constant it follows that the number of relevant values is $2|\beta_2(M/c^2)-\beta_2(1)|\cdot n+O(1)$. Thus, taking into account the above mentioned symmetry, and recalling $\beta_2(1)=1/2>\beta_2(M/c^2)$, we have
	\begin{equation*}
	c_o=2-4\beta_2\left({M\over c^2}\right).
	\end{equation*}
	We also note that this value is constant for points $p'$ on the elliptic caustic through $p$.
	
	If $n$ is even things are similar but we have to consider also hyperbolic caustics. By an easy argument as above the relevant value is found to be
	\begin{equation*}
	c_e=2\left(1-\beta_2\left({M\over c^2}\right)-\beta_2\left({m\over c^2}\right)\right).
	\end{equation*}
\end{proof} 

Proposition \ref{P.reali1} below will provide another approach to the proof of Theorem \ref{T.esiste}, which also applies to the existence part in Theorem \ref{T.ritorno}.

\begin{rem}\label{R.irrational}  (i)  {\tt About elementary formulae and rational  values}.  It is to be remarked that the values so obtained for $c_o,c_e$ are not expressible by elementary functions (in the classical sense) of the parameters, at any rate for {\it generic} values of them. This may be proved from the formulae \eqref{E.betti1} and \eqref{E.betti2},  or from the formulae in Remark \ref{R.maninb},  using for instance the theory for the differential equations satisfied by the periods and similar  functions, which are known not to be satisfied by elementary functions. We cannot pause more on this issue here, and refer to the article by F. Beukers  in \cite{Wald}. In any case we point out that  $M,m$ are not constant (their level curves are resp. elliptic and hyperbolic caustics) and are algebraically independent, so we may consider them as independent variables.
	
	Also, the function $\beta_2$ is not rational but  when $c\in\Q$ it is  a ratio of functions from a finite dimensional space (over $\C$) spanned by $G$-functions (as is not difficult to prove). Then it may be proved that  the {\it rational} values of $c_o,c_e$ at rational points $p$ are subject to (severe) restrictions. This follows e.g. from  results in the paper \cite{DZ} of P. D\`ebes with the second author.\footnote{Known theorems of Bombieri-Pila - see Appendix A in \cite{Z} - can also yield some results, however subject to restrictions on heights.}  In particular, we may state the following conclusion:
	
	{\it  For a given ellipse $C$ with $c\in\Q$, there are infinitely many rational points $p\in\T^o$ such that $c_o,c_e$ are both irrational}.
	
	Note that if e.g. $p\in C$ then  $c_o=2$, hence $\T^o(\Q)$ cannot be replaced with $C(\Q)$. We leave the formal proof of this assertion to the interested readers. 
	
	(ii)  {\tt Bicyclotomic polynomials}. In view of formula \eqref{E.masser} the values of $\lambda$ which make torsion the billiard section are the same which make torsion the section with constant abscissa $1/c^2$. Hence these values are the real roots of the polynomials studied in the paper \cite{MZb} of Masser  and the second author, and called {\it Bicyclotomic} therein.  The above formulae give therefore estimates for the number of these real roots. 
\end{rem}

\subsection{Asymptotic estimates on torsion values for general sections}

We put ourselves in the context of Section \ref{BettiMap}, and we assume that
we have two algebraic sections $\sigma,\tau$ of $\Aa \to B$, both defined over $\R$. We prove

%

\begin{prop}\label{P.reali1}
	If the algebraic section $\sigma$  is not torsion, then for  large enough integer $N$, the set of points $x$ of $B(\R)$ such that $\tau(x)=N\sigma(x)$ is  non empty and for varying $N$ is dense in $B(\R)$. Actually, in any neighbourhood $I$ of a $b_0\in B(\R)$ the number of such points is $c\cdot N + O(1)$, for a $c=c(I)>0$ independent of $N$.
\end{prop}

\begin{proof}[Proof of Proposition \ref{P.reali1}] The proof is conceptually very easy, however near the points of bad reduction of the elliptic scheme one needs some control of the section and the Betti map. This could be dealt with directly, but here we proceed on invoking some results on definability of the relevant Betti maps.

	Consider the (real) Betti maps $\beta_\sigma,
	\beta_\tau$ associated to the two sections, well-defined as above on an open disk $U\subset B$, supposed to contain $b_0\in B(\R)$. We may assume that $U \cap B(\R)=I$. Let $N$ be a large integer  and consider $\beta_\sigma-N^{-1}\beta_\tau$ on $U$. 
	
	The function $\beta_\sigma-N^{-1}\beta_\tau$, evaluated at  a point $x \in I$ takes a rational value with denominator dividing $N$ precisely when $\tau (x)= N\sigma(x)$. We call this number $A_N$.
	
	We need now the following definition and three observations.

	For a continuous function $f:I \rightarrow \R$, we say that the \emph{monotone number} of $f$ is the minimum integer $r$ (possibly $\infty$) such that there exists a partition of $I=[a,b]$ in $r$ intervals $I=[a,a_1]\cup \cdots \cup [a_{r-1},b], \  a \leq a_1 \leq \cdots \leq a_{r-1} \leq b$, and such that $f$ is (weakly) monotone on each of these $r$ intervals.
	
	\underline{First observation}. Let $N$ be a natural number, and let $f:I \rightarrow \R$ be a continuous, piece-wise differential function of monotone number $r < \infty$. Then the number of $x \in I$ such that $f(x)=\frac{m}{N}$, for some $m \in \Z$, is $N \int_I |\d f|+O(r)$, where the big $O$ is absolute.
	
	\underline{Second observation}. Let $f(x,y):I \times [0,1] \rightarrow \R$ be a definable function (by definable we will always mean definable in $\R_{an,exp}$, we refer to \cite[p. 16]{tametopology} for the definition of definability and $o$-minimal models). Then the monotone number of $f(x,\epsilon):I \rightarrow \R$ is uniformly bounded for $\epsilon \in [0,1]$. In fact, letting $\tilde{f}=(f,id_\epsilon): I \times [0,1]\rightarrow \R \times [0,1]$, using Hardt's Theorem (\cite[Theorem 9.1.2]{tametopology}), one shows that there exists a finite cell decomposition $I \times [0,1]=C_1 \cup \cdots \cup C_R$ such that $\tilde{f}|_{C_i^o}$ (where $C_i^o$ denotes the open part of $C_i$) is continuous and it is either injective or of the form $(c,id_\epsilon)$, where $c$ is a constant function. In particular, for a fixed $\epsilon \in [0,1]$, $f(x,\epsilon)$ is weakly monotone on the $R$ (possibly degenerate) intervals $C_1|_{I \times \{\epsilon\}}, \ldots, C_R|_{I \times \{\epsilon\}}$.
	
	\underline{Third observation}. For a piece-wise differentiable function $f:I \rightarrow \R$ of monotone number $r$, we have that $\int_I |\d f| < r \lambda (f(I))$, where $\lambda$ denotes the Lebesgue measure. In particular, a bounded function with bounded monotone number satisfies $\int_I |\d f| < \infty$.
	
	Now, by a result of Jones and Schmidt \cite{JS}, $\beta_\sigma$ and $\beta_\tau$ are bounded definable functions on the domain $I$. In particular, by the second observation above, the monotone number of these functions is bounded, and we deduce by the third observation that $\int_I |\d (\beta_\sigma)| < \infty , \ \int_I |\d (\beta_\tau)| < \infty.$
	
	Now, as $N \to \infty$, we have:
	$$
	\frac{A_N}{N}=\int_I |\d (\beta_\sigma-N^{-1}\beta_\tau)| +O(1/N),
	$$
	where the equality follows from the first observation.
%
%

	Moreover, since  $\sigma$ is not a torsion section, its Betti map $\beta_\sigma$ is not constant, by a (special case of a) theorem of Manin (see  \cite{Man}, \cite{CZPonc} or \cite{ACZ} for results in higher dimensions). 
	In particular, the finite integral $c=\int_I |\d (\beta_\sigma)| \neq 0$ is non-zero, as desired.
\end{proof}

 \begin{small}
We recall a self-contained argument for the  assertion alluded in the proof, actually for then general case of $\R^d$, which would be useful for proving a complex analogue. {\it Let $f:U\to\R^d$ be a $C^1$-map  from a ball  $U\subset \R^d$ centred at $0$. Assume that for $x,y\in U/2$,   it satisfies  $|\d f(x)^{-1}|\le c$, where $c\ge 1$,   and that  $f(x+y)=f(x)+\d f(x)y+k(x,y)$, where $|k(x,y)|\le  (2c^2)^{-1} |y|$. Then  we assert that  there exists an open ball  $V$ depending only on $U,c$ such that $f(U)$ contains $f(0)+V$.}  
 
 For a proof one can follow  Newton's method. Let $V$  be a disk of small enough radius $r$ centred at $0$;   given $v\in f(0)+V$,  define a sequence $x_n$ as follows. Put $x_0=0$ and, having defined $x_n$, let $x_{n+1}:=x_n+\d f(x_n)^{-1}(v-f(x_n))$. Setting $x_{n+1}=x_n+y$, note that $f(x_{n+1})=f(x_n)+\d f(x_n)y+k(x_n,y)=v+k(x_n,y)$. Hence $|f(x_{n+1})-v|\le |(2c^2)^{-1}|\d f(x_n)^{-1}(v-f(x_n))|\le (2c)^{-1} |f(x_n)-v|$, so $|f(x_n)-v|\le r(2c)^{-n}$ by induction. Note also that $|x_{n+1}|\le |x_n|+c|v-f(x_n)|\le |x_n|+cr(2c)^{-n}$.   Hence the sequence $x_n$ lies in $U/2$ if $r$ is small enough and converges  to a solution of $f(x)=v$.
 \end{small}

\noindent {\it Proof of the existence part in Theorem \ref{T.ritorno}}. Recall that it remains to prove that, given an interior  point $p\in \mathcal{T}$, there exist infinitely many boomerang shots from $p$ of type $(2)$ and $(3)$: namely, trajectories passing once again through $p$ with  same direction but opposite orientation (type $(2)$) or with  the other possible direction (type $(3)$). 

Using the notation of the proof of the finiteness part of Theorem \ref{T.ritorno}, we associate to the point $p$ four sections $\pm \sigma_p, \pm \sigma_p'$ of the elliptic scheme $\mathcal{X}_p \to B_p$, obtained after a base change (e.g. from the Legendre scheme) depending on $p$. 
The sections $\sigma_p,\sigma_p'$ correspond to different choice of tangents for each caustic; change of sign corresponds to inversion of the orientation of the path.

Denote again by $\kappa:B_p\to \mathcal{X}_p$ the billiard section. As explained in the proof of the first part of the theorem,  boomerang shot of type $(2)$ correspond to  points $s\in B_p(\R)$v such that for some integer $n>0$, 
\begin{equation*}
\sigma_p(s)+n\kappa(s)=-\sigma_p(s)
\end{equation*}
i.e. $2\sigma_p(s)=-n\kappa(s)$;
  those of type $(3)$ are given by the relation 
\begin{equation*}
\sigma_p(s)+n\kappa(s)=\sigma_p'(s)
\end{equation*}
i.e. $(\sigma_p-\sigma_p')(s)=-n\kappa(s)$. 

Since, as already proved, $\sigma_p,\sigma_p'$ are linearly independent, so  in particular $2\sigma_p$ and $\sigma_p-\sigma_p'$ are both non-torsion, we can apply Proposition \ref{P.reali1}, concluding that both equations above admit infinitely many solutions $s\in B_p(\R)$ 
\qed

\begin{rem} \label{R.densereal} {\tt Complex and $p$-adic points}.  (i) An analogue of Proposition \ref{P.reali1} would follow for the complex points, following the same method of proof. More on this in Theorem \ref{ThmComplex}.
	
	Note also that there are examples proving that the restriction that $N$ has to be large cannot be omitted. This restriction will not be necessary for the sections coming from the billiard map, as follows from Theorem \ref{T.esiste} that we shall soon prove. 
	
	See the paper of B. Lawrence \cite{Lau} and \S 9  of  \cite{ACZ}   for a study of the Betti map on  the real points of a certain higher-dimensional base, giving density results similar in spirit to the present ones. On the other hand, density fails in the $p$-adic context, see  \cite{LZ} for an instance. 
	
	
	(ii) In the complex case, Theorem \ref{ThmComplex} (proven in \cite{CDMZ}, we give a quick sketch of the proof below) gives such an estimate. In this case, moreover, it has to be noted that the limit $\lim_{n \to \infty} \frac{A_n}{n^2}$ appearing in the theorem has also another meaning. Namely, one can show that it is equal to the \emph{canonical height} $\hat{h}(\sigma)$ of the section $\sigma$. To prove this, one has to use the fact that the points in  the base $B$ such that the differential  of the Betti map vanishes are a finite amount. This is proven in \cite{CDMZ}.
	
	(iii) 	In the case that $\sigma_M$ is the special Masser section, i.e. the one defined on the Legendre scheme $\mathcal{A} \to B'$ by $\sigma_M(\lambda):= (2,\sqrt{2(2-\lambda)})$, then the canonical height $\widehat{h}(\sigma_M)$ is equal to $\frac{1}{2}$ (see e.g. \cite[Example 3.4]{CDMZ}). $B'$ here denotes the cover of $\P_{1,\lambda}$ defined by the quadratic field extension $\C(\lambda) \subset \C(\sqrt{2-\lambda})$.
	
	One may use this to calculate the height of the billiard section $B(\lambda)$, as described in Remark \ref{R.bmap}. Let us denote by $B_c(\lambda)$ the billiard section associate to the ellipse with parameter $c$ (following the present notation of \eqref{E.C}). We note that, for $c_0=1/\sqrt{2}$, we have that $B_{c_0}(\lambda)=\sigma_M(\lambda)+ T_2$, where $T_2 = (0,\lambda)$ is a torsion section of order $2$. In particular, by general facts on heights, we have that $\widehat{h}(B_{c_0})= \widehat{h}(\sigma_M)$. Moreover, since $\widehat{h}(\sigma)$ is rational for every section $\sigma$ (see e.g. \cite[Section 11.8]{ellipticsurfaces}), and $\widehat{h}(B_{c})$ varies with continuity for $c \in \C \setminus \{0,\pm 1,\infty\}$, we find that $\widehat{h}(B_{c})$ is always equal to $1/2$ for $c \neq 0,\pm 1$.
\end{rem}

We note that the methods of \cite{DMM}, which deal with the distribution of points at which a section of an elliptic scheme attains torsion value, do not give results such as Proposition \ref{P.reali1}, although they give similar results for the complex points (which are independent from those on the real points, see Theorem \ref{ThmComplex} below).

\begin{rem}\label{R.segno}{\tt Non monotonicity of the (real) Betti map}. 
	We see from the proof of Proposition \ref{P.reali1} that we have that $c = \int_I |\d \beta_{\sigma}|$. It would be interesting if one {\em could} remove the absolute value from the formula. In fact, one can show, with arguments that would go beyond the scope of the paper, that $\int_{B(\R)}  \d \beta_{\sigma}$ is related to some intersection numbers on (a complete model of) the smooth real surface $\mathcal{A}(\R)$, when this complete model happens to be orientable. 
	However, it is not true for a general $\sigma$ that $\beta_{\sigma}$ is monotone (although it is in some specific cases, for instance the case where $\sigma$ is the section associated to the billiard shot and $I$ is a connected component of $B(\R)$, as shown in Theorem \ref{T.esiste}). In Example \ref{E.segnovaria1} below we provide three (classes of) examples of sections in which the sign of $\d \beta_{\sigma}$ is not constant on a connected segment $I$ of $B(\R)$, so that, in those cases one has that:
	\[
	c = \int_{I} |\d \beta_{\sigma}| \neq \left|\int_{I} \d \beta_{\sigma}\right|.
	\]
\end{rem}

\begin{example}\label{E.segnovaria1}{\tt Non-monotonicity of the (real) Betti map: 

Counterexample via linear combination.}
	Let $\sigma_1$ and $\sigma_2$ be two linearly independent   algebraic sections of $\mathcal{A}$ over $B$, both defined over $\R$. For simplicity, we assume that $\mathcal{A} \to B$ is (a base change of) the Legendre scheme, so that, on some small neighbourhood $U$ of a point $p_0 \in B(\R)$, we have a choice of real and imaginary period as in Section \ref{Periods}. For $p \in U$, let us denote by 
	$b_1(p)=(\tilde{\beta_1}(p),\beta_1(p))$ and $b_2=(\tilde{\beta_2}(p),\beta_2(p))$ the Betti maps of $\sigma_1(p)$ and $\sigma_2(p)$, where $\beta_1(p)$ and $\beta_2(p)$ are the real Betti coordinates of $\sigma_1(p)$ and $\sigma_2(p)$.

	We have that the ratio\footnote{In this example the symbol $\d$ will always denote the differential on the real domain $U(\R)$, and not the differential on the complex domain $U(\C)$.} $\d \beta_1(p)/\d \beta_2(p), \ p \in U(\R)$ is nowhere locally constant. Indeed, if it were, we would have that $\beta_1 = c_0 \beta_2+ k_0, \ c_0,k_0 \in \R$ on $U(\R)$.

	Consider now the {\em analytic} section $\sigma := \sigma_1 - c_0\sigma_2-k_0$, defined on $U$. The Betti map of $\sigma(p)$ is $b(p):= b_1(p) - c_0 b_2(p) - k_0=(\tilde{\beta_1},\beta_1)- c_0 (\tilde{\beta_2},\beta_2)- (0,k_0) \in (\R/\Z)^2$. Note that, since $\tilde{\beta_1}(p), \tilde{\beta_2}(p)  \in \frac{1}{2}\Z$ for $p \in U(\R)$, the function $b(p)$ is constant on $U(\R)$. Since $U(\R)$ is a (real) variety of dimension $1$, by \cite[Proposition 1.1]{CMZ}, this would imply that, for $p \in U$, $b(p)=b(p_0) \in (\frac{1}{2}\Z)^2$. Hence $\sigma$ would be torsion of order $2$.
	
	In particular we would have that $2\sigma_1=2c_0\sigma_2+2k_0$ on $U$.
	Now some non-trivial monodromy arguments (see e.g. \cite{CZPonc}, Theorem 6.3.10) show that, in this case, $k_0, c_0 \in \Q$. Hence $\sigma_1$ and $\sigma_2$ would be linearly equivalent as algebraic sections.

	As a consequence, there exist points $p, q \in U(\R)$ such that $\d \beta_1/\d \beta_2(p) < \d \beta_1/\d \beta_2(q)$. We choose integers $N,M \neq 0$ such that $\d \beta_1/\d \beta_2(p) < N/M <\d \beta_1/\d \beta_2(q)$.
	
	If we define now $\tau := [M]\sigma_1-[N]\sigma_2$, 
	we see  that $\d \beta_\tau(p) = M \d \beta_{\sigma_1}(p)-N\d \beta_{\sigma_2}(p)<0$ and $\d \beta_\tau(q) = M \d \beta_{\sigma_1}(q)-N\d \beta_{\sigma_2}(q)>0$. Hence the sign of $\d \beta_{\tau}(p), p \in U(\R)$ is not constant on $U(\R)$ (as it attains different values on $p$ and $q$). In the example below we provide an explicit class of examples, of dynamic nature, that are instances of the phenomenon just described.


{\tt Counterexample on the billiard.}
	Let $C$ be an elliptical billiard.  We  choose a point $p_0 \in \mathcal{T}^o$ (we remind the reader that $\mathcal{T}^o$ denotes the interior of the billiard), not lying on the line connecting the two foci (i.e. the axis $y=0$). 
	
	We denote by $\L \to \P_1$ the Legendre elliptic scheme, we choose a point $\lambda_0 \in \P_1(\R)$ corresponding to an elliptical caustic (through the identification \ref{E.cr}), and we choose a neighborhood $U$ of $\lambda_0$, where we can make a choice of a real and an imaginary period as in Section \ref{Periods}.
	We denote by $\beta_2(\lambda):U \to \R/\Z$ a local branch of the real Betti map of the billiard section $B:U \to \L$ (note that the billiard section is not algebraic over $U$, but this does not represent an issue for the counterexample).
	
	We consider the base changed elliptic scheme $\L':= \L \times_{\P_1} C \to C$, where the map $\phi:C \to \P_1$ is the one that sends a point $c \in C$ to the caustic associated to the shot from $p_0$ directed towards $c$ (and the successive bounces). We denote by $I \subset \P_1(\R)$ the interval parametrizing elliptical caustics, and by $I'$ the inverse image $\phi^{-1}(I)\cap C(\R)$. I.e. $I'$ is the set of $c \in C(\R)$ such that the line $p_0c$ defines an elliptical caustic. Note that this set is the disjoint union of two intervals.
	
	The restriction $\phi|_{I'}:I'\to I$ is not monotone on each of the two connected components of $I'$. In fact, it has local extrema at the two points $c \in I' \subset C(\R)$ such that the line $p_0c$ is tangent to the ellipse confocal to $C$ passing through $p_0$.
	
	Hence, keeping in mind that the real Betti map $\beta_2(\lambda):I \to \R/\Z$ of the billiard section $B:\P_1 \to \L$ is monotone (as shown in the proof of Theorem \ref{T.esiste}), we see that the composition $\beta_2 \circ \phi:I' \to \R/\Z$, which is the Betti map associated to the billiard section on the base changed elliptic scheme $\L':= \L \times_{\P_1} C\to C$, is not monotone on each of the two connected components of $I'$.
	
	This provides the sought example of non-monotonicity of the Betti map of a billiard section. We leave to the interested reader the exercise of extending this counterexample by combining it with the previous one.

{\tt Counterexample via  analytic methods.}
	Choose a segment $I \subset B(\R)$. For the sake of exposition, let us assume for simplicity that $\mathcal{A} \to B$ is the Legendre scheme, and that $\lambda \neq 0,1,\infty$ on $I$. Let $\beta_2(\lambda): I \to \R$ be any real-analytic function whose derivative is of non-constant sign on $I$, and choose a complex (connected) neighbourhood $U$, containing $I$, such that $\beta_2(\lambda)$ extends to an analytic function $U \to \C$ (note that such an extension is always unique). We define an analytic section $\sigma$ of $\mathcal{A}$ on $U$ as the abelian exponential of $\beta_2(\lambda) \omega_2(\lambda)$ (where $\omega_2(\lambda)$ designates the real period). Note that, by construction, the Betti coordinates of $\sigma(\lambda), \lambda \in I$, are $(0,\beta_2(\lambda))$. 
	We can approximate the analytic section $\sigma$ with algebraic sections $\sigma_n$ of the Legendre scheme (we remind that an algebraic section is a rational section defined over a finite cover $B' \to B$), as one may easily prove using the Stone-Weierstrass theorem. Moreover, one may choose these sections to be real.
	
	It follows now that any section $\sigma_n$ sufficiently near to $\sigma$ will be such that $\d (\beta_2)_{\sigma_n}$ has non-constant sign on $I$, providing again a class of examples where the Betti map of the section is not monotone.
\end{example}

\medskip


\begin{thm}\label{ThmComplex}\cite[Theorem 3.2]{CDMZ}
	Let $\sigma$ be a non-torsion algebraic section of the {complex} space $\mathcal{A}(\C) \to B(\C)$, defined on a finite covering $B' \to B$. Then we have the following asymptotic:
	\begin{equation}
		\int_{B'(\C)} \d \beta_1 \wedge \d \beta_2 = \lim_{n \to \infty} \frac{A_n}{n^2},
	\end{equation}
	where $A_n := \{p \in B'(\C) \mid \sigma(p) \text{ is torsion of order dividing } n\}$.
\end{thm}
\begin{proof}
	We refer to \cite{CDMZ} for a complete proof, and just hint at the main idea here. It starts with the following two facts about definability. 
	
	\underline{First fact}. Let $C \subset \R^2$ be any bounded definable set (again by this we mean definable in $\R_{an,exp}$, we refer to \cite[p. 16]{tametopology} for the notion of definability, but the reader may just think of $C$ as a finite union of closed compact sets of the form $f_1(x) \leq y \leq f_2(x), x \in [a,b]$, where $f_1$ and $f_2$ are piece-wise analytic functions). Then we have by a theorem of Barroero and Widmer \cite[Theorem 1.3]{BW} that the number $$A_n(C):= \{p \in C \mid p \text{ has rational coordinates with denominator dividing } n\}$$ satisfies $\lim_{n \to \infty} A_n(C)/n^2 = \lambda(C)$, where $\lambda$ denotes the Lebesgue measure. 
	
	\underline{Second fact}.  We remind the reader that the Betti map $\beta_{\sigma}$ is definable.
	Using Hardt's theorem \cite[Theorem 9.1.2]{tametopology}, one can show that there exists a finite decomposition of $B'(\C)= \sqcup_i D_i$ in definable sets $D_i \subset B'(\C)$, such that, on each set $D_i$, the function $\beta := (\beta_1,\beta_2)$ is injective. One then shows that $A_n=\sum_i A_n(\beta(D_i))$.
	
	 The result can be deduced from the two facts above.
\end{proof}

\begin{rem}\label{R.maninb} {\tt Algorithms for checking if a section is torsion}. 
	In Proposition \ref{P.reali1} an essential assumption was that $\sigma $ was not torsion. There are effective algorithms which allow to check such facts, for general sections (provided everything is defined over a `computable' field). For instance  one may appeal  (i) to  results about Galois theory of torsion sections (as in work going back to Fricke and Weber): the Galois group becomes large for large torsion order so one can bound the possible order. (ii)  on good reduction: torsion sections are defined over fields unramified outside the bad reduction. For instance by \eqref{E.masser}  the minimal  field of definition  for the billiard section is ramified above $\lambda=1/c^2$ which is of good reduction for $c^2\neq 0,1$. (iii) Height theory: the height of torsion sections is bounded.  
	(iv) A further algorithm to check whether a section is torsion is due to Manin. It is very practical, though it works only over function fields, and moreover if the answer is `yes' it does not give the torsion order. This algorithm requires merely computing the Gauss-Legendre operator on the elliptic logarithm of the section. As proved by Manin in general, this yields  always an algebraic function, given  explicitly in \cite{Man}, however in a form which needs a small correction, carried out in (6.65) of \cite{CDMZ}.  This algebraic function is a differential expression  in terms of the coordinates of the section and is additive.  It vanishes if and only if the section is torsion, which provides the algorithm.

	In the case of the Billiard section   this function is   $2c(1-c^2)^{1/2}(1-c^2\lambda)^{-3/2}$. This gives another proof that this section is non torsion, but can be useful for other purposes. For instance it shows that the Betti coordinate $\beta_2(\lambda)$ is a product of functions satisfying  a differential equation of Fuchsian type. 
\end{rem}

\bigskip

\bigskip

\bigskip

\noindent Pietro Corvaja \\ Dipartimento di Scienze Matematiche, Informatiche e Fisiche \\ Universit\`a di Udine \\ Via delle Scienze, 206 \\ 33100 Udine - Italy
\\{\tt pietro.corvaja@uniud.it}

\bigskip

\noindent Umberto Zannier \\ Scuola Normale Superiore \\ Piazza dei Cavaleri, 7 \\ 56100 Pisa - Italy \\ {\tt u.zannier@sns.it}

\bigskip

\noindent Julian Lawrence Demeio \\ Scuola Normale Superiore \\ Piazza dei Cavaleri, 7 \\ 56100 Pisa - Italy \\ {\tt julian.demeio@sns.it}

\end{document}